\documentclass[letterpaper,11pt,oneside,reqno]{amsart}

\usepackage{bm}
\usepackage{mathrsfs}
\usepackage{amsfonts,amsmath, amssymb,amsthm,amscd,stmaryrd}
\usepackage[margin=1in]{geometry}
\usepackage{verbatim}
\usepackage{hyperref}
\usepackage{cleveref}
\usepackage{graphicx}
\usepackage[latin1]{inputenc}
\usepackage{latexsym}
\usepackage{lscape}
\usepackage{epsfig}
\usepackage{todonotes}
\usepackage{array}

\usepackage{caption}
\usepackage[foot]{amsaddr}
\usepackage{setspace}

\allowdisplaybreaks

\usepackage{dsfont}
\usepackage{bm}
\usepackage{color}

\usepackage{enumerate}
\usepackage{enumitem}
\usepackage{tikz}
\usetikzlibrary{positioning,shapes,arrows}
\usepackage{graphicx}
\usepackage{epstopdf}
\usepackage[parfill]{parskip}
\include{bibtex}
\usepackage{mathrsfs}
\usepackage{stmaryrd}
\usepackage[ruled,vlined]{algorithm2e}
    
    \SetCommentSty{mycommfont}

\newtheorem{theorem}{Theorem}[section]
\newtheorem{proposition}[theorem]{Proposition}
\newtheorem{lemma}[theorem]{Lemma}
\newtheorem{corollary}[theorem]{Corollary}
\newtheorem{definition}[theorem]{Definition}
\newtheorem{conjecture}[theorem]{Conjecture}
\newtheorem{example}[theorem]{Example}
\newtheorem{remark}[theorem]{Remark}

\newcommand{\R}{\mathbb{R}}

\newcommand{\N}{\mathbb{N}}

\newcommand{\Z}{\mathbb{Z}}
\renewcommand{\P}{\mathbb{P}}
\newcommand{\E}{\mathbb{E}}
\newcommand{\1}{\mathbf{1}}

\newcommand{\var}{\mathrm{Var}}

\newcommand{\A}{\mathcal{A}}

\newcommand{\eps}{\varepsilon}

\renewcommand{\H}{\mathbb{H}}
\newcommand{\rad}{\mathrm{rad}}

\newcommand{\bdy}{\partial}

\newcommand{\dist}{\mathrm{dist}}

\newcommand{\PP}{\mathbf{P}}
\newcommand{\EE}{\mathbf{E}}

\newcommand{\B}{\mathcal{B}}
\newcommand{\diam}{\mathrm{diam}}
\renewcommand{\dist}{\mathrm{dist}}
\newcommand{\bdyo}{\partial^{\kern 0.05em \mathrm{int}}}
\newcommand{\bdyexp}{\partial_{\kern 0.05em \mathrm{exp}}}

\newcommand{\sep}{\mathrm{sep}}

\newcommand{\Mid}{\mathsf{Mid}}

\newcommand{\iso}{\mathrm{Iso}}
\newcommand{\noniso}{\mathrm{NonIso}}

\renewcommand{\aa}{\mathfrak{a}}
\newcommand{\Act}{\mathsf{Act}}
\newcommand{\Dep}{\mathsf{Dep}}

\newcommand{\Rec}{\mathrm{Rec}}
\newcommand{\TT}{\mathcal{T}}
\newcommand{\UU}{\mathcal{F}}

\newcommand{\Quick}{\mathsf{Timely}}

\renewcommand{\ss}{\mathfrak{s}}
\renewcommand{\tt}{\mathfrak{t}}
\newcommand{\mm}{N}

\newcommand{\wt}{\widehat}
\newcommand{\MM}{\mathscr{M}}
\newcommand{\CC}{\mathscr{C}}
\newcommand{\Close}{\mathsf{Close}}
\newcommand{\Many}{\mathsf{Many}}
\newcommand{\ee}{e}

\newcommand{\bdye}{\partial_{\mathrm{ext}}^\ast}
\newcommand{\Arc}{\mathrm{Arc}}
\newcommand{\Sec}{\mathrm{Sec}}
\newcommand{\Ann}{\mathrm{Ann}}
\newcommand{\Circ}{\mathrm{Circ}}
\newcommand{\II}{\mathds{I}}

\newcommand{\wtnoniso}{\wt{\mathrm{N}} \mathrm{onIso}}

\newcommand{\ii}{I}
\newcommand{\jj}{J}

\title[Harmonic activation and transport]{Collapse and diffusion in\\ harmonic activation and transport}

\author{Jacob Calvert, Shirshendu Ganguly, \and Alan Hammond}
\address{Departments of Mathematics and Statistics\\
 U.C. Berkeley \\
  Evans Hall \\
  Berkeley, CA, 94720-3840 \\
  U.S.A.}
  \email{jacob\_calvert@berkeley.edu, sganguly@berkeley.edu, alanmh@berkeley.edu}

\thanks{J.C. was partially supported by NSF grant DMS-1512908. S.G. was partially supported by NSF grant DMS-1855688, NSF CAREER Award DMS-1945172, and a Sloan Fellowship. A.H. was partially supported by NSF grants DMS-1512908 and DMS-1855550, and a Miller Professorship from the Miller Institute for Basic Research in Science.}

\begin{document}

\begin{abstract} For an $n$-element subset $U$ of $\Z^2$, select $x$ from $U$ according to harmonic measure from infinity, remove $x$ from $U$, and start a random walk from $x$. If the walk leaves from $y$ when it first enters $U$, add $y$ to $U$. Iterating this procedure constitutes the process we call \textit{Harmonic Activation and Transport} (HAT).

HAT exhibits a phenomenon we refer to as \textit{collapse}: informally, the diameter shrinks to its logarithm over a number of steps which is comparable to this logarithm. Collapse implies the existence of the stationary distribution of HAT, where configurations are viewed up to translation, and the exponential tightness of diameter at stationarity. Additionally, collapse produces a renewal structure with which we establish that the center of mass process, properly rescaled, converges in distribution to two-dimensional Brownian motion.

To characterize the phenomenon of collapse, we address fundamental questions about the extremal behavior of harmonic measure and escape probabilities. Among $n$-element subsets of $\Z^2$, what is the least positive value of harmonic measure? What is the probability of escape from the set to a distance of, say, $d$? Concerning the former, examples abound for which the harmonic measure is exponentially small in $n$. We prove that it can be no smaller than exponential in $n \log n$. Regarding the latter, the escape probability is at most the reciprocal of $\log d$, up to a constant factor. We prove it is always at least this much, up to an $n$-dependent factor.
\end{abstract}

\maketitle

\setcounter{tocdepth}{1}
\tableofcontents


\section{Introduction}\label{sec: intro}

\subsection{Harmonic activation and transport}

Consider simple random walk $(S_j)_{j \in \N}$ on $\Z^2$ and with $S_0 = x$, the distribution of which we denote by $\P_x$. For a finite, nonempty subset $A \subset \Z^2$, the hitting distribution of $A$ from $x \in \Z^2$ is the function $\H_A (x, \cdot) : \Z^2 \to [0,1]$ defined as $\H_A (x,y) = \P_x (S_{\tau_A} =y)$, where $\tau_A = \inf\{j \geq 1 : S_j \in A\}$. The recurrence of random walk on $\Z^2$ guarantees that $\tau_A$ is almost surely finite, and the existence of the limit $\H_A (y) = \lim_{|x| \to \infty} \H_A (x,y)$, called the harmonic measure of $A$, is well known \cite{lawler2013intersections}.

In this paper, we introduce a Markov chain called \textit{Harmonic Activation and Transport} (HAT), wherein the elements of a subset of $\Z^2$ (respectively styled as ``particles'' of a ``configuration'') are iteratively selected according to harmonic measure and replaced according to the hitting distribution of a random walk started from the location of the selected element. We say that, at each step, a particle is ``activated'' and then ``transported.''

\begin{definition}[Harmonic activation and transport]
Given a finite subset $U_0$ of $\Z^2$ with at least two elements, HAT is the Markov chain $(U_t)_{t\in \N}$ on subsets of $\Z^2$, the dynamics of which consists of the following steps (Figure~\ref{fig: dynamics}).
\begin{enumerate}
\item[] \textbf{Activation}. At time $t \in \N$, sample $X$ from $U_t$ according to $X \sim \H_{U_t}$. 
\item[] \textbf{Transport}. Given $X$, set $S_0 = X$ and denote $\tau = \tau_{U_t{\setminus}\{X\}}$. Form $U_{t+1}$ as \begin{equation}\label{eq: eta add} U_{t+1} = U_t \cup \big\{S_{\tau - 1}\big\} {\setminus} \{X\}\end{equation} and repeat the activation and transport steps with $U_{t+1}$ in the place of $U_t$. 
\end{enumerate}
The sequence $(U_t)_{t \in \N}$ is a Markov chain with inhomogeneous transition probabilities given by 
\begin{equation*}
\PP \left(U_{t+1} = U_t \cup \{y\} {\setminus} \{x\} \bigm\vert U_t \right) = \H_{U_t} (x) \, \P_x \left( S_{\tau - 1} = y\right).
\end{equation*} In particular, the transition probabilities are only nonzero if $x \in \bdyo U_t$, where $\bdyo U_t = \{x \in U_t: |x-y| =1\,\,\text{for some $y\notin U_t$}\}$ is the ``interior boundary,'' and if $y \in \bdy (U_t {\setminus} \{x\})$.
\end{definition}

\begin{figure}
\centering {\includegraphics[width=0.7\linewidth]{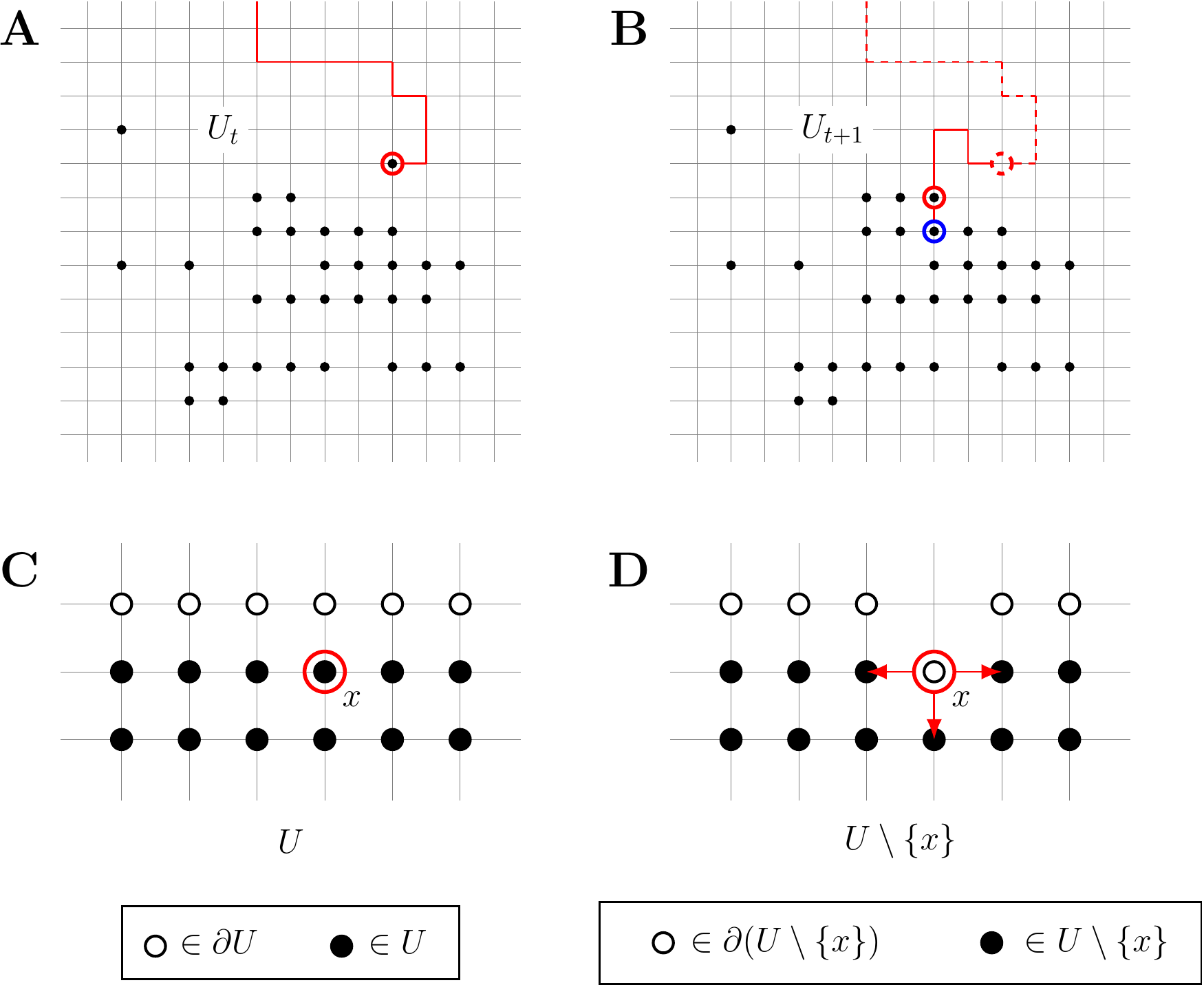}}
\caption{The harmonic activation and transport dynamics. (\textbf{A}) A particle (indicated by a solid, red circle) in the configuration $U_t$ is activated according to harmonic measure. (\textbf{B}) The activated particle (following the solid, red path) hits another particle (indicated by a solid, blue circle); it is then fixed at the site visited during the previous step (indicated by a solid, red circle), giving $U_{t+1}$. (\textbf{C}) A particle of $U$ (indicated by a red circle) is activated and (\textbf{D}) if it tries to move into $U {\setminus} \{x\}$, the particle will be placed at $x$.} 
\label{fig: dynamics}
\end{figure}

\begin{remark} The reader may wonder why we use the random time $\tau - 1$ in \eqref{eq: eta add} as opposed to, say, the first hitting time $\sigma_{\bdy (U_t {\setminus} \{x\})}$ of the exterior boundary of $U_t {\setminus} \{x\}$. For the scenario depicted in Figure~\ref{fig: dynamics}C--D, wherein $x$ neighbors elements of $U_t {\setminus} \{x\}$, we would have $\sigma_{\bdy (U_t {\setminus} \{x\})} = 0$ and therefore $U_{t+1} = U_t$. This possibility would complicate arguments in Section~\ref{sec: stat dist} and is therefore undesirable.
\end{remark}

To guide the presentation of our results, we highlight four features of HAT.
\begin{itemize}
\item \textit{Conservation of mass}. HAT conserves the number of particles in the initial configuration.
\item \textit{Invariance under symmetries of $\Z^2$}. Denoting by $\mathcal{G}$ the symmetry group of $\Z^2$, one can see that $\PP (U_{t+1} \bigm\vert U_t ) = \PP (g \cdot U_{t+1} \bigm\vert g \cdot U_t)$ for any element $g$ of $\mathcal{G}$. In words, the HAT dynamics is invariant under the symmetries of $\Z^2$. Accordingly, to each configuration $U$, we can associate an equivalence class \[\wt U = \left\{ V \subseteq \Z^2: \mathcal{G} \cdot V = \mathcal{G} \cdot U\right\}.\]
\item \textit{Variable connectivity}. The HAT dynamics does not preserve connectivity. Indeed, a configuration which is initially connected will eventually be disconnected by the HAT dynamics, and the resulting components may ``treadmill'' away from one another, adopting configurations of arbitrarily large diameter.
\item \textit{Asymmetric behavior of diameter}. While the diameter of a configuration can increase by at most one with each step, it can decrease abruptly. For example, if the configuration is a pair of particles separated by $d$, then the diameter will decrease by $d-1$ in one step.
\end{itemize}

We will shortly state the existence of the stationary distribution of HAT. By the invariance of the HAT dynamics under the symmetries of $\Z^2$, the stationary distribution will be supported on equivalence classes of configurations which, for brevity, we will simply refer to as configurations. In fact, the HAT dynamics cannot reach all such configurations. By an inductive argument, we will prove that the HAT dynamics is irreducible on the collection of configurations whose boundary elements belong to connected components which are not exclusively singletons.

\begin{definition}
Denote by ${\rm{Iso}}(n)$ the collection of $n$-element subsets $U$ of $\Z^2$ such that every $x$ in $U$ with $\H_U (x) > 0$ belongs to a singleton connected component. In other words, all exposed elements of $U$ are isolated: they lack nearest neighbors in $U$. We will denote the collection of all other $n$-element subsets of $\Z^2$ by $\noniso (n)$, and the corresponding equivalence class by
\[ \wtnoniso (n) = \big\{ \wt U: U \in \noniso (n) \big\}.\]
\end{definition}

The variable connectivity of HAT configurations and concomitant opportunity for unchecked diameter growth seem to jeopardize the positive recurrence of the HAT dynamics on $\wtnoniso (n)$. Indeed, if the diameter were to grow unabatedly, the HAT dynamics could not return to a configuration or equivalence class thereof, and would therefore be doomed to transience. However, due to the asymmetric behavior of diameter under the HAT dynamics, this will not be the case. For an arbitrary initial configuration of $n \geq 2$ particles, we will prove---up to a factor depending on $n$---sharp bounds on the ``collapse'' time which, informally, is the first time the diameter is at most a certain function of $n$.

\begin{definition} For a positive real number $R$, we define the level-$R$ collapse time to be $\TT (R) = \inf\{t \geq 1: \diam (U_t) \leq R\}$.
\end{definition}

For a real number $r \geq 0$, we define $\theta_m = \theta_m (r)$ through \begin{equation}\label{eq: radii} \theta_0 = r \quad\text{and} \quad \theta_m = \theta_{m-1} + e^{\theta_{m - 1}} \,\,\, \text{for $m \geq 1$}.\end{equation} In particular, $\theta_n (r)$ is approximately the $n$\textsuperscript{th} iterated exponential of $r$.

\begin{theorem}\label{thm: cons} Let $U$ be a finite subset of $\Z^2$ with $n \geq 2$ elements and denote the diameter of $U$ by $d$. There exists a universal positive constant $c$ such that, if $d$ exceeds $\theta_{4n} (cn)$, then \[ \PP_{U} \left( \TT (\theta_{4n} (c n)) \leq (\log d)^{1 + o_n(1)} \right) \geq 1 - e^{-n}.\]
For the sake of concreteness, this is true with $n^{-2}$ in the place of $o_n (1)$.
\end{theorem} In words, for a given $n$, it typically takes $(\log d)^{1+o_n (1)}$ steps before the configuration of initial diameter $d$ reaches a configuration with a diameter of no more than a large function of $n$.

As a consequence of Theorem~\ref{thm: cons} and the preceding discussion, it will follow that the HAT dynamics constitutes an aperiodic, irreducible, and positive recurrent Markov chain on $\wtnoniso (n)$. In particular, this means that, from any configuration of $\wtnoniso (n)$, the time it takes for the HAT dynamics to return to that configuration is finite in expectation. Aperiodicity, irreducibility, and positive recurrence imply the existence and uniqueness of the stationary distribution $\pi_n$, to which HAT converges from any $n$-element configuration. Moreover---again, due to Theorem~\ref{thm: cons}---the stationary distribution is exponentially tight.

\begin{theorem}\label{thm: stat dist} For every $n \geq 2$, from any $n$-element subset of $\Z^2$, HAT converges to a unique probability measure $\pi_n$ supported on $\wtnoniso (n)$. Moreover, $\pi_n$ satisfies the following tightness estimate. There exists a universal positive constant $c$ such that, for any $r \geq 2 \theta_{4n} (c n)$,
\[\pi_{n}\big({\rm{diam}}(\wt U)\ge r\big)\le \exp \left( - \frac{r}{(\log r)^{1+o_n(1)}} \right).\]
\end{theorem}

As a further consequence of Theorem~\ref{thm: cons}, we will find that the HAT dynamics exhibits a renewal structure which underlies the diffusive behavior of the corresponding center of mass process.

\begin{definition} For a sequence of configurations $(U_t)_{t \in \N}$, define the corresponding center of mass process $(\MM_t)_{t \geq 0}$ by $\MM_t = |U_t|^{-1} \sum_{x \in U_t} x$.
\end{definition}

For the following statement, denote by $\CC ([0,1])$ the continuous functions $f: [0,1] \to \R^2$ with $f(0) = (0,0)$, equipped with the topology induced by the supremum norm $\| f \| = \sup_{0 \leq t \leq 1} | f(t) |$.

\begin{theorem}\label{thm: cm} If $\MM_t$ is linearly interpolated, then the law of the process 
$\left(t^{-1/2} \MM_{st}, \, s \in [0,1]\right)$, viewed as a measure on $\CC ( [0,1] )$, converges weakly as $t \to \infty$ to two-dimensional Brownian motion on $[0,1]$ with coordinate diffusivity $\chi^2 = \chi^2 (n)$. Moreover, for a universal positive constant $c$, $\chi^2$ satisfies: 
\[ \theta_{5n} (cn)^{-1} \le \chi^2 \leq \theta_{5n} (cn).\]
\end{theorem}
We have not tried to optimize the bounds on $\chi^2$; indeed, they primarily serve to show that $\chi^2$ is positive and finite.

\subsection{Extremal behavior of harmonic measure}

As we elaborate in Section~\ref{iop}, the timescale of diameter collapse in Theorem~\ref{thm: cons} arises from novel estimates of harmonic measure and hitting probabilities, which control the activation and transport dynamics of HAT. Beyond their relevance to HAT, these results further the characterization of the extremal behavior of harmonic measure.

Estimates of harmonic measure often apply only to connected sets or depend on the diameter of the set. The discrete analogues of Beurling's projection theorem \cite{kesten1987hitting} and Makarov's theorem \cite{lawler1993discrete} are notable examples. 
Furthermore, estimates of hitting probabilities often approximate sets by disks which contain them (for example, the estimates in Chapter~2 of \cite{lawler2013intersections}). Such approximations work well for connected sets, but not for sets which are ``sparse'' in the sense that they have large diameters relative to their cardinality; we provide examples to support this claim in Section~\ref{subsec: novel}. For the purpose of controlling the HAT dynamics, which adopts such sparse configurations, existing estimates of harmonic and hitting measures are either inapplicable or suboptimal.

To highlight the difference in the behavior of harmonic measure for general (i.e., potentially sparse) and connected sets, consider a finite subset $A$ of $\Z^2$ with $n \geq 2$ elements. We ask: {\em What is the greatest value of $\H_A (x)$?} If we assume no more about $A$, then we can say no more than $\H_A (x) \leq \frac12$ (see Section 2.5 of \cite{lawler2013intersections} for an example). 
However, if $A$ is connected, then the discrete analogue of Beurling's projection theorem \cite{kesten1987hitting} provides a finite constant $c$ such that
\begin{equation*}
\H_A (x) \leq c n^{-1/2}.
\end{equation*} This upper bound is realized (up to a constant factor) when $A$ is a line segment and $x$ is one of its endpoints.

Our next result provides lower bounds of harmonic measure to complement the preceding upper bounds, addressing the question: {\em What is the least positive value of $\H_A (x)$?}

\begin{theorem}\label{thm: hm} There exists a universal positive constant $c$ such that, if $A$ is a subset of $\Z^2$ with $n \geq 1$ elements, then either $\H_A (x) = 0$ or
\begin{equation}\label{eq: ext hm thm}
\H_A (x) \geq e^{- c n \log n}.
\end{equation} If $A$ is connected, then \eqref{eq: ext hm thm} can be replaced by
 \begin{equation}\label{eq: ext hm thm2}
\H_A (x) \geq e^{- c n}.
\end{equation}
\end{theorem}

The lower bound of \eqref{eq: ext hm thm2} is optimal in terms of its dependence on $n$, as we can choose $A$ to be a narrow, rectangular ``tunnel'' with a depth of order $n$, in which case the harmonic measure at the ``bottom'' of the tunnel is exponentially small in $n$; we will shortly discuss a related example in greater detail. We expect that the bound in \eqref{eq: ext hm thm} can be improved to an exponential decay with a rate of order $n$ instead of $n \log n$.

If one could improve \eqref{eq: ext hm thm} as we anticipate, we believe that the resulting lower bound would be realized by the harmonic measure of the innermost element of a square spiral (Figure~\ref{fig: diamond}). The virtue of the square spiral is that, essentially, with each additional element, the shortest path to the innermost element lengthens by two steps. This heuristic suggests that the least positive value of harmonic measure should decay no faster than $4^{-2n}$, as $n \to \infty$. Indeed, Example~\ref{ex: diamond} suggests an asymptotic decay rate of $(2+\sqrt{3})^{-2n}$. We formalize this observation as a conjecture. To state it, denote the origin by $o = (0,0)$ and let $\mathscr{H}_n$ be the collection of $n$-element subsets $A$ of $\Z^2$ such that $\H_A (o) > 0$.

\begin{conjecture}\label{conj} Asymptotically, the square spiral of Figure~\ref{fig: diamond} realizes the least positive value of harmonic measure, in the sense that
\begin{equation*}
\lim_{n\to\infty} - \frac1n \log \inf_{A \in \mathscr{H}_n} \H_A (o) = 2 \log (2+\sqrt{3}).
\end{equation*}
\end{conjecture}

\begin{example}\label{ex: diamond}
Figure~\ref{fig: diamond} depicts the construction of an increasing sequence of sets $(A_1, A_2, \dots)$ such that, for all $n \geq 1$, $A_n$ is an element of $\mathscr{H}_n$, and the shortest path $\Gamma = (\Gamma_1, \Gamma_2, \dots, \Gamma_{|\Gamma|})$ from the exterior boundary of $A_n \cup \bdy A_n$ to the origin, which satisfies $\Gamma_i \notin A_n$ for $1 \leq i \leq |\Gamma| - 1$, has a length of $2(1-o_n (1))n$.

Since $\Gamma_1$ separates the origin from infinity in $A_n^c$, we have
\begin{equation}\label{eq: twofact}
\H_{A_n} (o) = \H_{A_n \cup \{\Gamma_1\}} (\Gamma_1) \cdot \P_{\, \Gamma_1} \left( S_{\tau_{A_n}} = o \right).
\end{equation}

Concerning the first factor of \eqref{eq: twofact}, one can show that there exist positive constants $b, c < \infty$ such that, for all sufficiently large $n$,
\begin{equation*}
c n^{-b} \leq \H_{A_n \cup \{\Gamma_1\}} (\Gamma_1) \leq 1.
\end{equation*}

To address the second factor of \eqref{eq: twofact}, we observe that
\begin{equation}\label{eq: firststep}
\P_{\,\Gamma_1} \left( S_{\tau_{A_n}} = o \right) = \P_{\, \Gamma_1} \left( S_1 = \Gamma_2 \bigm\vert \tau_{A_n} < \tau_{\, \Gamma_1} \right) \cdot \P_{\,\Gamma_2} \left( S_{\sigma_{A_n}} = o \bigm\vert \sigma_{A_n} < \sigma_{\,\Gamma_1} \right).
\end{equation}
It is easy to see that the first factor of \eqref{eq: firststep} satisfies
\begin{equation*}
\frac12 \leq \P_{\, \Gamma_1} \left( S_1 = \Gamma_2 \bigm\vert \tau_{A_n} < \tau_{\, \Gamma_1} \right) \leq 1.
\end{equation*}
The second factor of \eqref{eq: firststep} can be explicitly calculated using a system of difference equations. To this end, we define
\[ f(i) = \P_{\, \Gamma_i} \left( S_{\sigma_{A_n}} = o \bigm\vert \sigma_{A_n} < \sigma_{\, \Gamma_1} \right) \quad \forall\, 1 \leq i \leq |\Gamma | ,\]
which satisfies:
\begin{equation*}
f(1) = 0,\quad f(|\Gamma |) = 1, \quad \text{and} \quad f(i) = \frac14 f(i+1) + \frac14 f(i-1) \quad \forall\,2 \leq i \leq |\Gamma| - 1.
\end{equation*}
The solution of this system yields
\begin{equation}\label{eq: firststep2}
\P_{\,\Gamma_2} \left( S_{\sigma_{A_n}} = o \bigm\vert \sigma_{A_n} < \sigma_{\,\Gamma_1} \right) = \frac{2\sqrt{3}}{(2+\sqrt{3})^{|\Gamma|-1} - (2-\sqrt{3})^{|\Gamma|-1}}.
\end{equation} 

Combining \eqref{eq: twofact} through \eqref{eq: firststep2}, we find that, for all sufficiently large $n$,
\begin{equation}\label{eq: spiralbd1}
\cfrac{\frac12 c n^{-b}}{(2+\sqrt{3})^{|\Gamma|-1}} \leq \H_{A_n} (o) \leq \frac{1}{(2+\sqrt{3})^{|\Gamma|-2}}.
\end{equation} Substituting $|\Gamma | = 2(1-o_n (1)) n$ into \eqref{eq: spiralbd1} and simplifying, we obtain
\begin{equation*}
(2+\sqrt{3})^{-2(1+o_n(1))n} \leq \H_{A_n} (o) \leq (2+\sqrt{3})^{-2(1-o_n(1))n},
\end{equation*} which implies
\begin{equation*}
\lim_{n\to\infty} -\frac1n \log \H_{A_n} (o) = 2 \log (2+\sqrt{3}).
\end{equation*}
\end{example}

\begin{figure}[htbp]
\centering {\includegraphics[width=0.5\linewidth]{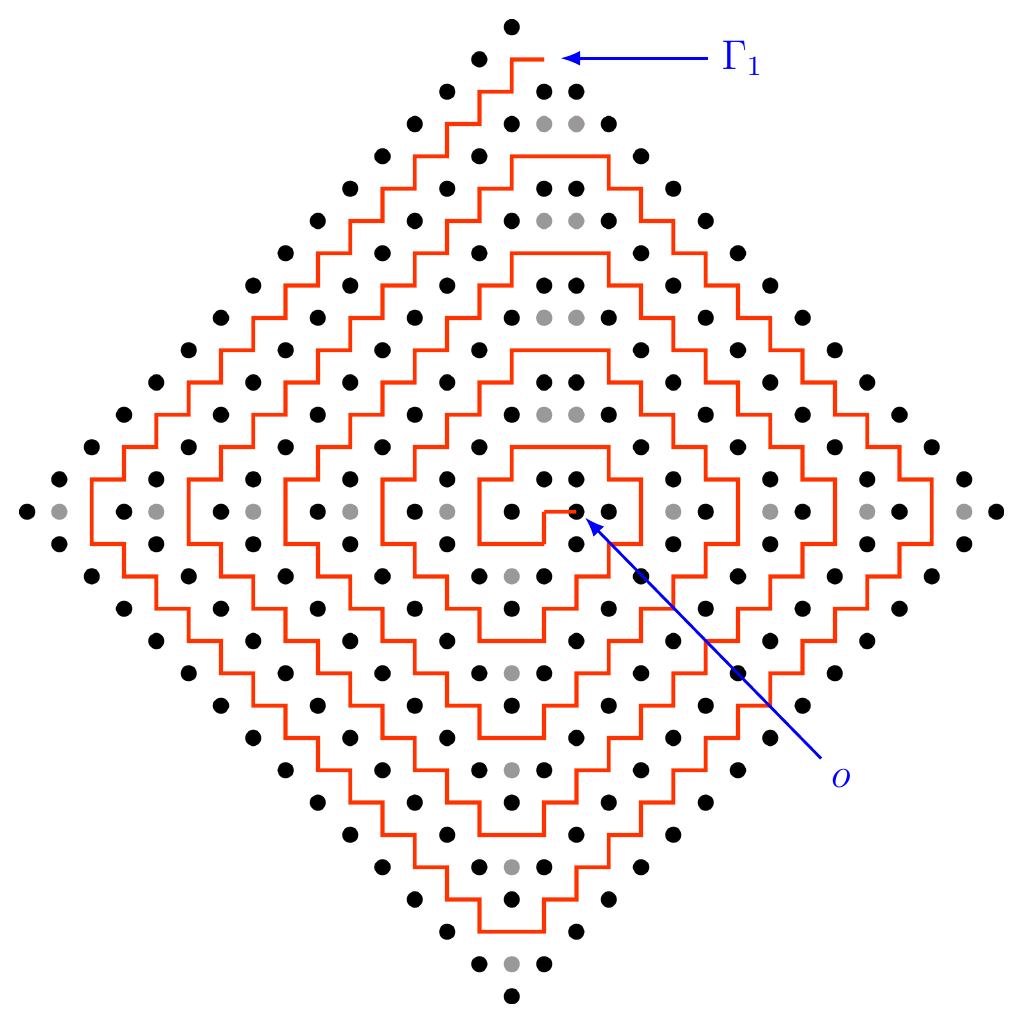}}
\caption{A square spiral. The shortest path $\Gamma$ (red) from $\Gamma_1$ to the origin, which first hits $A_n$ (black and gray dots) at the origin, has a length of approximately $2n$. Some elements (gray dots) of $A_n$ could be used to continue the spiral pattern (indicated by the black dots), but are presently placed to facilitate a calculation in Example~\ref{ex: diamond}.}
\label{fig: diamond}
\end{figure} 

We conclude the discussion of our main results by stating an estimate of hitting probabilities of the form $\P_x \left( \tau_{\bdy A_d} < \tau_A \right)$, for $x \in A$ and where $A_d$ is the set of all elements of $\Z^2$ within distance $d$ of $A$; we will call these {\em escape probabilities} from $A$. Among $n$-element subsets $A$ of $\Z^2$, when $d$ is sufficiently large relative to the diameter of $A$, the greatest escape probability to a distance $d$ from $A$ is at most the reciprocal of $\log d$, up to a constant factor. We find that, in general, it is at least this much, up to an $n$-dependent factor.

\begin{theorem} \label{thm: esc} There exists a universal positive constant $c$ such that, if $A$ is a finite subset of $\Z^2$ with $n \geq 2$ elements and if $d \geq 2\, \diam (A)$, then, for any $x \in A$,
\begin{equation}\label{eq: gen esc thm}
\P_x (\tau_{\bdy A_d} < \tau_A) \geq \frac{c \H_A (x)}{n \log d}.
\end{equation} In particular,
\begin{equation}\label{eq: esc thm}
\max_{x \in A} \P_x \left( \tau_{\bdy A_d} < \tau_A \right) \geq \frac{c}{n^2 \log d}.
\end{equation}
\end{theorem}

In the context of the HAT dynamics, we will use \eqref{eq: esc thm} to control the transport step, ultimately producing the $\log d$ timescale appearing in Theorem~\ref{thm: cons}. In the setting of its application, $A$ and $d$ will respectively represent a subset of a HAT configuration and the separation of $A$ from the rest of the configuration. Reflecting the potential sparsity of HAT configurations, $d$ may be arbitrarily large relative to $n$.

\subsection*{Organization} HAT motivates the development of new estimates of harmonic measure and escape probabilities. We attend to these estimates in Section~\ref{sec: hm est}, after we provide a conceptual overview of the proofs of Theorems \ref{thm: cons} and \ref{thm: stat dist} in Section~\ref{iop}. To analyze configurations of large diameter, we will decompose them into well separated ``clusters,'' using a construction introduced in Section~\ref{sec: clust} and used throughout Section \ref{sec: cons}. The estimates of Section~\ref{sec: hm est} control the activation and transport steps of the dynamics and serve as the critical inputs to Section~\ref{sec: cons}, in which we analyze the ``collapse'' of HAT configurations. We then identify the class of configurations to which the HAT dynamics can return and prove the existence of a stationary distribution supported on this class; this is the primary focus of Section~\ref{sec: stat dist}. The final section, Section~\ref{sec: cm}, uses an exponential tail bound on the diameter of configurations under the stationary distribution---a result we obtain at the end of Section~\ref{sec: stat dist}---to show that the center of mass process, properly rescaled, converges in distribution to two-dimensional Brownian motion.

\subsection*{Acknowledgements}
J.C. thanks Joseph Slote for useful discussions concerning Conjecture~\ref{conj} and Example~\ref{ex: diamond}. A.H. thanks Dmitry Belyaev for helpful discussions concerning the behavior of HAT configurations with well separated clusters and for simulating HAT dynamics.


\section{Conceptual overview}\label{iop}

\subsection{Estimating the collapse time and proving the existence of the stationary distribution}

Before providing precise details, we discuss some of the key steps in the proofs of Theorems~\ref{thm: cons} and \ref{thm: stat dist}.
Since the initial configuration $U$ of $n$ particles is arbitrary, it will be advantageous to decompose any such configuration into clusters such that the separation between any two clusters is at least exponentially large relative to their diameters. 
For the purpose of illustration, let us start by assuming that $U$ consists of just two clusters with separation $d$ and hence the individual diameters of the clusters are no greater than $\log d$ (Figure~\ref{fig: two_clust_fig}).

\begin{figure}
\centering {\includegraphics[width=0.4\linewidth]{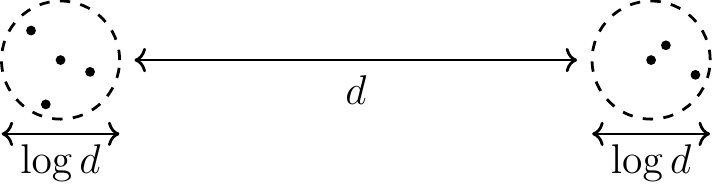}}
\caption{Exponentially separated clusters.}
\label{fig: two_clust_fig}
\end{figure}

The first step in our analysis is to show that in time comparable to $\log d,$ the diameter of $U$ will shrink to $\log d$. This is the phenomenon we call \emph{collapse}. 
Theorem~\ref{thm: hm} implies that every particle with positive harmonic measure has harmonic measure of at least $e^{-c n \log n}$. In particular, the particle in each cluster with the greatest escape probability  from that cluster has at least this harmonic measure. Our choice of clustering will ensure that each cluster is separated by a distance which is at least twice its diameter and has positive harmonic measure. Accordingly, we will treat each cluster as the entire configuration and Theorem~\ref{thm: esc} will imply that the greatest escape probability from each cluster will be at least $(\log d)^{-1}$, up to a factor depending upon $n$.

Together, these results will imply that, in $O_n (\log d)$ steps, with a probability depending only upon $n$, all the particles from one of the clusters in Figure~\ref{fig: two_clust_fig} will move to the other cluster. Moreover, since the diameter of a cluster grows at most linearly in time, the final configuration will have diameter which is no greater than the diameter of the surviving cluster plus $O_n (\log d)$. Essentially, we will iterate this estimate---by clustering anew the surviving cluster of Figure~\ref{fig: two_clust_fig}---each time obtaining a cluster with a diameter which is the logarithm of the original diameter, until $d$ becomes smaller than a deterministic function $\theta_{4n}$, which is approximately the $4n$\textsuperscript{th} iterated exponential of $cn$, for a constant $c$. 

Let us denote the corresponding stopping time by $\TT (\text{below $\theta_{4n}$}).$  In the setting of the application, there may be multiple clusters and we collapse them one by one, reasoning as above. If any such collapse step fails, we abandon the experiment and repeat it. Of course, with each failure, the set we attempt to collapse may have a diameter which is additively larger by $O_n (\log d)$. Ultimately, our estimates allow us to conclude that the attempt to collapse is successful within the first $(\log d)^{1+o_n (1)}$ tries with a high probability.

The preceding discussion roughly implies the following result, uniformly in the initial configuration~$U$:
\begin{equation*}
\PP_U \left( \TT (\text{below $\theta_{4n}$}) \le (\log d)^{{1+o_n (1)}} \right)\ge 1 - e^{- n}.
\end{equation*}

At this stage, we prove that, given any configuration $\wt U$ and any configuration $\wt V \in \wtnoniso (n)$, if $K$ is sufficiently large in terms of $n$ and the diameters of $\wt U$ and $\wt V$, then
$$\PP_{\wt U} \left( \TT (\text{hits $\wt V$}) \leq K^5 \right) \geq 1 - e^{-K},$$ where $\TT (\text{hits $\wt V$})$ is the first time the configuration is $\wt V$. This estimate is obtained by observing that the particles of $\wt U$ form a line segment of length $n$ in $K^3$ steps with high probability, and then showing by induction on $n$ that any other non-isolated configuration $\wt V$ is reachable from the line segment in $K^5$ steps, with high probability.  In addition to implying irreducibility of the HAT dynamics on $\wtnoniso (n)$, we use this result to obtain a finite upper bound on the expected return time to any non-isolated configuration (i.e., it proves the positive recurrence of HAT on $\wtnoniso (n)$). Irreducibility and positive recurrence on $\wtnoniso (n)$ imply the existence and uniqueness of the stationary distribution.

\subsection{Improved estimates of hitting probabilities for sparse sets}\label{subsec: novel}

HAT configurations may include subsets with large diameters relative to the number of elements they contain, and in this sense they are sparse. Two such cases are depicted in Figure~\ref{fig: diffs_fig_0}. A key component of the proofs of Theorems~\ref{thm: hm} and \ref{thm: esc} is a method which improves two standard estimates of hitting probabilities when applied to sparse sets, as summarized by Table~\ref{table: sum}.

\begin{figure}[ht]
\centering {\includegraphics[width=0.85\linewidth]{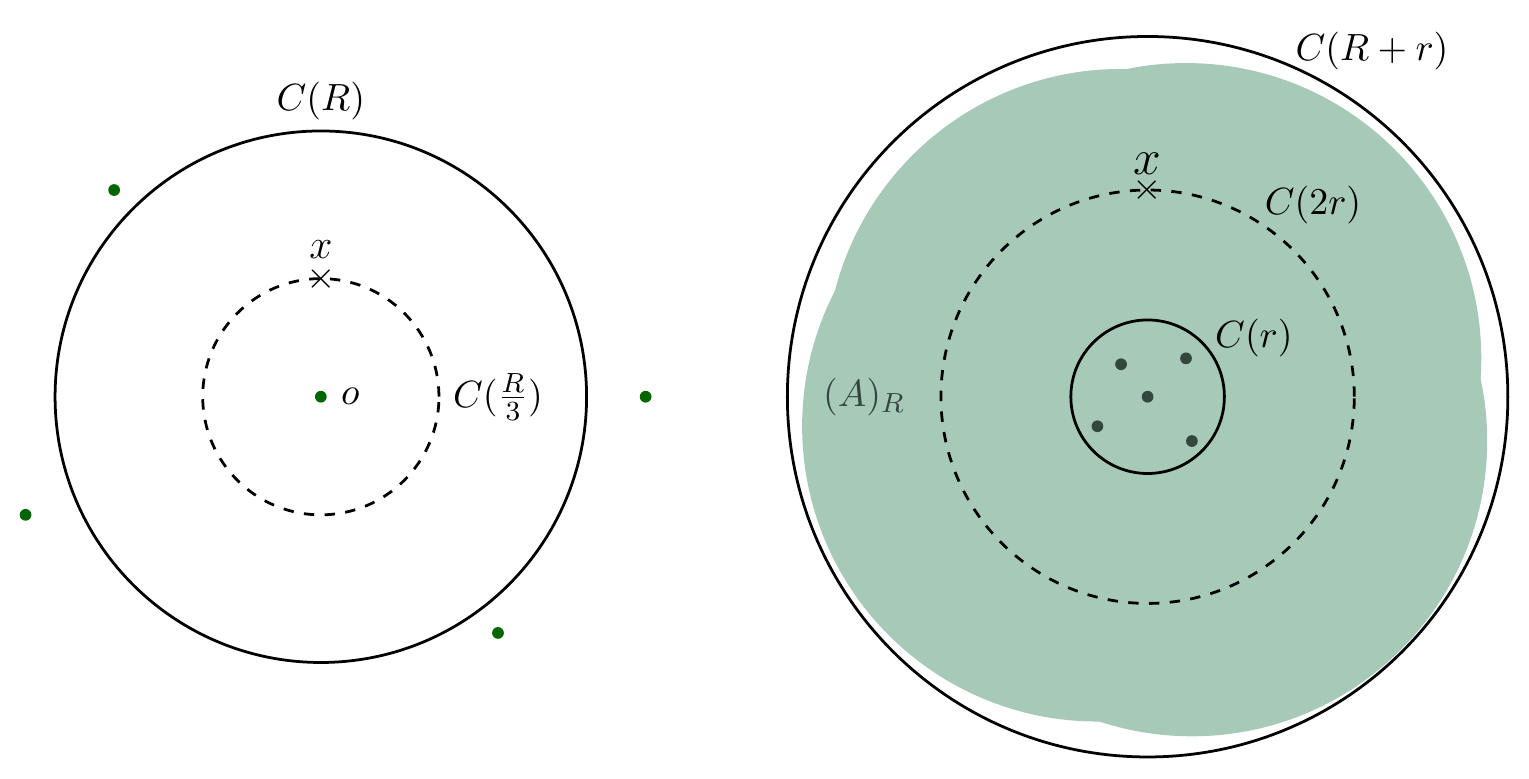}} 
\caption{Sparse sets like ones which appear in the proofs of Theorems~\ref{thm: hm} (left) and \ref{thm: esc} (right). The elements of $A$ are represented by dark green dots. On the left, $A {\setminus} \{o\}$ is a subset of $D(R)^c$. On the right, $A$ is a subset of $D(r)$ and the $R$-fattening of $A$ (shaded green) is a subset of $D(R+r)$. The figure is not to scale, as $R \geq e^n$ on the left, while $R \geq e^r$ on the right.}
\label{fig: diffs_fig_0}
\end{figure}

\begin{table}[ht]
\captionsetup{skip=2pt}
\centering
\caption{Summary of improvements to standard estimates in sparse settings. The origin is denoted by $o$ and $A_R$ denotes the set of all points in $\Z^d$ within a distance $R$ of $A$.}
\begin{tabular}{ |c|c|c|c|}
\hline
Setting & Quantity & Standard estimate & New estimate\\
\hline\hline
Fig.~\ref{fig: diffs_fig_0} (left), $R \geq e^n$ & $\P_x (\tau_{o} < \tau_{A \cap D(R)^c})$ & $\Omega \left( \frac{1}{\log R} \right)$ & $\Omega \left(\frac1n \right)$ \\[0.5ex]
\hline
Fig.~\ref{fig: diffs_fig_0} (right), $R \geq e^r$ & $\P_x (\tau_{\bdy A_R} < \tau_A)$ & $\Omega_n \left( \frac{1}{\log R} \right)$ & $\Omega_n \left( \frac{\log r}{\log R} \right)$ \\[0.5ex]
\hline
\end{tabular}
\label{table: sum}
\end{table}

For the scenario depicted in Figure~\ref{fig: diffs_fig_0} (left), we estimate the probability that a random walk from $x \in C(\tfrac{R}{3})$ hits the origin before any element of $A{\setminus} \{o\}$. Since $C(R)$ separates $x$ from $A{\setminus}\{o\}$, this probability is at least $\P_x (\tau_o < \tau_{C(R)})$. 
We can calculate this lower bound by combining the fact that the potential kernel (defined in Section \ref{sec: hm est}) is harmonic away from the origin with the optional stopping theorem (e.g., Proposition~1.6.7 of \cite{lawler2013intersections}):
\begin{equation*}
\P_x \left( \tau_{o} < \tau_{C(R)} \right) = \frac{\log R - \log |x| + O (R^{-1})}{\log R + O (R^{-1})}.
\end{equation*}
This implies $\P_x (\tau_{o} < \tau_{A \cap D(R)^c}) = \Omega (\tfrac{1}{\log R})$, since $x \in C(\tfrac{R}{3})$ and $R \geq e^n$.

We can improve the lower bound to $\Omega (\tfrac{1}{n})$ by using the sparsity of $A$. We define the random variable $W = \sum_{y \in A{\setminus} \{o\}} \1 \left( \tau_y < \tau_{o} \right)$ and write \[ \P_x \left
(\tau_{o} < \tau_{A {\setminus} \{o\}} \right) = \P_x \left( W = 0 \right) = 1 - \frac{\E_x W}{\E_x [ W \bigm\vert W > 0 ]}.\] 
We will show that $\E_x [ W \bigm\vert W > 0 ] \geq \E_x W + \delta$ for some $\delta$ which is uniformly positive in $A$ and $n$. We will be able to find such a $\delta$ because random walk from $x$ hits a given element of $A{\setminus}\{o\}$ before $o$ with a probability of at most $1/2$, so conditioning on $\{W>0\}$ effectively increases $W$ by $1/2$. Then
\[ \P_x \left( \tau_o < \tau_{A {\setminus} \{o\}} \right) \geq 1 - \frac{\E_x W}{\E_x W + \delta} \geq 1 - \frac{n}{n+\delta} = \Omega (\tfrac{1}{n}).\]
The second inequality follows from the monotonicity of $\tfrac{\E_x W}{\E_x W + \delta}$ in $\E_x W$ and the fact that $|A| \leq n$, so $\E_x W \leq n$. This is a better lower bound than $\Omega (\tfrac{1}{\log R})$ when $R$ is at least $e^n$.

A variation of this method also improves a standard estimate for the scenario depicted in Figure~\ref{fig: diffs_fig_0} (right). In this case, we estimate the probability that a random walk from $x \in C(2r)$ hits $\bdy A_R$ before $A$, where $A$ is contained in $D(r)$ and $A_R$ consists of all elements of $\Z^2$ within a distance $R \geq e^r$ of $A$.
We can bound below this probability using the fact that
\[ \P_x \left( \tau_{\bdy A_R} < \tau_A \right) \geq \P_x ( \tau_{C(R+r)} < \tau_{C(r)} ). \]
A standard calculation using the potential kernel of random walk (e.g., Exercise 1.6.8 of \cite{lawler2013intersections}) shows that this lower bound is $\Omega_n (\tfrac{1}{\log R})$, since $R \geq e^r$ and $r = \Omega (n^{1/2})$.

We can improve the lower bound to $\Omega_n (\tfrac{\log r}{\log R})$ by using the sparsity of $A$. We define $W' = \sum_{y \in A} \1 \left( \tau_y < \tau_{\bdy A_R} \right)$ and write \[ \P_x \left( \tau_{\bdy A_R} < \tau_A \right) = 1 - \frac{\E_x W'}{\E_x [W' \bigm\vert W' > 0 ]} \geq 1 - \frac{n\alpha}{1 + (n-1)\beta},\] where $\alpha$ bounds above $\P_x \left( \tau_y  < \tau_{\bdy A_R}\right)$ and $\beta$ bounds below $\P_z \left( \tau_y < \tau_{\bdy A_R}\right)$, uniformly for $x \in C(2r)$ and distinct $y,z \in A$. 
We will show that $\alpha \leq \beta$ and $\beta \leq 1 - \tfrac{\log (2r)}{\log R}$. The former is plausible because $|x-y|$ is at least as great as $|y-z|$; the latter because $\dist (z,A) \geq R$ while $|y-z| \leq 2r$, and because of \eqref{potker}. 
We apply these facts to the preceding display to conclude
\[ \P_x \left( \tau_{\bdy A_R} < \tau_A \right) \geq n^{-1}(1-\beta) = \Omega_n (\tfrac{\log r}{\log R}).\]
This is a better lower bound than $\Omega_n (\tfrac{1}{\log R})$ because $r$ can be as large as $\log R$. 

In summary, by analyzing certain conditional expectations, we can better estimate hitting probabilities for sparse sets than we can by applying standard results. This approach may be useful in obtaining other sparse analogues of hitting probability estimates.


\section{Harmonic measure estimates}\label{sec: hm est}

The purpose of this section is to prove Theorem~\ref{thm: hm}. We will describe the proof strategy in Section \ref{subsec: hm strat}, before proving several estimates in Section \ref{prelims} which will streamline the presentation of the proof in Section \ref{subsec: thm hm}. The majority of our effort is devoted to the proof of \eqref{eq: ext hm thm}; we will obtain \eqref{eq: ext hm thm2} as a corollary of a geometric lemma in Section \ref{stage 4}.

Consider a subset $A$ of $\Z^2$ with $n \geq 2$ elements, which satisfies $\H_A (o) > 0$ (i.e., $A \in \mathscr{H}_n$). We frame the proof of Theorem~\ref{thm: hm}---in particular, the proof of \eqref{eq: ext hm thm}---in terms of ``advancing'' a random walk from infinity to the origin in three or four stages, while avoiding all other elements of $A$. These stages are defined in terms of a sequence of annuli which partition $\Z^2$.

Denote the disk of radius $r$ about $x$ by $D_x (r) = \{y \in \Z^2: |x - y| < r \}$, or $D(r)$ if $x = o$, and denote its boundary by $C_x (r) = \bdy D_x (r)$, or $C(r)$ if $x = o$. Additionally, denote by $\A (r,R) = D(R) {\setminus} D(r)$ the annulus with inner radius $r$ and outer radius $R$. We will frequently need to reference the subset of $A$ which lies within or beyond a disk. We denote $A_{< r} = A \cap D(r)$ and $A_{\geq r} = A \cap D(r)^c$.

Define radii $R_1, R_2, \dots$ and annuli $\A_1, \A_2, \dots$ through $R_1 = 10^5$, and $R_\ell = R_1^\ell$ and $\A_\ell = \A (R_\ell, R_{\ell+1})$ for $\ell \geq 1$. We fix $\delta = 10^{-2}$ for use in intermediate scales, like $C(\delta R_{\ell+1}) \subset \A_\ell$. Additionally, we denote by $n_{0}$, $n_\ell$, $m_\ell$, and $n_{>J}$ the number of elements of $A$ in $D(R_1)$, $\A_\ell$, $\A_\ell \cup \A_{\ell+1}$, and $D(R_{J+1})^c$, respectively.

We will split the proof of \eqref{eq: ext hm thm} into an easy case when $n_0 = n$ and a difficult case when $n_0 \neq n$. If $n_0 \neq n$, then $A_{\geq R_1}$ is nonempty and the following indices $\ii = \ii(A)$ and $\jj =\jj(A)$ are well defined:
\begin{align*}
\ii &= \min \{\ell \geq 1: \text{$\A_\ell$ contains an element of $A {\setminus} \{o\}$} \}, \,\,\text{and}\\ 
\jj &= \min \{\ell > \ii: \text{$\A_\ell$ contains no element of $A {\setminus} \{o\}$}\}.
\end{align*}
We explain the roles of $I$ and $J$ in the following subsection.

\subsection{Strategy for the proof of Theorem~\ref{thm: hm}}\label{subsec: hm strat}
This section outlines a proof of \eqref{eq: ext hm thm} by induction on $n$. The induction step is easy when $n_0 = n$; the following strategy concerns the difficult case when $n_0 \neq n$. The proof of \eqref{eq: ext hm thm2} is a simple consequence of an input to the proof of \eqref{eq: ext hm thm}, so we address it separately, in Section \ref{stage 4}.

{\em Stage 1: Advancing to $C(R_{\jj})$.} Assume $n_0 \neq n$ and $n \geq 3$. By the induction hypothesis, there is universal constant $c_1$ such that the harmonic measure at the origin is at least $e^{-c_1 k \log k}$, for any set in $\mathscr{H}_k$, $1 \leq k < n$. Denote the law of random walk from $\infty$ by $\P$ (without a subscript) and let $k = n_{>J}+1$. Because a random walk from $\infty$ which hits the origin before $A_{\geq R_{\jj}}$ also hits $C(R_{\jj})$ before $A$, the induction hypothesis applied to $A_{\geq R_{\jj}}\cup \{o\} \in \mathscr{H}_{k}$ implies that $\P (\tau_{C(R_{\jj})} < \tau_A)$ is no smaller than exponential in $k \log k$. Note that $k < n$ because $A_{<R_{\ii+1}}$ has at least two elements by the definition of $I$.

The reason we advance the random walk to $C(R_{\jj})$ instead of the boundary of a smaller disk is that an adversarial choice of $A$ could produce a ``choke point'' which likely dooms the walk to be intercepted by $A{\setminus} \{o\}$ in the second stage of advancement (Figure~\ref{fig: patho}). To avoid a choke point when advancing to the boundary of a disk $D$, it suffices for the conditional hitting distribution of $\bdy D$ given $\{\tau_{\bdy D} < \tau_A\}$ to be comparable to the uniform hitting distribution on $\bdy D$. To prove this comparison, the annular region immediately beyond $D$ and extending to a radius at least twice that of $D$ must be empty of $A$, hence the need for exponentially growing radii and for $\A_\jj$ to be empty of $A$.

\begin{figure}[htbp]
\centering {\includegraphics[width=0.8\linewidth]{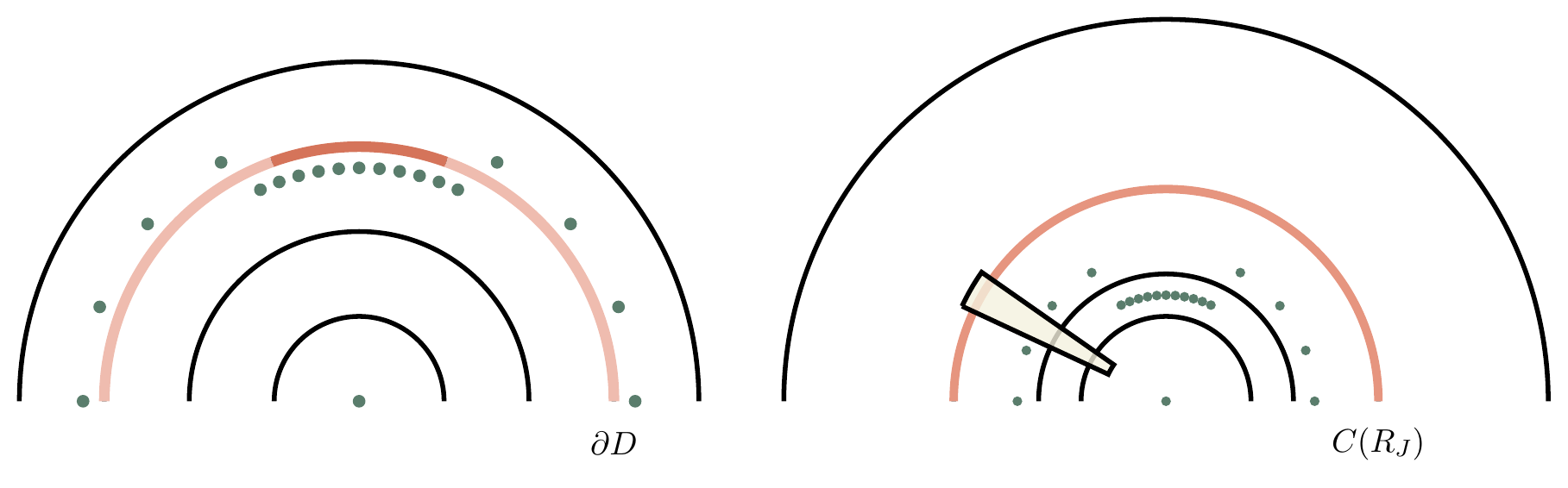}}
\caption{An example of a ``choke point'' (left) and a strategy for avoiding it (right). The hitting distribution of a random walk conditioned to reach $\bdy D$ before $A$ (green dots) may favor the avoidance of $A \cap D^c$ in a way which localizes the walk (e.g., as indicated by the dark red arc of $\bdy D$) prohibitively close to $A \cap D$. The hitting distribution on $C(R_{J})$ will be approximately uniform if the radii grow exponentially. The random walk can then avoid the choke point by ``tunneling'' through it (e.g., by passing through the tan-shaded region).}
\label{fig: patho}
\end{figure}

{\em Stage 2: Advancing into $\A_{\ii-1}$.} For notational convenience, assume $I \geq 2$ so that $\A_{I-1}$ is defined; the argument is the same when $I=1$. Each annulus $\A_\ell$, $\ell \in \{\ii,\dots,\jj - 1\}$, contains one or more elements of $A$, which the random walk must avoid on its journey to $\A_{\ii-1}$. We build an overlapping sequence of rectangular and annular {\em tunnels}, through and between each annulus, which are empty of $A$ and through which the walk can enter $\A_{\ii-1}$ (Figure~\ref{fig: big_tunnel}). (In fact, depending on $A$, we may not be able to tunnel into $\A_{\ii-1}$, but this case will be easier; we address it at the end of this subsection.) Specifically, the walk reaches a particular subset $\Arc_{\ii-1}$ in $\A_{\ii-1}$ at the conclusion of the tunneling process. We will define $\Arc_{\ii-1}$ in Lemma \ref{lem: pigeon} as an arc of a circle in $\A_{\ii-1}$.

By the pigeonhole principle applied to the radial coordinate, for each $\ell \geq I+1$, there is a sector of aspect ratio $m_\ell = n_\ell + n_{\ell-1}$, from the lower ``$\delta$\textsuperscript{th}'' of $\A_{\ell}$ to that of $\A_{\ell-1}$, which contains no element of $A$ (Figure \ref{fig: big_tunnel}). To reach the entrance of the analogous tunnel between $\A_{\ell-1}$ and $\A_{\ell-2}$, the random walk may need to circle the lower $\delta$\textsuperscript{th} of $\A_{\ell-1}$. We apply the pigeonhole principle to the angular coordinate to conclude that there is an annular region contained in the lower $\delta$\textsuperscript{th} of $\A_{\ell-1}$, with an aspect ratio of $n_{\ell-1}$, which contains no element of $A$.

The probability that the random walk reaches the annular tunnel before exiting the rectangular tunnel from $\A_{\ell}$ to $\A_{\ell-1}$ is no smaller than exponential in $m_\ell$. Similarly, the random walk reaches the rectangular tunnel from $\A_{\ell-1}$ to $\A_{\ell-2}$ before exiting the annular tunnel in $\A_{\ell-1}$ with a probability no smaller than exponential in $n_{\ell-1}$. Overall, we conclude that the random walk reaches $\Arc_{\ii-1}$ without leaving the union of tunnels---and therefore without hitting an element of $A$---with a probability no smaller than exponential in $\sum_{\ell = \ii}^{\jj - 1} n_\ell$.

\begin{figure}[htbp]
\centering {\includegraphics[width=0.7\linewidth]{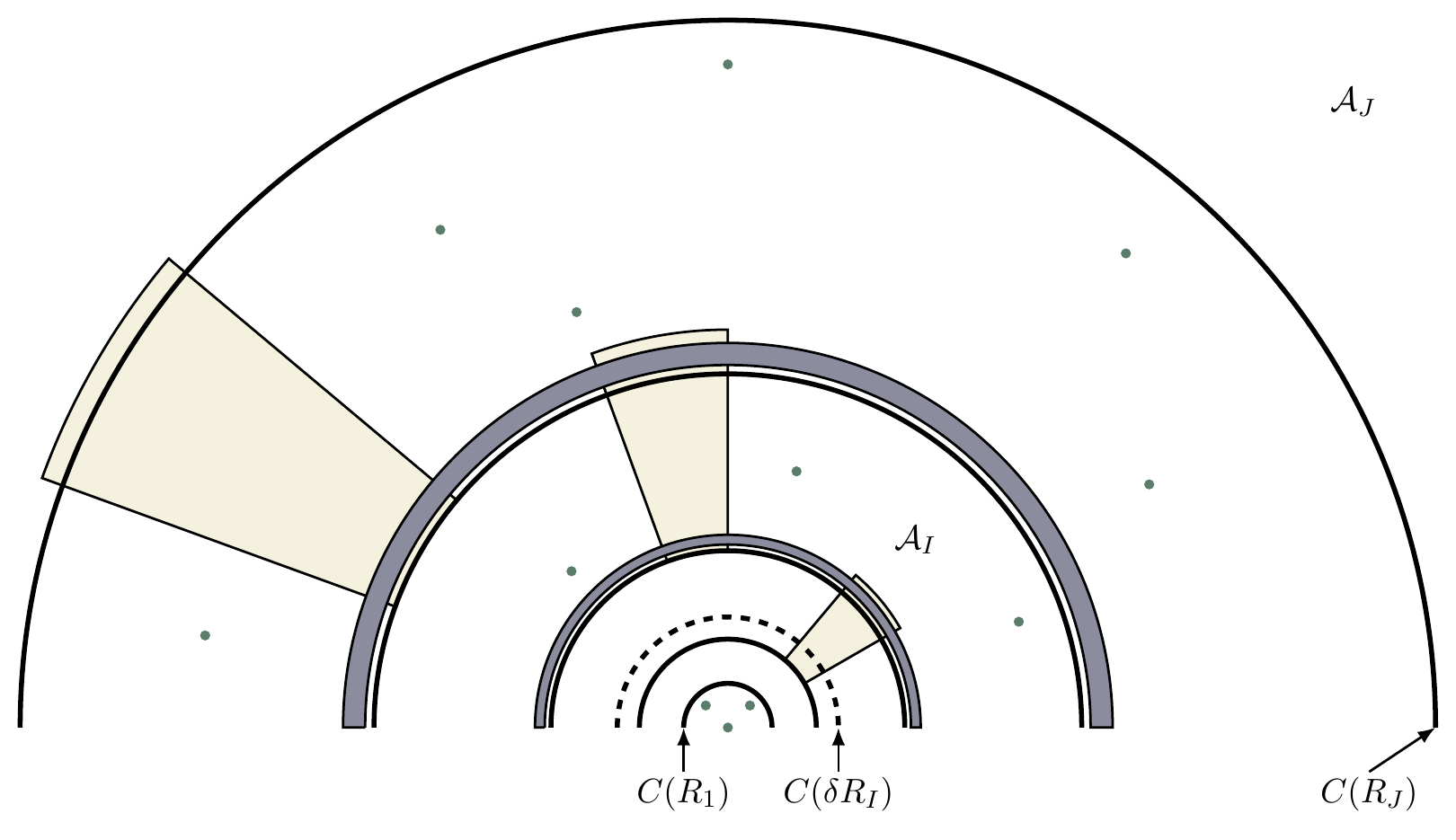}}
\caption{Tunneling through nonempty annuli. We construct a contiguous series of sectors (tan) and annuli (blue) which contain no elements of $A$ (green dots) and through which the random walk may advance from $C(R_{\jj - 1})$ to $C(\delta R_{\ii - 1})$ (dashed).}
\label{fig: big_tunnel}
\end{figure}

{\em Stage 3: Advancing to $C(R_1)$.} Figure~\ref{fig: diffs_fig_0} (left) essentially depicts the setting of the random walk upon reaching $x \in \Arc_{\ii-1}$, except with $C(R_\ii)$ in the place of $C(R)$ and the circle containing $\Arc_{\ii-1}$ in the place of $C(\frac{R}{3})$, and except for the possibility that $D(R_1)$ contains other elements of $A$. Nevertheless, if the radius of $\Arc_{\ii-1}$ is at least $e^n$, then by pretending that $A_{<R_1} = \{o\}$, the method highlighted in Section \ref{subsec: novel} will show that $\P_x (\tau_{C(R_1)} < \tau_{A}) = \Omega (\frac1n)$. A simple calculation will give the same lower bound (for a potentially smaller constant) in the case when the radius is less than $e^n$.

{\em Stage 4: Advancing to the origin.} Once the random walk reaches $C(R_1)$, we are in the setting of Lemma~\ref{lem: n path}. There can be no more than $O(R_1^2)$ elements of $A_{<R_1}$, so there is a path of length $O(R_1^2)$ to the origin which avoids all other elements of $A$, and a corresponding probability of at least a constant that the random walk follows it.

{\em Conclusion of Stages 1--4.} The lower bounds from the four stages imply that there are universal constants $c_1$ through $c_4$ such that
\begin{equation*}
\H_A (o) \geq e^{-c_1 k \log k - c_2 \sum_{\ell=I}^{J-1} n_\ell - \log (c_3 n) - \log c_4} \geq e^{-c_1 n \log n}.\end{equation*} It is easy to show that the second inequality holds if $c_1 \geq 8\max\{1, c_2, \log c_3,\log c_4\}$, using the fact that $n - k = \sum_{\ell=I}^{\jj-1} n_\ell > 1$ and $\log n \geq 1$. We are free to adjust $c_1$ to satisfy this bound, because $c_2$ through $c_4$ do not depend on the induction hypothesis. This concludes the induction step.

{\em A complication in Stage 2.} If $R_{\ell}$ is not sufficiently large relative to $m_\ell$, then we cannot tunnel the random walk through $\A_{\ell}$ into $\A_{\ell-1}$. We formalize this through the failure of the condition
\begin{equation}\label{eq: case1}
\delta R_{\ell} > R_1 (m_\ell+1).
\end{equation}
The problem is that, if \eqref{eq: case1} fails, then there are too many elements of $A$ in $\A_{\ell}$ and $\A_{\ell-1}$, and we cannot guarantee that there is a tunnel between the annuli which avoids $A$. We note that, while it may seem that this problem could be avoided by choosing $R_1$ in proportion to $n$, this choice would ultimately worsen \eqref{eq: ext hm thm} to $e^{-cn^2}$.

Accordingly, we will stop {\em Stage 2} tunneling once the random walk reaches a particular subset $\Arc_{K-1}$ of a circle in $\A_{K-1}$, where $\A_{K-1}$ is the outermost annulus which fails to satisfy \eqref{eq: case1}. Specifically, we define $K$ as:
\begin{equation}
K = \begin{cases}
I, & \text{if \eqref{eq: case1} holds for $\ell \in \{\ii,\dots,\jj\}$;}\\ 
\min \{k \in \{\ii,\dots,\jj\}: \text{\eqref{eq: case1} holds for $\ell \in \{k,\dots,\jj\}$} \}, & \text{otherwise.}
\end{cases}\label{k}
\end{equation}

The failure of \eqref{eq: case1} for $\ell = K-1$ when $K \neq I$ will imply that there is a path of length $O(\sum_{\ell=I}^{K-1} n_\ell)$ from $\Arc_{K-1}$ to the origin which otherwise avoids $A$. In this case, {\em Stage 3} consists of random walk from $\Arc_{K-1}$ following this path to the origin with a probability no smaller than exponential in $\sum_{\ell=I}^{K-1} n_\ell$, and there is no {\em Stage 4}.

Overall, if $K \neq I$, {\em Stages 2,3} contribute a rate of $\sum_{\ell=I}^{J-1} n_\ell$. This rate is smaller than the one contributed by {\em Stages 2--4} when $K = I$, so the preceding conclusion holds.

\subsection{Preparation for the proof of Theorem~\ref{thm: hm}}\label{prelims}
First, we introduce some conventions, notation, and some objects associated with random walk.

All universal constants will be positive and finite. For subsets $B$ and elements $x$ of $\Z^2$, we will denote corresponding hitting times by $\sigma_B = \inf \{t \geq 0: S_t \in B\}$ or $\sigma_x$. For $r > 0$, we will denote the $r$-fattening of $B$ by $B_r = \left\{x \in \Z^2: \dist (x, B) < r\right\}$. We will use $\rad (C)$ to denote the radius of a circle $C$ (e.g., $\rad (C(r)) = r$). We will denote the minimum of random times $\tau_1$ and $\tau_2$ by $\tau_1 \wedge \tau_2$.

We will use the potential kernel associated with random walk on $\Z^2$. We denote the former by $\mathfrak{a}$. It has the form
\begin{equation}\label{potker} 
\mathfrak{a}(x)=\frac{2}{\pi}\log {|x|}+\kappa+O\left(|x|^{-2}\right),
\end{equation}
where $\kappa \in (1.02,1.03)$ is an explicit constant. The potential kernel satisfies $\mathfrak{a} (o) = 0$ and is harmonic on $\Z^2 {\setminus} \{o\}$. As shown in \cite{kozma2004asymptotic}, the constant hidden in the error term, which we call $\lambda$, is less than $0.06882$. In some instances, we will want to apply $\aa$ to an element which belongs to $C(r)$. It will be convenient to denote, for $r > 0$,
\begin{equation*}
\mathfrak{a}' (r) = \frac{2}{\pi} \log r + \kappa.
\end{equation*}

\subsubsection{Input to Stage 1}

Let $A \in \mathscr{H}_n$. Like in Section \ref{subsec: hm strat}, we assume that $n_0 \neq n$ (i.e., $A_{\geq R_1} \neq \emptyset$) and defer the simpler complementary case to Section \ref{subsec: thm hm}. The annulus $\A_{\jj}$ is important because of the following result. To state it, denote the uniform distribution on $C(R_{J})$ by $\mu_{J}$.

\begin{lemma}\label{lem: near unif} There is a constant $c_1$ such that, for every $z \in C(R_{\jj})$,
\begin{equation}\label{eq: near unif}
\P \big( S_{\tau_{C(R_{\jj})}} = z \bigm\vert \tau_{C(R_{\jj})} < \tau_A \big) \geq c_1 \mu_{\jj} (z).
\end{equation}
\end{lemma}

Under the conditioning in \eqref{eq: near unif}, the random walk reaches $C(\delta R_{\jj+1})$ before hitting $A$, and typically proceeds to hit $C(R_{J})$ before returning to $C(R_{J+1})$. The inequality \eqref{eq: near unif} then follows from the fact that harmonic measure on $C(R_\jj)$ is comparable to $\mu_\jj$.

\begin{proof}[Proof of Lemma \ref{lem: near unif}] 
Under the conditioning, the random walk must reach $C(\delta R_{\jj+1})$ before $C(R_{\jj})$. It therefore suffices to prove that there exists a positive constant $c_1$ such that, uniformly for all $x \in C(\delta R_{\jj+1})$ and $z \in C(R_{\jj})$, 
\begin{equation}\label{eq: unif low} \P_x \big( S_\eta = z \bigm\vert \tau_{C(R_{\jj})} < \tau_A \big) \geq c_1 \mu_{R_{\jj}} (z),
\end{equation}
where $\eta = \tau_{C(R_{\jj})} \wedge \tau_A$. Because $\bdy \A_{\jj}$ separates $x$ from $A$, the conditional probability in \eqref{eq: unif low} is at least
\begin{equation}\label{eq: double cond}
\P_x \big( S_\eta = z \bigm\vert \tau_{C(R_{\jj})} < \tau_{C(R_{\jj+1})}, \, \tau_{C(R_{\jj})} < \tau_A\big) \P_x \big( \tau_{C(R_{\jj})} < \tau_{C(R_{\jj+1})}\big).
\end{equation}
The first factor of \eqref{eq: double cond} simplifies to
\begin{equation}\label{eq: simple cond}
\P_x \big( S_{\tau_{C(R_{\jj})}} = z \bigm\vert \tau_{C(R_{\jj})} < \tau_{C(R_{\jj+1})}\big),
\end{equation}
which we will bound below using Lemma~\ref{lem: unif}.

We will verify the hypotheses of Lemma~\ref{lem: unif} with $\eps = \delta$ and $R = R_{\jj+1}$. The first hypothesis is $R \geq 10\eps^{-2}$, which is satisfied because $R_{\jj+1} \geq R_1 = 10\delta^{-2}$. The second hypothesis is \eqref{eq: min p} which, in our case, can be written as
\begin{equation}\label{a4 hyp}
\max_{x \in C(\delta R_{J+1})} \P_x \left( \tau_{C(R_{J+1})} < \tau_{C(R_{\jj})} \right) < \tfrac{9}{10}.
\end{equation}
Exercise~1.6.8 of \cite{lawler2013intersections} states that
\begin{equation}\label{app168}
\P_x \left( \tau_{C(R_{J+1})} < \tau_{C(R_{\jj})} \right) = \frac{\log (\tfrac{|x|}{R_{\jj}}) + O(R_{\jj}^{-1})}{\log (\tfrac{R_{J+1}}{R_{\jj}}) + O (R_{\jj}^{-1} + R_{J+1}^{-1})},
\end{equation}
where the implicit constants are at most $2$ (i.e., the $O(R_{\jj}^{-1})$ term is at most $2R_{\jj}^{-1}$). For the moment, ignore the error terms and assume $|x| = \delta R_{J+1}$, in which case \eqref{app168} evaluates to $\tfrac{5\log 10 - \log 25}{5\log 10} < 0.73$. Because $R_{J} \geq 10^5$, even after allowing $|x|$ up to $\delta R_{J+1} + 1$ and accounting for the error terms, \eqref{app168} is less than $\frac{9}{10}$, which implies \eqref{a4 hyp}.

Applying Lemma~\ref{lem: unif} to \eqref{eq: simple cond}, we obtain a constant $c_2$ such that
\begin{equation}\label{eq: simple cond 2}
\P_x \big( S_{\tau_{C(R_{\jj})}} = z \bigm\vert \tau_{C(R_{\jj})} < \tau_{C(R_{\jj+1})}\big) \geq c_2 \mu_{\jj} (z).
\end{equation} By \eqref{a4 hyp}, the second factor of \eqref{eq: double cond} is bounded below by $\frac{1}{10}$. We conclude the claim of \eqref{eq: unif low} by combining this bound and \eqref{eq: simple cond 2} with \eqref{eq: double cond}, and by setting $c_1 = \frac{1}{10} c_2$.
\end{proof}

\subsubsection{Inputs to Stage 2}

We continue to assume that $n_0 \neq n$, so that $I$, $J$, and $K$ are well defined; the $n_0 = n$ case is easy and we address it in Section \ref{subsec: thm hm}. In this subsection, we will prove an estimate of the probability that a random walk passes through annuli $\A_{J-1}$ to $\A_K$ without hitting $A$. First, in Lemma \ref{lem: pigeon}, we will identify a sequence of ``tunnels'' through the nonempty annuli, which are empty of $A$. Second, in Lemma \ref{lem: sec tun} and Lemma \ref{lem: arc tun}, we will show that random walk traverses these tunnels through a series of rectangles, with a probability which is no smaller than exponential in the number of elements in $\A_K, \dots, \A_{J-1}$. We will combine these estimates in Lemma \ref{lem: to l}.

Recall from Section \ref{subsec: hm strat} that $\A_K$ is the last annulus before the random walk encounters an annulus which fails to satisfy \eqref{eq: case1}. We call the set of such $\ell$ by $\II = \{K,\dots,\jj\}$. For each $\ell \in \II$, we define the annulus $\B_\ell = \A (R_{\ell - 1}, \delta R_{\ell+1})$. The inner radius of $\B_\ell$ is at least $R_1$ because
\begin{equation*}
\ell \in \II \implies R_\ell > \delta^{-1} R_1 (m_\ell + 1) \geq 10^7 \implies \ell \geq 2.
\end{equation*}
The first implication is due to \eqref{eq: case1} and \eqref{k}; the second is due to the fact that $R_\ell = 10^{5\ell}$.

The following lemma identifies subsets of $\B_\ell$ which are empty of $A$ (Figure \ref{fig: emptyregions}). Recall that $m_\ell = n_\ell + n_{\ell-1}$.

\begin{figure}[htbp]
\centering {\includegraphics[width=0.75\linewidth]{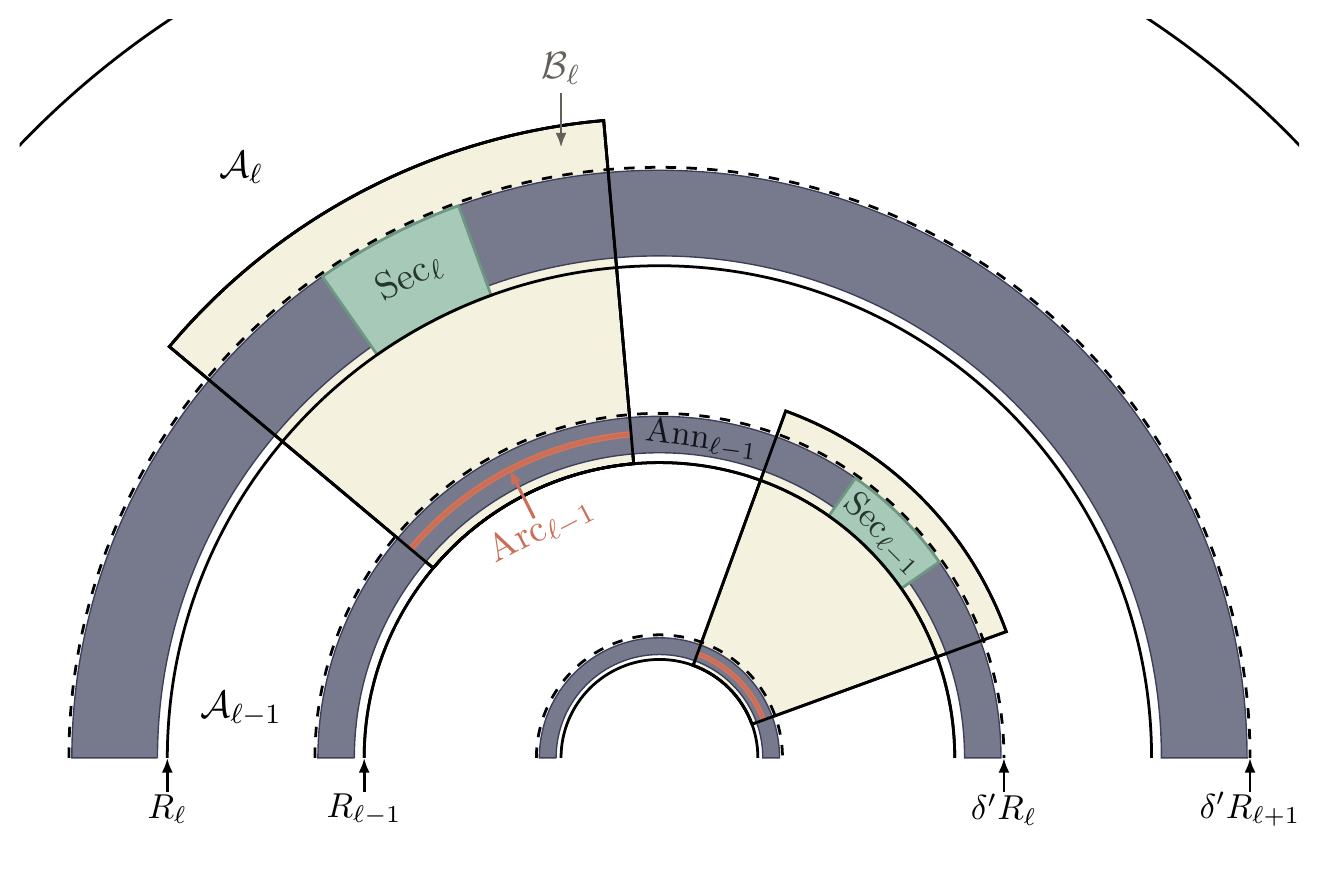}}
\caption{The regions identified in Lemma \ref{lem: pigeon}.}
\label{fig: emptyregions}
\end{figure} 

\begin{lemma}\label{lem: pigeon} Let $\ell \in \II$. Denote $\eps_\ell = (m_\ell+1)^{-1}$ and $\delta' = \delta/10$. For every $\ell \in \II$, there is an angle $\vartheta_\ell \in [0,2\pi)$ and a radius $a_{\ell-1} \in [10 R_{\ell-1}, \delta' R_{\ell})$
such that the following regions contain no element of $A$:
\begin{itemize}
\item the sector of $\B_\ell$ subtending the angular interval $\left[\vartheta_\ell, \vartheta_\ell + 2\pi \eps_\ell \right)$ and, in particular, the ``middle third'' sub-sector
\[ \Sec_\ell = \left[ R_{\ell}, \delta' R_{\ell+1} \right) \times \left[ \vartheta_\ell + \tfrac{2\pi}{3} \eps_\ell, \, \vartheta_\ell + \tfrac{4\pi}{3} \eps_\ell \right);\,\,\,\text{and}\]
\item the sub-annulus $\Ann_{\ell-1} = \A (a_{\ell-1}, b_{\ell-1})$ of $\B_\ell$, where we define
\begin{equation*}
b_{\ell-1} = a_{\ell-1} + \Delta_{\ell-1} \quad \text{for} \quad \Delta_{\ell-1} = \delta' \eps_{\ell} R_{\ell}
\end{equation*} and, in particular, the circle $\Circ_{\ell-1} = C\big( \tfrac{a_{\ell-1} + b_{\ell-1}}{2} \big)$ and the ``arc''
\[ \Arc_{\ell-1} = \Circ_{\ell-1} \cap \left\{x \in \Z^2: \arg x \in \left[\vartheta_\ell, \vartheta_\ell + 2\pi \eps_\ell \right) \right\}.\]
\end{itemize}
\end{lemma}

We take a moment to explain the parameters and regions. Aside from $\B_\ell$, which overlaps $\A_\ell$ and $\A_{\ell-1}$, the subscripts of the regions indicate which annulus contains them (e.g., $\Sec_\ell \subset \A_\ell$ and $\Ann_{\ell-1} \subset \A_{\ell-1}$). The proof uses the pigeonhole principle to identify regions which contain none of the $m_\ell$ elements of $A$ in $\B_\ell$ and $\Ann_{\ell-1}$; this motivates our choice of $\eps_\ell$. The key aspect of $\Sec_\ell$ is that it is separated from $\bdy \B_\ell$ by a distance of at least $R_{\ell-1}$. We also need the inner radius of $\Ann_{\ell-1}$ to be at least $R_{\ell-1}$ greater than that of $\B_\ell$, hence the lower bound on $a_{\ell-1}$. 
The other key aspect of $\Ann_{\ell-1}$ is its overlap with $\Sec_{\ell-1}$. 
The specific constants (e.g., $\tfrac{2\pi}{3}$, $10$, and $\delta'$) are otherwise unimportant.

\begin{proof}[Proof of Lemma \ref{lem: pigeon}] Fix $\ell \in \II$. For $j \in \{0,\dots, m_\ell\}$, form the intervals
\[ 2\pi \eps_\ell \left[j, j+1 \right) \,\,\, \text{and} \,\,\, 10 R_{\ell - 1} + \Delta_{\ell-1} [j,j+1).\]
$\B_\ell$ contains at most $m_\ell$ elements of $A$, so the pigeonhole principle implies that there are $j_1$ and $j_2$ in this range and such that, if $\vartheta_\ell = j_1 2 \pi \eps_\ell$ and if $a_{\ell-1} = 10 R_{\ell-1} + j_2 \Delta_{\ell-1}$, then
\[ \B_\ell \cap \left\{x \in \Z^2: \arg x \in \big[\vartheta_\ell, \vartheta_\ell + 2 \pi \eps_\ell \big) \right\} \cap A = \emptyset, \quad \text{and} \quad \A (a_{\ell-1}, a_{\ell-1} + \Delta_{\ell-1} ) \cap A = \emptyset.
\]
Because $\B_\ell \supseteq \Sec_\ell$ and $\Ann_{\ell-1} \supseteq \Arc_{\ell-1} $, for these choices of $\vartheta_\ell$ and $a_{\ell-1}$, we also have $\Sec_\ell \cap A = \emptyset$ and $\Arc_{\ell-1} \cap A = \emptyset$.
\end{proof}

The next result bounds below the probability that the random walk tunnels ``down'' from $\Sec_\ell$ to $\Arc_{\ell-1}$. We state it without proof, as it is a simple consequence of the fact that random walk exits a rectangle through its far side with a probability which is no smaller than exponential in the aspect ratio of the rectangle (Lemma \ref{lem: aspect ratio}). In this case, the aspect ratio is $O(m_\ell)$.

\begin{lemma}\label{lem: sec tun} There is a constant $c$ such that, for any $\ell \in \II$ and every $y \in \Sec_\ell$,
\begin{equation*}
\P_y \big( \tau_{\Arc_{\ell-1}} < \tau_A \big) \geq c^{m_\ell}.
\end{equation*}
\end{lemma}  

The following lemma bounds below the probability that the random walk tunnels ``around'' $\Ann_{\ell-1}$, from $\Arc_{\ell-1}$ to $\Sec_{\ell-1}$. Like Lemma \ref{lem: sec tun}, we state it without proof because it is a simple consequence of Lemma \ref{lem: aspect ratio}. Indeed, random walk from $\Arc_{\ell-1}$ can reach $\Sec_{\ell-1}$ without exiting $\Ann_{\ell-1}$ by appropriately exiting each rectangle in a sequence of $O(m_\ell)$ rectangles of aspect ratio $O(1)$. Applying Lemma \ref{lem: aspect ratio} then implies \eqref{eq: arc tun}.

\begin{lemma}\label{lem: arc tun}
There is a constant $c$ such that, for any $\ell \in \II$ and every $z \in \Arc_{\ell-1}$,
\begin{equation}\label{eq: arc tun}
\P_z \big( \tau_{\Sec_{\ell -1}} < \tau_A \big) \geq c^{m_\ell}.
\end{equation}
\end{lemma}

The next result combines Lemma \ref{lem: sec tun} and Lemma \ref{lem: arc tun} to tunnel from $\A_J$ into $\A_{K-1}$. Because the random walk tunnels from $\A_\ell$ to $\A_{\ell-1}$ with a probability no smaller than exponential in $m_\ell = n_\ell + n_{\ell-1}$, the bound in \eqref{eq: to l} is no smaller than exponential in $\sum_{\ell=K-1}^{J-1} n_\ell$ (recall that $n_J = 0$).

\begin{lemma}\label{lem: to l} There is a constant $c$ such that
\begin{equation}\label{eq: to l}
\P_{\, \mu_{\jj}}  \left( \tau_{\Arc_{K-1}} < \tau_A \right) \geq c^{\sum_{\ell = K-1}^{\jj - 1} n_\ell}.
\end{equation}
\end{lemma}

\begin{proof}
Denote by $G$ the event
\[ \big\{ \tau_{\Arc_{J-1}} < \tau_{\Sec_{J-1}} < \tau_{\Arc_{J-2}} < \cdots < \tau_{\Arc_{K}} < \tau_{\Sec_{K}} < \tau_{\Arc_{K-1}} < \tau_A \big\}.\]
Lemma \ref{lem: sec tun} and Lemma \ref{lem: arc tun} imply that there is a constant $c_1$ such that 
\begin{equation}\label{to l2}
\P_z (G) \geq c_1^{\sum_{\ell=K-1}^{J-1} n_\ell} \,\,\,\text{for $z \in C(R_{\jj}) \cap \Sec_{\jj}$.}
\end{equation}
The intersection of $\Sec_{\jj}$ and $C(R_{\jj})$ subtends an angle of at least $n_{\jj-1}^{-1}$, so there is a constant $c_2$ such that
\begin{equation}\label{to l3}
\mu_{\jj} (\Sec_{\jj}) \geq c_2 n_{\jj-1}^{-1}.
\end{equation}
The inequality \eqref{eq: to l} follows from $G \subseteq \{\tau_{\Arc_{K-1}} < \tau_A \}$, and \eqref{to l2} and \eqref{to l3}:
\[ 
\P_{\, \mu_{\jj}}  \left( \tau_{\Arc_{K-1}} < \tau_A \right) \geq \P_{\, \mu_J} (G) \geq c_2 n_{\jj-1}^{-1} \cdot c_1^{\sum_{\ell=K-1}^{J-1} n_\ell} \geq c_3^{\sum_{\ell=K-1}^{J-1} n_\ell}.
\] For the third inequality, we take $c_3 = (c_1 c_2)^2$.
\end{proof}

\subsubsection{Inputs to Stage 3 when $K=I$}

We continue to assume that $n_0 \neq n$, as the alternative case is addressed in Section \ref{subsec: thm hm}. Additionally, we assume $K = I$. We briefly recall some important context. When $K=I$, at the end of {\em Stage 2}, the random walk has reached $\Circ_{I-1} \subseteq \A_{I-1}$, where $\Circ_{I-1}$ is a circle with a radius in $[R_{I-1}, \delta' R_I)$. Since $\A_I$ is the innermost annulus which contains an element of $A$, the random walk from $\Arc_{I-1}$ must simply reach the origin before hitting $A_{> R_I}$. 
In this subsection, we estimate this probability.

We will need the following standard hitting probability estimate (see, for example, Proposition~1.6.7 of \cite{lawler2013intersections}), which we state as a lemma because we will use it in other sections as well.

\begin{lemma}\label{lem: gr est}
Let $y \in D_x (r)$ for $r \geq 2 (|x| + 1)$ and assume $y \neq o$. Then \begin{equation}\label{eq: hit or} \P_y \left( \tau_o < \tau_{C_x (r)} \right) = \frac{\mathfrak{a}' (r) - \mathfrak{a} (y) + O \left( \frac{| x | + 1}{r} \right)}{ \mathfrak{a}' (r) + O \left( \frac{|x| + 1}{r} \right)}.\end{equation} The implicit constants in the error terms are less than one. 
\end{lemma}

If $R_{\ii} < e^{4n}$, then no further machinery is needed to prove the {\em Stage 3} estimate.
\begin{lemma}\label{lem: hit is har}
There exists a constant $c$ such that, if $R_{\ii} < e^{4n}$, then
\begin{equation*}
\P \left( \tau_{C(R_1)} < \tau_{A} \bigm\vert \tau_{\Circ_{I-1}} < \tau_A \right) \geq \frac{c}{n}.
\end{equation*}
\end{lemma} 

The bound holds because the random walk must exit $D(R_I)$ to hit $A_{\geq R_I}$. By a standard hitting estimate, the probability that the random walk hits the origin first is inversely proportional to $\log R_I$ which is $O(n)$ when $R_I < e^{4n}$.

\begin{proof}[Proof of Lemma \ref{lem: hit is har}]
Uniformly for $y \in \Circ_{I-1}$, we have
\begin{equation}\label{hit is har1}
\P_y \left( \tau_{C(R_1)} < \tau_A \right) \geq \P_y \left( \tau_{o} < \tau_{C(R_{\ii})} \right) \geq \frac{\aa' (R_{\ii}) - \aa' (\delta R_{I-1}) - \tfrac{1}{R_{\ii}} - \tfrac{1}{\delta R_{I}}}{\aa' (R_{\ii}) + \tfrac{1}{R_{\ii}}} \geq \frac{1}{\aa' (R_{\ii})}.
\end{equation}
The first inequality follows from the observation that $C(R_1)$ and $C(R_{I})$ separate $y$ from $o$ and $A$. The second inequality is due to Lemma~\ref{lem: gr est}, where we have replaced $\aa (y)$ by $\aa' (\delta R_{I}) + \tfrac{1}{\delta R_{I}}$ using \eqref{eq: hit or} of Lemma \ref{lem: bdy est} and the fact that $|y| \leq \delta R_{I}$. The third inequality follows from $\delta R_{\ii} \geq 10^3$. To conclude, we substitute $\aa' (R_{I}) = \tfrac{2}{\pi} \log R_{I} + \kappa$ into \eqref{hit is har1} and use assumption that $R_{I} < e^{4n}$.
\end{proof}

We will use the rest of this subsection to prove the bound of Lemma \ref{lem: hit is har}, but under the complementary assumption $R_{I} \geq e^{4n}$. This is one of the two estimates we highlighted in Section \ref{subsec: novel}.

Next is a standard result, which enables us to express certain hitting probabilities in terms of the potential kernel. We include a short proof for completeness.
\begin{lemma}\label{lem: w and p} For any pair of points $x, y \in \Z^2$, define
\begin{equation*}
M_{x,y}(z) =\frac{ \aa (x-z)-\aa(y-z)}{2\aa(x-y)}+\frac{1}{2}.
\end{equation*}
Then $M_{x,y} (z) = \P_z (\sigma_y < \sigma_x)$.
\end{lemma}

\begin{proof} Fix $x, y \in \Z^2$. Theorem~1.4.8 of \cite{lawler2013intersections} states that for any proper subset $B$ of $\Z^2$ (including infinite $B$) and bounded function $F: \bdy B \to \R$, the unique bounded function $f: B \cup \bdy B \to \R$ which is harmonic in $B$ and equals $F$ on $\bdy B$ is $f(z) = \E_z [ F (S_{\sigma_{\bdy B}})] $. Setting $B = \Z^2 {\setminus} \{x,y\}$ and $F(z) = \1 (z = y)$, we have $f(z) = \P_z (\sigma_y < \sigma_x)$. Since $M_{x,y}$ is bounded, harmonic on $B$, and agrees with $f$ on $\bdy B$, the uniqueness of $f$ implies $M_{x,y} (z) = f(z)$.
\end{proof}

The next two results partly implement the first estimate that we discussed in Section \ref{subsec: novel}.

\begin{lemma}\label{lem: p diff} For any $z,z' \in \Circ_{I-1}$ and $y \in D(R_{\ii})^c$,
\begin{equation}\label{eq: p diff}
\P_z (\tau_{y}<\tau_{o}) \le \frac{1}{2} \quad\text{and} \quad \left| \P_z (\tau_{y}<\tau_{o})-\P_{z'}(\tau_{y}<\tau_{o}) \right| \leq \frac{1}{\log R_{\ii}}.\end{equation}
\end{lemma}

The first inequality in \eqref{eq: p diff} holds because $z$ is appreciably closer to the origin than it is to $y$. The second inequality holds because a Taylor expansion of the numerator of $M_{z,y}(o) - M_{z',y} (o)$ shows that it is $O(1)$, while the denominator of $2\aa (y)$ is at least $\log R_{\ii}$.

\begin{proof}[Proof of Lemma \ref{lem: p diff}] By Lemma~\ref{lem: w and p},
\[ \P_z (\tau_y < \tau_{o}) = \frac12 + \frac{\aa (z) - \aa (y - z)}{2 \aa (y)}.\] 
The first inequality of \eqref{eq: p diff} holds because $\aa (y-z) \geq \aa (z)$. Indeed, $\Circ_{I-1}$ is a subset of $D(\delta R_{I})$, so $|z| \leq \delta R_{\ii} + 1$ and $|y-z| \geq (1-\delta) R_{\ii} -1$ by assumption. The latter is at least twice the former and $|z| \geq 2$, so by (1) of Lemma~\ref{lem: potkerbds}, $\aa(y-z) \geq \aa(z)$.

Using Lemma~\ref{lem: w and p}, the difference in \eqref{eq: p diff} can be written as
\begin{equation}\label{eq: pdiff1}
\left| M_{z,y} (o) - M_{z',y} (o) \right| = \frac{|\mathfrak{a} (y - z') - \mathfrak{a}(y - z)|}{2 \mathfrak{a} (y)}.
\end{equation} 
Concerning the denominator, $|y|$ is at least one, so $\aa (y)$ is at least $\tfrac{2}{\pi} \log |y| \geq \tfrac{2}{\pi} \log R_{I}$ by (2) of Lemma~\ref{lem: potkerbds}. We apply (3) of Lemma \ref{lem: potkerbds} with $R = R_{I}$ and $r = \rad (\Circ_{I-1}) \leq \delta R_{I}$ to bound the numerator by $\tfrac{4}{\pi}$. Substituting these bounds into \eqref{eq: pdiff1} gives the second inequality in \eqref{eq: p diff}.
\end{proof}

Label the $k$ elements in $A_{\geq R_1}$ by $x_i$ for $1 \leq i \leq k$. Then let $Y_i = \1 (\tau_{x_i} < \tau_o)$ and $W = \sum_{i=1}^{k} Y_i$. In words, $W$ counts the number of elements of $A_{\geq R_1}$ which have been visited before the random walk returns to the origin.

\begin{lemma}\label{lem: cond1} If $R_{\ii} \geq e^{4n}$, then, for all $z \in \Circ_{I-1}$,
\begin{equation}\label{eq: cond1}
\E_z [W\mid W>0]\ge \E_z W+\frac{1}{4}.
\end{equation}
\end{lemma}

The constant $\tfrac14$ in \eqref{eq: cond1} is unimportant, aside from being positive, independently of $n$. The inequality holds because random walk from $\Circ_{I-1}$ hits a given element of $A_{\geq R_1}$ before the origin with a probability of at most $\frac12$. Consequently, given that some such element is hit, the conditional expectation of $W$ is essentially larger than its unconditional one by a constant.

\begin{proof}[Proof of Lemma \ref{lem: cond1}]
Fix $z \in \Circ_{I-1}$. When $\{W>0\}$ occurs, some labeled element, $x_f$, is hit first. After $\tau_{x_f}$, the random walk may proceed to hit other $x_i$ before returning to $\Circ_{I-1}$ at a time $\eta = \min \left\{ t \geq \tau_{x_f}: S_t \in \Circ_{I-1} \right\}.$ Let $\mathcal{V}$ be the collection of labeled elements that the walk visits before time $\eta$, $\{i: \tau_{x_i} < \eta \}$. In terms of $\mathcal{V}$ and $\eta$, the conditional expectation of $W$ is
 \begin{equation}\label{eq: w cond w}
 \E_z [W\mid W>0]=\E_z \Big[ |\mathcal{V} |  +
 \E_{S_\eta} \sum_{i\notin \mathcal{V}} Y_i  \Bigm\vert W > 0\Big].
\end{equation}
Let $V$ be a nonempty subset of the labeled elements and let $z' \in \Circ_{I-1}$. We have
\[ \Big| \,\E_z \sum_{i \notin V} Y_i - \E_{z'} \sum_{i \notin V} Y_i \,\Big| \leq \frac{n}{\log R_{\ii}} \leq \frac14.\] 
The first inequality is due to Lemma \ref{lem: p diff} and the fact that there are at most $n$ labeled elements outside of $V$. The second inequality follows from the assumption that $R_{\ii} \geq e^{4n}$.

We use this bound to replace $S_\eta$ in \eqref{eq: w cond w} with $z$:
\begin{equation}\label{eq: exp lower}
\E_z [W \bigm\vert W > 0] \geq \E_z \Big[ |\mathcal{V}| + \E_z \sum_{i\notin \mathcal{V}} Y_i \Bigm\vert W > 0 \Big] - \frac14.
\end{equation}
 
By Lemma~\ref{lem: p diff}, $\P_z (\tau_{x_i} < \tau_{o}) \leq \tfrac12$. Accordingly, for a nonempty subset $V$ of labeled elements, \[ \E_z  \sum_{i \notin V} Y_i   \geq \E_z W - \frac12 |V|.\] Substituting this into the inner expectation of \eqref{eq: exp lower}, we find
\begin{align*}
\E_z [W \bigm\vert W > 0] &\geq \E_z \Big[ |\mathcal{V}| + \E_z W - \frac12 |\mathcal{V}| \Bigm\vert W > 0 \Big] - \frac14\\
& \geq \E_z W + \E_z \left[ \frac12 |\mathcal{V}| \Bigm\vert W > 0 \right] - \frac14.
\end{align*} Since $\{W > 0\} = \{ |\mathcal{V}| \geq 1\}$, this lower bound is at least $\E_z W + \frac14$.
\end{proof}

We use the preceding lemma to prove the analogue of Lemma \ref{lem: hit is har} when $R_{I} \geq e^{4n}$. The proof uses the method highlighted in Section \ref{subsec: novel} and Figure \ref{fig: diffs_fig_0} (left).
\begin{lemma}\label{lem: bayes11} There exists a constant $c$ such that, if $R_{\ii} \geq e^{4n}$, then 
\begin{equation}\label{eq: bayes11}
\P \left( \tau_{C(R_1)} < \tau_A \bigm\vert \tau_{\Circ_{I-1}} < \tau_A \right) \geq \frac{c}{n}.
\end{equation}
\end{lemma}

\begin{proof}
Conditionally on $\{\tau_{\Circ_{I-1}} < \tau_A\}$, let the random walk hit $\Circ_{I-1}$ at $z$. Denote the positions of the $k \leq n$ particles in $A_{\geq R_1}$ as $x_i$ for $1 \leq i \leq k$. Let $Y_{i}=\1 (\tau_{x_i} < \tau_{o})$ and $W=\sum_{i = 1}^{k} Y_i$, just as we did for Lemma \ref{lem: cond1}. 
The claimed bound \eqref{eq: bayes11} follows from
\[ \P_z (\tau_{C(R_1)} < \tau_A) \geq \P_z (W>0)=\frac{\E_z W}{\E_z [W\mid W>0]} \leq \frac{\E_z W}{\E_z W + 1/4} \leq \frac{n}{n+1/4} \leq 1 - \frac{1}{5n}.\]
The first inequality follows from the fact that $C(R_1)$ separates $z$ from the origin. The second inequality is due to Lemma \ref{lem: cond1}, which applies because $R_{\ii} \geq e^{4n}$. Since the resulting expression increases with $\E_z W$, we obtain the third inequality by substituting $n$ for $\E_z W$, as $\E_z W \leq n$. The fourth inequality follows from $n \geq 1$.
\end{proof}

\subsubsection{Inputs to Stage 4 when $K = I$ and Stage 3 when $K\neq I$}\label{stage 4} 
The results in this subsection address the last stage of advancement in the two sub-cases of the case $n_0 \neq n$: $K = I$ and $K \neq I$. In the former sub-case, the random walk has reached $C(R_1)$; in the latter sub-case, it has reached $\Circ_{K-1}$. Both sub-cases will be addressed by corollaries of the following geometric lemma.

Let $\Z^{2\ast}$ be the graph with vertex set $\Z^2$ and with an edge between distinct $x$ and $y$ in $\Z^2$ when $x$ and $y$ differ by at most one in each coordinate. For $B \subseteq \Z^2$, we will define the $\ast$-exterior boundary of $B$ by:
\begin{align}\label{ast bd}
\bdye B = \{x \in \Z^2: & \text{\,$x$ is adjacent in $\Z^{2\ast}$ to some $y \in B$,} \nonumber\\ 
& \quad \qquad \text{and there is a path from $\infty$ to $x$ disjoint from $B$}\}.
\end{align}

\begin{lemma}\label{lem: n path}
Let $A \in \mathscr{H}_n$ and $r > 0$. From any $x \in C(r){\setminus}A$, there is a path $\Gamma$ in $(A{\setminus}\{o\})^c$ from $\Gamma_1 = x$ to $\Gamma_{|\Gamma|} = o$ with a length of at most $10 \max \{r,n\}$. Moreover, if $A \subseteq D(r)$, then $\Gamma$ lies in $D(r+2)$.
\end{lemma}

We choose the constant factor of $10$ for convenience; it has no special significance. We use a radius of $r+2$ in $D(r+2)$ to contain the boundary of $D(r)$ in $\Z^{2\ast}$.

\begin{proof}[Proof of Lemma \ref{lem: n path}]
Let $\{B_\ell\}_\ell$ be the collection of $\ast$-connected components of $A{\setminus}\{o\}$. By Lemma~2.23 of \cite{kesten1986} (alternatively, Theorem~4 of \cite{timar2013}), because $B_\ell$ is finite and $\ast$-connected, $\bdye B_\ell$ is connected.

Fix $r > 0$ and $x \in C(r){\setminus}A$. Let $\Gamma$ be the shortest path from $x$ to the origin. If $\Gamma$ is disjoint from $A {\setminus} \{o\}$, then we are done, as $|\Gamma|$ is no greater than $2r$. Otherwise, let $\ell_1$ be the label of the first $\ast$-connected component intersected by $\Gamma$. Let $i$ and $j$ be the first and last indices such that $\Gamma$ intersects $\bdye B_{\ell_1}$, respectively. Because $\bdye B_{\ell_1}$ is connected, there is a path $\Lambda$ in $\bdye B_{\ell_1}$ from $\Gamma_i$ to $\Gamma_j$. We then edit $\Gamma$ to form $\Gamma'$ as 
\[ \Gamma' = \left( \Gamma_1, \dots, \Gamma_{i-1}, \Lambda_1, \dots, \Lambda_{|\Lambda|}, \Gamma_{j+1},\dots, \Gamma_{|\Gamma|} \right).\]

If $\Gamma'$ is disjoint from $A {\setminus} \{o\}$, then we are done, as $\Gamma'$ is contained in the union of $\Gamma$ and $\bigcup_\ell \bdye B_\ell$. Since $\bigcup_\ell B_\ell$ has at most $n$ elements, $\bigcup_\ell \bdye B_\ell$ has at most $8n$ elements. Accordingly, the length of $\Gamma'$ is at most $2r + 8n \leq 10 \max\{r,n\}$. Otherwise, if $\Gamma'$ intersects another $\ast$-connected component of $A {\setminus} \{o\}$, we can simply relabel the preceding argument to continue inductively and obtain the same bound. 

Lastly, if $A \subseteq D(r)$, then $\bigcup_\ell \bdye B_\ell$ is contained in $D(r+2)$. Since $\Gamma$ is also contained in $D(r+2)$, this implies that $\Gamma'$ is contained in $D(r+2)$.
\end{proof}

We now state three corollaries of Lemma \ref{lem: n path}. The first corollary addresses {\em Stage 4} when $K = I$. It follows from $|A_{<R_1}| = O(R_1^2)$ and $A_{\geq R_1} \subseteq D(R_1+2)^c$.
\begin{corollary}\label{lem: to or1}
There is a constant $c$ such that 
\begin{equation}\label{eq: to or1}
\P ( \tau_o \leq \tau_A \mid \tau_{C(R_1)} < \tau_A ) \geq c.
\end{equation}
\end{corollary}

The second corollary addresses {\em Stage 3} when $K \neq I$.

\begin{corollary}
Assume that $n_0 =1$ and $K \neq I$. There is a constant $c$ such that
\begin{equation}\label{eq: to or2}
\P ( \tau_o \leq \tau_A \mid \tau_{\Circ_{K-1}} < \tau_A ) \geq c^{\sum_{\ell=I}^{K-1} n_\ell}.
\end{equation}
\end{corollary}

The bound \eqref{eq: to or2} follows from Lemma \ref{lem: n path} because $K \neq I$ implies that the radius $r$ of $\Circ_{K-1}$ is at most a constant factor times $|A_{< r}|$. Lemma \ref{lem: n path} then implies that there is a path $\Gamma$ from $\Circ_{K-1}$ to the origin with a length of $O (|A_{<r}|)$, which remains in $D(r+2)$ and otherwise avoids the elements of $A_{< r}$. In fact, because $\Circ_{K-1}$ is a subset of $\Ann_{K-1}$, which contains no elements of $A$, by remaining in $D(r+2)$, $\Gamma$ avoids $A_{\geq r}$ as well. This implies \eqref{eq: to or2}.

The third corollary implies \eqref{eq: ext hm thm2} of Theorem~\ref{thm: hm} because any connected set belonging to $\mathscr{H}_n$ is contained in $D(n)$.

\begin{corollary}\label{cor: con h bd}
Let $n \geq 1$. There is a constant $c$ such that, for any connected $A \in \mathscr{H}_n$,
\begin{equation*}
\H_A (o) \geq e^{-cn}.
\end{equation*}
\end{corollary}

\subsection{Proof of Theorem~\ref{thm: hm}}\label{subsec: thm hm}
We only need to prove \eqref{eq: ext hm thm}, because Corollary \ref{cor: con h bd} establishes \eqref{eq: ext hm thm2}. The proof is by induction on $n$. Since \eqref{eq: ext hm thm} clearly holds for $n=1$ and $n=2$, we assume $n \geq 3$.

Let $A \in \mathscr{H}_n$. There are three cases: $n_0 = n$, $n_0 \neq n$ and $K = I$, and $n_0 \neq n$ and $K \neq I$. The first of these cases is easy: When $n_0 = n$, $A$ is contained in $D(R_1)$, so Corollary \ref{eq: to or1} implies that $\H_A (o)$ is at least a universal constant. Accordingly, in what follows, we assume that $n_0 \neq n$ and address the two sub-cases $K=I$ and $K \neq I$.

{\bf First sub-case}: $K = I$. If $K=I$, then we write
\[ \H_A (o) = \P (\tau_o \leq \tau_A) \geq \P ( \tau_{C(R_{J})} < \tau_{\Circ_{I-1}} < \tau_{C(R_1)} < \tau_o \leq \tau_A).\] 
Because $C(R_{J})$, $\Circ_{I-1}$, and $C(R_1)$ respectively separate $\Circ_{I-1}$, $C(R_1)$, and the origin from $\infty$, we can express the lower bound as the following product:
\begin{multline}\label{eq: h prod1}
\H_A (o) \geq \P (\tau_{C(R_{J})} < \tau_A) \times \P \big(\tau_{\Circ_{I-1}} < \tau_A \bigm\vert \tau_{C(R_{J})} < \tau_A \big)\\ \times \P \big( \tau_{C(R_1)} < \tau_A \bigm\vert \tau_{\Circ_{I-1}} < \tau_A \big) 
\times \P \big( \tau_o \leq \tau_A \bigm\vert \tau_{C(R_1)} < \tau_A \big).
\end{multline}

We address the four factors of \eqref{eq: h prod1} in turn. First, by the induction hypothesis, there is a constant $c_1$ such that
\[ \P (\tau_{C(R_{J})} < \tau_A) \geq e^{-c_1 k \log k},\]
where $k = n_{>J} + 1$. 
Second, by the strong Markov property applied to $\tau_{C(R_{\jj})}$ and Lemma~\ref{lem: near unif}, and then by Lemma \ref{lem: to l}, there are constants $c_2$ and $c_3$ such that
\begin{equation}\label{eq: mu int}
\P \big( \tau_{\Circ_{I-1}} < \tau_A \bigm\vert \tau_{C(R_{\jj})} < \tau_A \big) \geq c_2 \P_{\, \mu_{\jj}} \left( \tau_{\Arc_{I-1}} < \tau_A \right) \geq e^{-c_3 \sum_{\ell=I}^{J-1} n_\ell}.
\end{equation}
Third and fourth, by Lemma \ref{lem: hit is har} and Lemma \ref{lem: bayes11}, and by Corollary \ref{lem: to or1}, there are constants $c_4$ and $c_5$ such that
\[ \P \big( \tau_{C(R_1)} \leq \tau_A \bigm\vert \tau_{\Circ_{I-1}} < \tau_A \big) \geq (c_4n)^{-1} \quad \text{and} \quad \P \big( \tau_{o} \leq \tau_A \bigm\vert \tau_{C(R_1)} \leq \tau_A \big) \geq c_5.\]
Substituting the preceding bounds into \eqref{eq: h prod1} completes the induction step for this sub-case:
\begin{equation*}
\H_A (o) \geq e^{-c_1 k \log k -c_3 \sum_{\ell=I}^{J-1} n_\ell - \log (c_4 n) + \log c_5} \geq e^{-c_1 n \log n}.
\end{equation*}
The second inequality follows from $n-k = \sum_{\ell=I}^{J-1} n_\ell > 1$ and $\log n \geq 1$, and from potentially adjusting $c_1$ to satisfy $c_1 \geq 8 \max\{1,c_3,\log c_4, -\log c_5\}$. We are free to adjust $c_1$ in this way, since the other constants do not arise from the use of the induction hypothesis.

{\bf Second sub-case}: $K \neq I$. If $K \neq I$, then we write $\H_A (o) \geq \P (\tau_{C(R_{J})} < \tau_{\Circ_{K-1}} < \tau_o \leq \tau_A)$. 
Because $C(R_{J})$ and $\Circ_{K-1}$ separate $\Circ_{K-1}$ and the origin from $\infty$, we can express the lower bound as:
\begin{equation}\label{eq: h prod21}
\H_A (o) \geq \P (\tau_{C(R_{J})} < \tau_A) \times \P \big(\tau_{\Circ_{K-1}} < \tau_A \bigm\vert \tau_{C(R_{J})} < \tau_A \big) \times \P \big(\tau_o \leq \tau_A \bigm\vert \tau_{\Circ_{K-1}} < \tau_A \big).
\end{equation}
As in the first sub-case, the first factor is addressed by the induction hypothesis and the lower bound \eqref{eq: mu int} applies to the second factor of \eqref{eq: h prod21} with $K$ in the place of $I$. Concerning the third factor, corollary \ref{lem: to or1} implies that there is a constant $c_6$ such that
\[ \P \big(\tau_o \leq \tau_A \bigm\vert \tau_{\Circ_{K-1}} < \tau_A \big) \geq e^{-c_6 \sum_{\ell=I}^{K-1} n_\ell}.\]
Substituting the three bounds into \eqref{eq: h prod21} concludes the induction step in this sub-case:
\[ \H_A (o) \geq e^{-c_1 k \log k - c_3 \sum_{\ell=K}^{J-1} n_\ell - c_6 \sum_{\ell=I}^{K-1} n_\ell} \geq e^{-c_1 n \log n}.\]
The second inequality follows from potentially adjusting $c_1$ to satisfy $c_1 \geq 8 \max \{1,c_3,c_6\}$.

This completes the induction and establishes \eqref{eq: ext hm thm}.
\qed

\section{Escape probability estimates}

The purpose of this section is to prove Theorem \ref{thm: esc}. It suffices to prove the escape probability lower bound \eqref{eq: gen esc thm}, as \eqref{eq: esc thm} follows from \eqref{eq: gen esc thm} by the pigeonhole principle. Let $A$ be an $n$-element subset of $\Z^2$ with at least two elements. We assume w.l.o.g.\ that $o \in A$. Denote $b = \diam (A)$, and suppose $d \geq 2b$. We aim to show that there is a constant $c$ such that, if $d \geq 2b$, then, for every $x \in A$,
\[ 
\P_x (\tau_{\bdy A_d} < \tau_A) \geq \frac{c \H_A (x)}{n \log d}.
\]
In fact, by adjusting $c$, we can reduce to the case when $d \geq kb$ for $k = 200$ and when $b$ is at least a large universal constant, $b'$. 
We proceed to prove \eqref{eq: gen esc thm} when $d \geq 200 b$, for sufficiently large $b$. Since $C(kb)$ separates $A$ from $\bdy A_d$, we can write the escape probability as the product of two factors:
\begin{equation}\label{esc split}
\P_x (\tau_{\bdy A_d} < \tau_A) = \P_x (\tau_{C(kb)} < \tau_A) \, \P_x \big(\tau_{\bdy A_d} < \tau_A \bigm\vert \tau_{C(kb)} < \tau_A \big).
\end{equation}

Concerning the first factor of \eqref{esc split}, we have the following lemma.

\begin{lemma}\label{esc split1}
Let $x \in A$. Then
\begin{equation}\label{esc split1a}
\P_x ( \tau_{C(kb)} < \tau_A ) \geq \frac{\H_A (x)}{4\log (kb)}.
\end{equation} 
\end{lemma}

The factor of $\log (kb)$ arises from evaluating the potential kernel at elements of $C(kb)$; the factor of $4$ is unimportant. The proof is an application of the optional stopping theorem to the martingale $\aa (S_{j \wedge \tau_o})$.

\begin{proof}[Proof of Lemma \ref{esc split1}]
Let $x \in A$. By conditioning on the first step, we have
\begin{equation}\label{esc split1a1}
\P_x ( \tau_{C(kb)} < \tau_A ) = \frac14 \sum_{y \notin A, y \sim x} \P_y ( \tau_{C(kb)} < \tau_A ),
\end{equation}
where $y \sim x$ means $|x-y|=1$. We apply the optional stopping theorem to the martingale $\aa (S_{j \wedge \tau_o})$ with the stopping time $\tau_A \wedge \tau_{C(kb)}$ to find:
\begin{equation}\label{esc split1a2}
\frac14 \sum_{y \notin A, y \sim x} \P_y ( \tau_{C(kb)} < \tau_A )
= \frac14 \sum_{y \notin A, y \sim x} \frac{\aa (y) - \E_y \aa (S_{\tau_A})}{\E_y \big[ \aa (S_{\tau_{C(kb)}}) - \aa (S_{\tau_A}) \bigm\vert \tau_{C(kb)} < \tau_A \big]}.
\end{equation}
We need two facts. First, $\H_A (x)$ can be expressed as $\frac14 \sum_{y \notin A, y \sim x} \big( \aa (y) - \E_y \aa (S_{\tau_A}) \big)$ \cite[Definition 3.15, Theorem 3.16]{popov2021two}. Second, for any $z \in C(kb)$, $\aa (z) \leq 4 \log (kb)$ by Lemma \ref{lem: potkerbds}. Applying these facts to \eqref{esc split1a2}, and the result to \eqref{esc split1a1}, we find
\[ 
\P_x ( \tau_{C(kb)} < \tau_A ) \geq \frac{1}{4 \log (kb)} \cdot \frac14 \sum_{y \notin A, y \sim x} \big( \aa (y) - \E_y \aa (S_{\tau_A}) \big) = \frac{\H_A (x)}{4\log (kb)}.
\]
\end{proof}

Concerning the second factor of \eqref{esc split}, given that $\{\tau_{C(kb)} < \tau_A\}$ occurs, we are essentially in the setting depicted on the right side of Figure~\ref{fig: diffs_fig_0}, with $x = S_{\tau_{C(kb)}}$, $r = b$, $kb$ in the place of $2r$, and $R = d$. The argument highlighted in Section~\ref{subsec: novel} suggests that the second factor of \eqref{esc split} is at least proportional to $\frac{\log b}{n \log d}$. We will prove this lower bound and combine it with \eqref{esc split} and \eqref{esc split1a} to obtain \eqref{eq: gen esc thm} of Theorem~\ref{thm: esc}.

\begin{lemma}\label{esc split2}
Let $y \in C(kb)$. If $d \geq kb$ and if $b$ is sufficiently large, then
\begin{equation}\label{eq: step3 bd} 
\P_y (\tau_{\bdy A_d} < \tau_A) \geq \frac{ \log b}{2n\log d}.
\end{equation} 
\end{lemma}

\begin{proof}
Let $y \in C(kb)$. We will follow the argument of Section \ref{subsec: novel}. Label the points of $A$ as $x_1, x_2, \dots, x_n$ and define
\[
Y_{i} = \1\left(\tau_{x_i} < \tau_{\bdy A_d}\right) \quad \text{and} \quad W =\sum_{i=1}^{n} Y_i.
\] From the definition of $W$, we see that $\{W = 0\} = \{\tau_{\bdy A_d} < \tau_A \}$. Thus to obtain the lower bound in \eqref{eq: step3 bd}, it suffices to get a complementary upper bound on 
\begin{equation}\label{eq: w rat} \P_y (W>0)=\frac{\E_y W}{\E_y [W\mid W>0]}.\end{equation}

We will find $\alpha$ and $\beta$ such that, uniformly for $y \in C(kb)$ and $x_i, x_j \in A$,
\begin{equation}\label{eq: a and b}
\P_y \left( \tau_{x_i} < \tau_{\bdy A_d} \right) \leq \alpha \quad \text{and} \quad \P_{x_i} \left( \tau_{x_j} < \tau_{\bdy A_d} \right) \geq \beta.
\end{equation} Moreover, $\alpha$ and $\beta$ will satisfy
\begin{equation}\label{eq: a and b2}
\alpha \leq \beta \quad \text{and} \quad 1 - \beta \geq \frac{\log b}{2\log d}.
\end{equation}
The requirement that $\alpha \leq \beta$ prevents us from choosing $\beta = 0$. Essentially, we will be able to satisfy \eqref{eq: a and b} and the first condition of \eqref{eq: a and b2} because $|x_i - x_j|$ is smaller than $|y-x_i|$. We will be able to satisfy the second condition because $\dist(x_i,\bdy A_d) \geq d$ while $|x_i-x_j| \leq b$, which implies that $\P_{x_i} (\tau_{x_j} < \tau_{\bdy A_d})$ is roughly $1 - \frac{\log b}{\log d}$.

If $\alpha, \beta$ satisfy \eqref{eq: a and b}, then we can bound \eqref{eq: w rat} as 
\begin{equation}\label{eq: cor up bd} \P_y (W > 0) \leq \frac{n\alpha}{1+(n-1)\beta}.\end{equation} Additionally, when $\alpha$ and $\beta$ satisfy \eqref{eq: a and b2}, \eqref{eq: cor up bd} implies
\begin{equation*}
\P_y (W = 0) \geq \frac{(1-\beta) + n (\beta - \alpha)}{(1-\beta) + n\beta} \geq \frac{1-\beta}{n} \geq \frac{\log b}{2n \log d},
\end{equation*} which gives the claimed bound \eqref{eq: step3 bd}.

\begin{figure}[ht]
\centering {\includegraphics[width=0.4\linewidth]{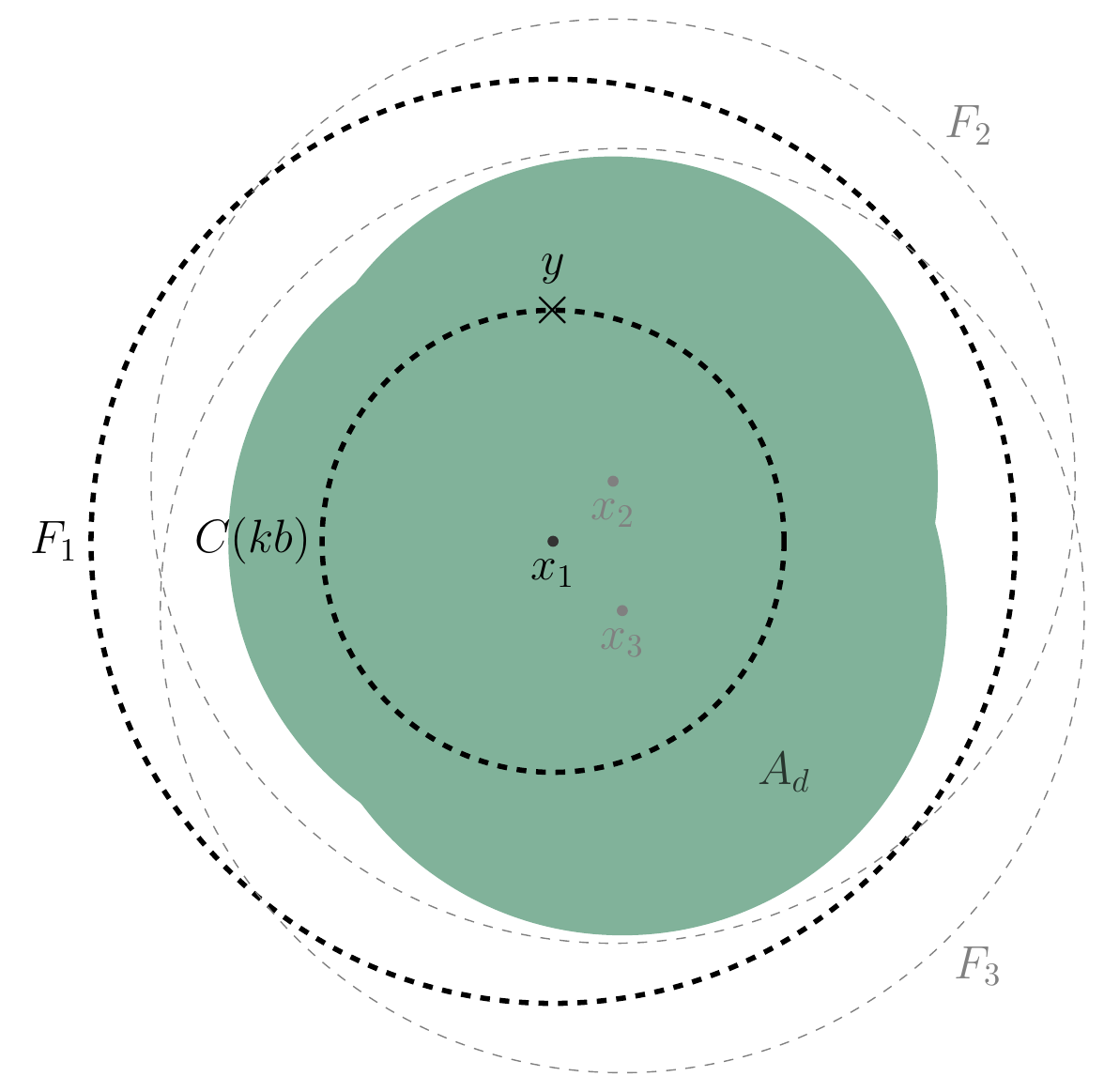}} 
\caption{Escape to $\bdy A_d$, for $n=3$. Each $F_i$ is a circle centered on $x_i \in A$, separating $A_d$ from infinity. Lemma~\ref{lem: gr est} bounds above the probability that the walk hits $x_i$ before $F_i$, uniformly for $y \in C(kb)$.}
\label{fig: efg2}
\end{figure}

{\em Identifying $\alpha$}. We now find the $\alpha$ promised in \eqref{eq: a and b}. Denote $F_i = C_{x_i} (d + b)$ (Figure~\ref{fig: efg2}). Since $\bdy A_d$ separates $y$ from $F_i$, we have
\begin{equation}\label{eq: yxi} \P_y \left( \tau_{x_i} < \tau_{\bdy A_d}\right) \leq \P_y \left( \tau_{x_i} < \tau_{F_i}\right) = \P_{y-x_i} \left( \tau_o < \tau_{C (d+b)} \right).
\end{equation}

The hypotheses of Lemma~\ref{lem: gr est} are met because $y - x_i \neq o$ and $y - x_i \in D(d+b)$. Hence \eqref{eq: hit or} applies as
\begin{equation}\label{eq: alpha1}
\P_{y-x_i} \left( \tau_{ o} < \tau_{C (d+b)} \right) = \frac{\mathfrak{a}' (d+b) - \mathfrak{a} (y-x_i) + O \left( |y-x_i|^{-1} \right)}{ \mathfrak{a}' (d+b) + O \left( |y-x_i|^{-1} \right)}.
\end{equation} 
Ignoring the error terms, the expression in \eqref{eq: alpha1} is at most $\frac{\log (d+b) - \log (kb)}{\log (d+b)}$. A more careful calculation gives
\begin{equation*}
\P_{y-x_i} \left( \tau_{ o} < \tau_{C (d+b)} \right) = \frac{\log (d+b) - \log (kb)}{\log (d+b)} + \delta_1 \leq \frac{(1+\eps) \log d - \log (kb)}{\log d} + \delta_1 =: \alpha,
\end{equation*}
where $\delta_1 = (\tfrac{\pi \kappa}{2} + O(b^{-1})) (\log d)^{-1}$ and $\eps = \frac{b}{d \log d}$. The inequality results from applying the inequality $\log (1+x) \leq x$, which holds for $x > -1$, to the $\log (d+b)$ term in the numerator, and reducing $\log (d+b)$ to $\log d$ in the denominator. By \eqref{eq: yxi}, $\alpha$ satisfies \eqref{eq: a and b}.

{\em Identifying $\beta$}. We now find a suitable $\beta$. Since $C_{x_i} (d)$ separates $A$ from $\bdy A_d$, we have
\begin{equation}\label{eq: beta1}
\P_{x_i} \left( \tau_{x_j} < \tau_{\bdy A_d} \right) \geq \P_{x_i} \big( \tau_{x_j} < \tau_{C_{x_i} (d)} \big) = \P_{x_i - x_j} \left( \tau_o < \tau_{C (d)} \right).
\end{equation} The hypotheses of Lemma~\ref{lem: gr est} are met because $x_i - x_j \neq o$ and $x_i - x_j \in D(d)$. Hence \eqref{eq: hit or} applies as
\begin{equation}\label{eq: beta2}
\P_{x_i - x_j} \left( \tau_{o} < \tau_{C(d)} \right) = \frac{ \mathfrak{a}' (d) - \mathfrak{a} (x_i - x_j) + O (|x_i - x_j|^{-1})}{\mathfrak{a}' (d) + O(|x_i - x_j|^{-1})}.
\end{equation}
Ignoring the error terms, \eqref{eq: beta2} is at least $\frac{\log d - \log b}{\log d + \kappa}$. A more careful calculation gives
\begin{equation*}
\P_{x_i - x_j} \left( \tau_{o} < \tau_{C(d)} \right) = \frac{\log d - \log b}{\log d} - \delta_2 =: \beta,
\end{equation*}
where $\delta_2 = (\tfrac{\pi \kappa}{2} + O(b^{-1}) ) (\log d)^{-1}$. By \eqref{eq: beta1}, $\beta$ satisfies \eqref{eq: a and b}.

{\em Verifying \eqref{eq: a and b2}}. To verify the first condition of \eqref{eq: a and b2}, we calculate
\begin{equation*}
(\beta - \alpha) \log d = \log k - \tfrac{b}{d} - \pi \kappa + O(b^{-1}) \geq 1 + O(b^{-1}).
\end{equation*}
The inequality is due to $k = 200$, $\tfrac{b}{d} \leq 0.5$, and $\pi \kappa < 3.5$. If $b$ is sufficiently large, then $1+O(b^{-1})$ is nonnegative, which verifies \eqref{eq: a and b2}.

Concerning the second condition of \eqref{eq: a and b2}, if $b$ is sufficiently large, then
\begin{equation*}
1 - \beta = \frac{\log b + 1}{\log d} \leq \frac{\log b}{2\log d}.
\end{equation*}
We have identified $\alpha, \beta$ which satisfy \eqref{eq: a and b} and \eqref{eq: a and b2} for sufficiently large $b$. By the preceding discussion, this proves \eqref{eq: step3 bd}.
\end{proof}

\begin{proof}[Proof of Theorem \ref{thm: esc}]
By \eqref{esc split}, Lemma \ref{esc split1}, and Lemma \ref{esc split2}, we have
\begin{equation}\label{almost esc}
\P_x (\tau_{\bdy A_d} < \tau_A) \geq \frac{\H_A (x)}{4\log (kb)} \cdot \frac{ \log b}{2n\log d} \geq \frac{\H_A (x)}{16 n \log d},
\end{equation}
whenever $x \in A$ and $d \geq kb$, for sufficiently large $b$. The second inequality is due to the fact that $\log (kb) \leq 2 \log b$ for sufficiently large $b$.

By the reductions discussed at the beginning of this section, \eqref{almost esc} implies that there is a constant $c$ such that \eqref{eq: gen esc thm} holds for $x \in A$ if $A$ has at least two elements and if $d \geq 2\, \diam (A)$. \eqref{eq: esc thm} follows from \eqref{eq: gen esc thm} because, by the pigeonhole principle, some element of $A$ has harmonic measure of at least $n^{-1}$.
\end{proof}


\section{Clustering sets of relatively large diameter}\label{sec: clust}

When a HAT configuration has a large diameter relative to the number of particles, we can decompose the configuration into clusters of particles, which are well separated in a sense. This is the content of Lemma~\ref{lem: clustering}, which will be a key input to the results in Section~\ref{sec: cons}.

\begin{definition}[Exponential clustering]
For a finite $A \subset \Z^2$ with $|A| = n$, an exponential clustering of $A$ with parameter $r \geq 0$, denoted $A \mapsto_r \{A^i, x_i, \theta^{(i)}\}_{i=1}^k$, is a partition of $A$ into clusters $A^1, A^2, \dots, A^k$ with $1 \leq k \leq n$, such that each cluster arises as $A^i = A \cap D_{x_i} (\theta^{(i)})$ for $x_i \in \Z^2$, with $\theta^{(i)} \geq r$, and \begin{equation}\label{eq: exp sep}
\dist (A^i, A^j) > \exp \big( \max\big\{\theta^{(i)}, \theta^{(j)} \big\} \big) \,\,\,\text{for $i \neq j$}.
\end{equation}
We will call $x_i$ the center of cluster $i$. In some instances, the values of $r$, $x_i$, or $\theta^{(i)}$ will be irrelevant and we will omit them from our notation. For example, $A \mapsto \{A^i\}_{i=1}^k$.
\end{definition}

An exponential clustering of $A$ with parameter $r$ always exists because, if $A^1 = A$, $x_1 \in A$, and $\theta^{(1)} \geq \max \{r,\diam (A)\}$, then $A \mapsto_r \{A^1,x_1,\theta^{(1)}\}$ is such a clustering. However, to ensure that there is an exponential clustering of $A$ (with parameter $r$) with more than one cluster, we require that the diameter of $A$ exceeds $2 \theta_{n-1} (r)$. Recall that we defined $\theta_m (r)$ in \eqref{eq: radii} through $\theta_0 (r) = 0$ and $\theta_m (r) = \theta_{m-1} (r) + e^{\theta_{m-1} (r)}$ for $m \geq 1$.

\begin{lemma}\label{lem: clustering} Let $|A| = n$. If $\diam (A) > 2\theta_{n-1} (r)$, then there exists an exponential clustering of $A$ with parameter $r$ into $k > 1$ clusters.
\end{lemma}

To prove the lemma, we will identify disks with radii of at most $\theta_{n-1} (r)$, which cover $A$. Although it is not required of an exponential clustering, the disks will be centered at elements of $A$. These disks will give rise to at least two clusters, since $\diam (A)$ exceeds $2\theta_{n-1} (r)$. The disks will be surrounded by large annuli which are empty of $A$, which will imply that the clusters are exponentially separated.

\begin{proof}[Proof of Lemma \ref{lem: clustering}]
For each $x \in A$ and $m \geq 1$, consider the annulus $\A_x (\theta_m) = D_x (\theta_m) {\setminus} D_x (\theta_{m-1})$. For each $x$, identify the smallest $m$ such that $\A_x (\theta_m) \cap A$ is empty and call it $m_x$. Note that since $|A|=n,$ $m_{x}$ can be no more than $n$ and hence $\theta_{m_x}\le \theta_n.$  Call the corresponding annulus $\A^\ast_x$, and denote $D_x^\ast = D_x (\theta_{m_x-1})$. For convenience, we label the elements of $A$ as $x_1, x_2, \dots, x_n$. 

For $x_i \in A$, we collect those disks $D_{x_j}^\ast$ which contain it as
\[ \mathcal{E} (x_{i}) = \big\{ D_{x_j}^\ast: x_i \in D_{x_j}^\ast,\,1\leq j \leq n\big\}.\]
We observe that $\mathcal{E} (x_i)$ is always nonempty, as it contains $D_{x_i}^\ast$. Now observe that, for any two distinct $D_{x_j}^\ast, D_{x_\ell}^\ast \in \mathcal{E}(x_i)$, 
it must be that
\begin{equation}\label{eq: contain}
D_{x_j}^\ast \cap A \subseteq D_{x_\ell}^\ast \cap A \quad \text{or} \quad D_{x_\ell}^\ast \cap A \subseteq D_{x_j}^\ast \cap A.
\end{equation}
To see why, assume for the purpose of deriving a contradiction that each disk contains an element of $A$ which the other does not. 
Without loss of generality, suppose $\theta_{m_{x_j}} \geq \theta_{m_{x_\ell}}$ and let $y_\ell \in (D_{x_\ell}^\ast {\setminus} D_{x_j}^\ast) \cap A$. Because each disk must contain $x_i$, we have $|y_\ell - x_i | \leq 2\theta_{m_{x_\ell} - 1}$ and $|x_i - x_j | \leq \theta_{m_{x_j} - 1}$. The triangle inequality implies  \[ |y_\ell - x_j| \leq \theta_{m_{x_j} - 1} + 2\theta_{m_{x_\ell} - 1} \leq \theta_{m_{x_j}} \implies y_\ell \in D_{x_j} (\theta_{m_{x_j}}) \cap A.\] By assumption, $y_\ell$ is not in $D_{x_j} (\theta_{m_{x_j} - 1}) \cap A$, so $y_\ell$ must be an element of $\A_{x_j} (\theta_{m_{x_j}}) \cap A$, which contradicts the construction of $m_{x_j}$.

By \eqref{eq: contain}, we may totally order the elements of $\mathcal{E} (x_i)$ by inclusion of intersection with $A$. For each $x_i$, we select the element of $\mathcal{E} (x_i)$ which is greatest in this ordering. If we have not already established it as a cluster, we do so. After we have identified a cluster for each $x_i$, we discard those $D_{x_j}^\ast$ which were not selected for any $x_i$. For the remainder of the proof, we only refer to those $D_{x_j}^\ast$ which were established as clusters, and we relabel the $x_i$ so that the clusters can be expressed as the collection $\big\{ D_{x_j}^\ast \big\}_{j=1}^k$, for some $1 \leq k \leq n$. We will show that $k$ is strictly greater than one.

The collection of clusters contains all elements of $A$, and is associated to the collection of annuli $\big\{\A_{x_j}^\ast \big\}_{j=1}^k$, which contain no elements of $A$. We observe that, for some distinct $x_j$ and $x_\ell$, it may be that $\A_{x_j}^\ast \cap D_{x_\ell}^\ast \neq \emptyset$. However, because the annuli contain no elements of $A$, it must be that
\begin{align*}
\dist(D_{x_j}^\ast \cap A, D_{x_\ell}^\ast \cap A) &> \max \left\{ \theta_{m_{x_j}} - \theta_{m_{x_{j-1}}}, \theta_{m_{x_\ell}} - \theta_{m_{x_{\ell-1}}} \right\}\\ &= \max \big\{ e^{\theta_{m_{x_j} - 1}}, e^{\theta_{m_{x_\ell} - 1}}\big\}\\
&= \exp \big( \max \big\{ \rad (D_{x_j}^\ast), \rad (D_{x_\ell}^\ast)\big\} \big),
\end{align*}
where we use $\rad$ to indicate the radius of a disk. As $D_{x_j}^\ast \cap A \subseteq D_{x_j}^\ast$ for any $x_j$ in question, we conclude the desired separation of clusters by setting $A^i = D_{x_i}^\ast \cap A$ for each $1 \leq i \leq k$.  Furthermore, since $m_{x_j} \leq n$ for all $j$, $\rad (D^*_{x_j})\le \theta_{n-1}$ for all $j$. Since $A$ is contained in the union of the clusters, if $\diam (A) > 2 \theta_{n-1}$, then there must be at least two clusters. Lastly, as $m_{x_j} \geq 0$ for all $j$, $\rad (D_{x_j}^\ast ) \geq r$ for all $j$.
\end{proof}




\section{Estimates of the time of collapse}\label{sec: cons}
We proceed to prove the main collapse result, Theorem \ref{thm: cons}. As the proof requires several steps, we begin by discussing the organization of the section and introducing some key definitions. We avoid discussing the proof strategy in detail before making necessary definitions; an in-depth proof strategy is covered in Section~\ref{sec: prop68strat}.

Briefly, to estimate the time until the diameter of the configuration falls below a given function of $n$, we will perform exponential clustering and consider the more manageable task of (i) estimating the time until some cluster loses all of its particles to the other clusters. By iterating this estimate, we can (ii) control the time it takes for the clusters to consolidate into a single cluster. We will find that the surviving cluster has a diameter which is approximately the logarithm of the original diameter. Then, by repeatedly applying this estimate, we can (iii) control the time it takes for the diameter of the configuration to collapse.

The purpose of Section~\ref{sec: pf cons} is to wield (ii) in the form of Proposition~\ref{prop: l clust col} and prove Theorem~\ref{thm: cons}, thus completing (iii). The remaining subsections are dedicated to proving the proposition. An overview of our strategy will be detailed in Section~\ref{sec: prop68strat}. In particular, we describe how the key harmonic measure estimate of Theorem~\ref{thm: hm} and the key escape probability estimate of Theorem~\ref{thm: esc} contribute to addressing (i). We then develop basic properties of cluster separation and explore the geometric consequences of timely cluster collapse in Section~\ref{sec: col prelims}. Lastly, in Section~\ref{sec: lclustcol pf}, we prove a series of propositions which collectively control the timing of individual cluster collapse, culminating in the proof of Proposition~\ref{prop: l clust col}.

Implicit in this discussion is a notion of ``cluster'' which persists over several steps of the dynamics.  We now make this precise in terms of an exponential clustering. Recall that an exponential clustering $U_0 \mapsto \{U_0^i, x_i, \theta^{(i)}\}_{i=1}^k$ of $U_0$ is defined such that: $\{U_0^i\}_{i=1}^k$ partitions $U_0$; each $U_0^i$ equals $U_0 \cap D_{x_i} (\theta^{(i)})$; and every distinct pair of clusters $U_0^i$, $U_0^j$ satisfies $\dist (U_0^i,U_0^j) > e^{\max \{\theta^{(i)},\theta^{(j)}\}}$.

\begin{definition}
Let $U_0$ have an exponential clustering $U_0 \mapsto \{U_0^i, x_i, \theta^{(i)}\}_{i=1}^k$. For any time $t \geq 1$, if $U_t$ is obtained from $t$ steps of the HAT dynamics from initial configuration $U_0$, then we recursively define $\{U_t^i\}_{i=1}^k$ as
\begin{equation}\label{eq: next clust}
U_t^i = U_t \cap \big( U_{t-1}^i \cup \bdy U_{t-1}^i \big).
\end{equation} 
\end{definition}
In principle, after many steps of the dynamics, clusters defined according to \eqref{eq: next clust} may intersect one another. However, in our application, clusters will be disjoint.

\begin{definition}[Cluster collapse times]
Suppose $U_0$ has the exponential clustering $U_0 \mapsto \{U_0^i\}_{i=1}^k$. We define the $\ell$-cluster collapse time as \[ \TT_\ell = \inf\left\{t \geq 0: U_t^{j_1} = U_t^{j_2} = \cdots = U_t^{j_\ell} = \emptyset,\,\,\text{for $1 \leq j_1 < j_2 < \cdots < j_\ell \leq k$}\right\}. \] 
We adopt the convention that $\TT_0 \equiv 0$. 
\end{definition}

By \eqref{eq: next clust}, if for some time $t$ the cluster $U_{t}^i$ is empty, then $U_{t'}^i$ is empty for all times $t' \geq t$. Consequently, the collapse times are ordered: $\TT_1 \leq \TT_2 \leq \cdots \leq \TT_\ell$.

\subsection{Proving Theorem~\ref{thm: cons}}\label{sec: pf cons}

We now state the proposition to which most of the effort in this section is devoted and, assuming it, prove Theorem~\ref{thm: cons}. We will denote by
\begin{itemize}
\item $n$, the number of elements of $U_0$;
\item $\Phi (r)$, the inverse function of $\theta_n (r)$ for all $r \geq 0$ ($\theta_n (r)$ is an increasing function of $r \geq 0$ for every $n$); and
\item $\UU_t$, the sigma algebra generated by the initial configuration $U_0$, the first $t$ activation sites $X_0, X_1, \dots, X_{t-1}$, and the first $t$ random walks $S^0, S^1, \dots, S^{t-1}$, which accomplish the transport component of the dynamics. 
\end{itemize} We note that $\Phi$ is defined so that, if $r = \Phi (\diam (U_0))$, then $\diam (U_0) > 2 \theta_{n-1} (r)$ and, by Lemma~\ref{lem: clustering}, exponential clustering of $U_0$ with parameter $r$ will produce at least two clusters.

\begin{proposition}\label{prop: l clust col} There is a constant $c$ such that, if the diameter $d$ of $U_0$ exceeds $\theta_{4n} (c n)$, then for any number of clusters $k$ resulting from exponential clustering of $U_0$ with parameter $r = \Phi (d)$ and with $\delta = (3n)^{-2}$, we have
\begin{equation}\label{eq: l clust col}
\PP_{U_0} \left( \TT_{k-1} \leq (\log d)^{1+7\delta} \right) \geq 1 - \exp \left( -2 n r^{\delta} \right).
\end{equation}
\end{proposition}

In words, if $U_0$ has a diameter of $d$, it takes no more than $(\log d)^{1+o_n (1)}$ steps to observe the collapse of all but one cluster, with high probability. Because no cluster begins with a diameter greater than $\log d$ (by exponential clustering) and, as the diameter of a cluster increases at most linearly in time, the remaining cluster at time $\TT_{k-1}$ has a diameter of no more than $(\log d)^{1+o_n (1)}$. We will obtain Theorem~\ref{thm: cons} by repeatedly applying Proposition~\ref{prop: l clust col}. We prove the theorem here, assuming the proposition, and then prove the proposition in the following subsections.

Our argument takes the form of Algorithm \ref{alg} and an analysis of its outputs. We organize the proof in this way because it more compact and direct than the alternative. In the context of a configuration with $k_\ell$ clusters, we will set $E_\ell =\big\{\TT_{k_\ell - 1} \leq (\log d_\ell)^{1+7\delta} \big\}$. The variable $\TT_{k_\ell-1}$ is the time it takes for the $k_\ell$ clusters to collapse into one cluster. The algorithm takes as input an initial configuration $U$ with number of elements $n$ and diameter $d$. It defines variables $V_\ell$, $d_\ell$, and $r_\ell$, which are the configuration, diameter, and clustering parameter after $\ell-1$ collapses. We set $V_1$ equal to $U$; $d_1$ equal to $d$; $r_1$ to be $\Phi (d)$; two counting variables, $\ell$ and $\TT$, equal to one and zero; and an indicator called $\mathsf{flag}$ to zero.

During the $\ell^{\text{th}}$ ``loop,'' the algorithm performs exponential clustering with parameter $r_\ell$ on configuration $V_\ell$ to obtain $k_\ell$ clusters and checks the occurrence of $E_\ell^c$. If $E_\ell^c$ occurs, the algorithm sets $\mathsf{flag}$ to one and ``breaks'' out of the current loop, upon which the algorithm terminates. If $E_\ell$ occurs, the algorithm assigns values for the configuration $V_{\ell+1}$, diameter $d_{\ell+1}$, and clustering parameter $r_{\ell + 1}$, which will be used in the next loop (if another loop is entered). Additionally, the algorithm updates $\TT$ to account for the $\TT_{k_\ell - 1}$ steps of the HAT dynamics and updates $\ell$ to $\ell + 1$ so that the next loop uses the new configuration, diameter, and clustering parameter.

The algorithm terminates if, at the beginning of the $\ell^\text{th}$ loop, the current HAT configuration $V_\ell$ has a diameter $d_\ell$ less than or equal to $\theta_{4n} (cn)$ or if, at any time, $\mathsf{flag} = 1$, indicating the occurrence of $E_{\ell-1}^c$. If the algorithm terminates with $\mathsf{flag} = 0$, then it must have terminated because $d_\ell \leq \theta_{4n} (cn)$ and therefore the value of $\TT$ returned by the algorithm is at least $\TT (\theta_{4n} (cn))$. If the algorithm terminates with $\mathsf{flag} = 1$, then we are unable to provide a bound on $\TT (\theta_{4n} (cn))$ in terms of $\TT$.

\begin{algorithm}[htbp]
\NoCaptionOfAlgo
\DontPrintSemicolon
\SetAlgoLined
\SetKwInOut{Input}{Input}
\SetKwInOut{Output}{Output}
    \Input{Configuration $U$, number of elements $n$, diameter $d = \diam (U)$}
    \Output{Indicator of failed collapse time estimate $\mathsf{flag}$, total collapse time $\TT$}
    \tcc{Assign initial values of parameters.}
 $V_1 \leftarrow U$, \quad $d_1 \leftarrow d$, \quad $r_1 \leftarrow \Phi (d)$,
 \quad $\ell \leftarrow 1$, \quad $\TT \leftarrow 0$, \quad and \quad $\mathsf{flag} \leftarrow 0$
    
\tcc{While the diameter is large and preceding collapse time estimates have succeeded \dots}
 \While{$d_\ell > \theta_{4n} (c n)$ and $\mathsf{flag} = 0$}{
  \tcc{Perform exponential clustering.}
  $V_\ell \mapsto_{r_\ell} \{U_0^i\}_{i=1}^{k_\ell}$\;
  \tcc{Try to observe the collapse of a cluster.}
  \eIf{$E_\ell^c$ occurs }{
   $\mathsf{flag} \leftarrow 1$ \tcp*{If collapse takes too long, indicate this with $\mathsf{flag}$ and terminate.}
   \textbf{break}\;
   }{
   $V_{\ell +1} \leftarrow U_{\TT_{k_\ell - 1}}, \quad d_{\ell+1} \leftarrow \diam (V_{\ell+1}), \quad r_{\ell+1} \leftarrow \Phi (d_{\ell+1})$ \tcp*{Else, prepare the next loop.}
  }
  $\TT \leftarrow \TT + \TT_{k_\ell - 1}, \quad \ell \leftarrow \ell + 1$ \tcp*{Restart the loop with the new configuration.}
 }
 \Return $\mathsf{flag}$, $\TT$
\caption{\textbf{Algorithm \ref{alg}}}
\label{alg}
\end{algorithm}

\begin{proof}[Proof of Theorem~\ref{thm: cons}]
In the context of the preceding discussion, it suffices to show that, with a probability of at least $1-e^{-n}$, the algorithm terminates with $\mathsf{flag} = 0$ and $\TT$ which satisfies
\begin{equation}\label{eq: tkn}
\TT \leq (\log d)^{1+o_n (1)}.
\end{equation}

By Proposition~\ref{prop: l clust col}, we have $\PP_{V_\ell} (E_\ell^c) \leq e^{-2 n r_\ell^{\delta}}$ for any $\ell$. Consequently, if $N$ is the number of loops (i.e., the number of times the \texttt{while} statement executes) before the algorithm terminates, then the procedure terminates with $\mathsf{flag} = 0$ unless $\cup_{\ell = 1}^N E_\ell^c$ occurs, which has a probability no greater than
\begin{equation}\label{eq: nen1}
\PP_{U} \left( \cup_{\ell=1}^N E_\ell^c \right) \leq \sum_{\ell=1}^N e^{-2n r_\ell^\delta} = e^{-2 n r_N^\delta} \sum_{\ell=1}^N e^{-2n (r_\ell^\delta - r_N^\delta)}.
\end{equation}

For all $\ell < N$, the event $E_\ell$ occurs which implies (by some algebra) that $d_{\ell+1}$ is less than $(\log d_\ell)^{1+8\delta}$. Using this bound and the fact that $d_\ell$ is at least $\theta_{4n} (cn)$, some simple but cumbersome algebra shows
\[ r_{\ell}^\delta - r_{\ell+1}^\delta = \Phi (d_\ell)^\delta - \Phi (d_{\ell+1})^\delta \geq 1.\] Using \eqref{eq: nen1}, this implies
\begin{equation*}
\PP_U \left( \cup_{\ell=1}^N E_\ell^c \right) \leq e^{-2n r_N^\delta} \sum_{\ell = 0}^{N-1} e^{-2n \ell} \leq 2e^{-2n r_N^\delta} \leq e^{-n}.
\end{equation*} This establishes that the algorithm terminates with $\mathsf{flag} = 0$ with a probability of at least $1-e^{-n}$. It remains to establish \eqref{eq: tkn} when $\cap_{\ell = 1}^N E_\ell$ occurs.

Again, because $d_{\ell+1}$ is less than $(\log d_\ell)^{1+8\delta}$ and by the lower bound on $d_\ell$, the ratio of $\log d_{\ell+1}$ to $\log d_\ell$ is at most $1/2$. In fact, it is much smaller, but this suffices to establish
\begin{equation*}
\TT  = \sum_{\ell = 1}^{N} \TT_{k_{\ell} -1} \leq \sum_{\ell = 1}^N (\log d_\ell)^{1+ 7\delta} \leq (\log d_1)^{1+7\delta} \sum_{\ell=0}^{N-1} 2^{-\ell} \leq (\log d_1)^{1+8\delta}.
\end{equation*} We conclude \eqref{eq: tkn}.
\end{proof}

For applications in Section \ref{sec: stat dist}, we extend Theorem~\ref{thm: cons} to a more general tail bound of $\TT (\theta_{4n})$.

\begin{corollary}[Corollary of Theorem \ref{thm: cons}]\label{cor: cons}
Let $U$ be an $n$-element subset of $\Z^2$ with a diameter of $d$. There exists a universal positive constant $c$ such that
\begin{equation}\label{eq: cor cons}
\PP_U \left( \TT (\theta_{4n} (cn)) > t (\log \max\{t,d\})^{1+o_n(1)} \right) \leq e^{-t}
\end{equation}
for all $t \geq 1$. For the sake of concreteness, this is true with $n^{-1}$ in the place of $o_n (1)$.
\end{corollary}

In the proof of the corollary, it will be convenient to have notation for the timescale of collapse after $j$ failed collapses, starting from a diameter of $d$. Because diameter increases at most linearly in time, if the initial configuration has a diameter of $d$ and collapse does not occur in the next $(\log d)^{1+o_n(1)}$ steps, then the diameter after this period of time is at most $d + (\log d)^{1+o_n(1)}$. In our next attempt to observe collapse, we would wait at most $\big( \log ( d + (\log d)^{1+o_n(1)}) \big)^{1+o_n(1)}$ steps. This discussion motivates the definition of the functions $g_j = g_j (d,\eps)$ by
\begin{equation*}
g_0 = (\log d)^{1+\eps} \quad \text{and} \quad g_j = \Big(\log \big(d+\sum_{i=0}^{j-1} g_i\big) \Big)^{1+\eps} \,\,\,\forall\, j \geq 1.
\end{equation*} We will use $t_j = t_j (d,\eps)$ to denote the cumulative time $\sum_{i=0}^j g_i$.

\begin{proof}[Proof of Corollary~\ref{cor: cons}]
Let $\eps = n^{-2}$ and use this as the $\eps$ parameter for the collapse timescales $g_j$ and cumulative times $t_j$. Additionally, denote $\theta = \theta_{4n} (cn)$ for the constant $c$ from Theorem~\ref{thm: cons} (this will also be the constant in the statement of the corollary). The bound \eqref{eq: cor cons} clearly holds when $d$ is at most $\theta$, so we assume $d \geq \theta$.

Because the diameter of $U$ is $d$ and as diameter grows at most linearly in time, conditionally on $F_j = \{\TT (\theta) > t_j\}$, the diameter of $U_{t_j}$ is at most $d+t_j$. Consequently, by the Markov property applied to time $t_j$, and by Theorem~\ref{thm: cons} (the diameter is at least $\theta$) and the fact that $n \geq 1$, the conditional probability $\PP_U (F_{j+1} | F_j)$ satisfies
\begin{equation}\label{eq: fifi1}
\PP_U (F_{j+1} | F_j) = \EE_U \left[ \PP_{U_{t_j}} (\TT (\theta) > g_{j+1}) \frac{\1_{F_j}}{\PP_U (F_j)} \right] \leq e^{-1} \,\,\,\text{for any $j \geq 0$.}
\end{equation}
In fact, Theorem \ref{thm: cons} implies that the inequality holds with $e^{-n}$ in the place of $e^{-1}$, but this will make no difference to us.

If the cumulative time $t_J$ is at most $t$ for an integer $J$, then there are at least $J$ consecutive collapse attempts which must fail in order for $\TT (\theta_{4n})$ to exceed $t$. Then for any such $J$, by \eqref{eq: fifi1}, 
\begin{equation}\label{eq: fifi}
\PP_U (\TT(\theta) > t) \leq \prod_{i=0}^{J-1} \PP_U (F_{i+1} | F_i) \leq e^{-J}.
\end{equation}
We now bound below $J$. The cumulative time $t_J$ is at most $t$, so the corresponding collapse timescale $g_J$ is at most $\big( \log (d+t) \big)^{1+\eps}$. Because $t_J$ is at most $(J+1) g_J$ and as $t_J$ is within $g_J$ of $t$, we have
\[ J \geq \frac{t_J - g_J}{g_J} \geq \frac{t - 2g_J}{g_J}.\]
Replacing $t$ with $8t \big(\log \max\{t,d\}\big)^{1+\eps} \leq t \big(\log \max\{t,d\} \big)^{1+n^{-1}}$ (the inequality holds because $d \geq \theta$) in the preceding display and simplifying, we find
\[ J \geq \frac{8 t \big(\log \max\{t,d\} \big)^{1+\eps} - 2 \big(\log (d+t) \big)^{1+\eps}}{\big(\log (d+t) \big)^{1+\eps}} \geq t.\]
Applying this to \eqref{eq: fifi} gives \eqref{eq: cor cons}.
\end{proof}

\subsection{Proof strategy for Proposition~\ref{prop: l clust col}}\label{sec: prop68strat}

We turn our attention to the proof of Proposition~\ref{prop: l clust col}, which finds a high-probability bound on the time it takes for all but one cluster to collapse. Heuristically, if there are only two clusters, separated by a distance $\rho_1$, then one of the clusters will lose all its particles to the other cluster in $\log \rho_1$ steps (up to factors depending on $n$), due to the harmonic measure and escape probability lower bounds of Theorems~\ref{thm: hm} and \ref{thm: esc}. This heuristic suggests that, among $k$ clusters, we should observe the collapse of {\em some} cluster on a timescale which depends on the smallest separation between any two of the $k$ clusters. Similarly, at the time the $\ell^{\text{th}}$ cluster collapses, if the least separation among the remaining clusters is $\rho_{\ell+1}$, then we expect to wait $\log \rho_{\ell+1}$ steps for the $(\ell+1)^{\text{st}}$ collapse.

If the timescale of collapse is small relative to the separation between clusters, the pairwise separation and diameters of clusters cannot appreciably change while collapse occurs. In particular, the separation between any two clusters will not significantly exceed the initial diameter $d$ of the configuration, which suggests an overall bound of order $(\log d)^{1+o_n(1)}$ steps for all but one cluster to collapse, where the $o_n (1)$ factor accounts for various $n$-dependent factors. This is the upper bound we establish. 

We now highlight some key aspects of the proof.

\subsubsection{Expiry time}
As described above, over the timescale typical of collapse, the diameters and separation of clusters will not change appreciably. Because these quantities determine the probability with which the least separated cluster loses a particle, we will be able to obtain estimates of this probability which hold uniformly from the time $\TT_{\ell-1}$ of the $(\ell-1)^{\text{st}}$ cluster collapse and until the next time $\TT_\ell$ that some cluster collapses, unless $\TT_\ell - \TT_{\ell-1}$ is atypically large. Indeed, if $\TT_{\ell} - \TT_{\ell-1}$ is as large as the separation $\rho_\ell$ of the least separated cluster at time $\TT_{\ell-1}$, then two clusters may intersect. We avoid this by defining a $\UU_{\TT_{\ell-1}}$-measurable {\em expiry time} $\tt_\ell$ (which will effectively be $(\log \rho_\ell)^2$) and restricting our estimates to the interval from $\TT_{\ell-1}$ to the minimum of $\TT_{\ell-1} + \tt_\ell$ and $\TT_\ell$. An expiry time of $(\log \rho_\ell)^2$ is short enough that the relative separation of clusters will not change significantly before it, but long enough so that some cluster will collapse before it with overwhelming probability. 

\subsubsection{Midway point}\label{subsubsec: midway}
From time $\TT_{\ell-1}$ to time $\TT_\ell$ or until expiry, we will track activated particles which reach a circle of radius $\tfrac12 \rho_\ell$ surrounding one of the least separated clusters, which we call the {\em watched} cluster. We will use this circle, called the {\em midway point}, to organize our argument with the following three estimates, which will hold uniformly over this interval of time (Figure~\ref{fig: prop_strat}).
\begin{enumerate}
\item Activated particles which reach the midway point deposit at the watched cluster with a probability of at most $0.51$.
\item With a probability of at least $(\log \rho_\ell)^{-1-o_n(1)}$, the activated particle reaches the midway point.
\item Conditionally on the activated particle reaching the midway point, the probability that it originated at the watched cluster is at least $(\log\log \rho_\ell )^{-1}$.
\end{enumerate} 

\begin{figure}[ht]
\centering {\epsfig{width=0.65\linewidth,file=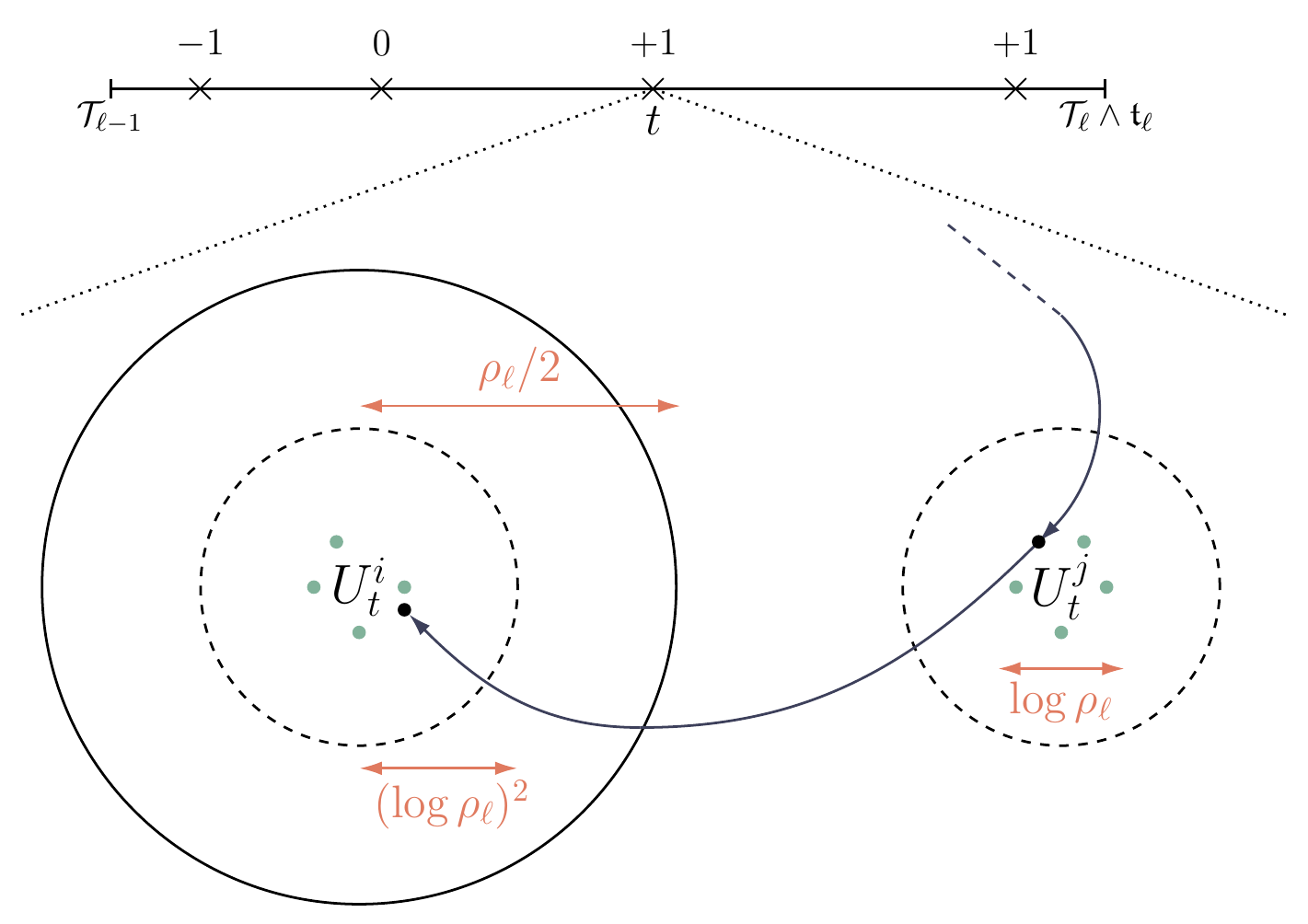}}
\caption{Setting of the proof of Proposition~\ref{prop: l clust col}. Least separated clusters $i$ and $j$ (cluster $i$ is the watched cluster), each with a diameter of approximately $\log \rho_\ell$, are separated by a distance $\rho_\ell$ at time $\TT_{\ell - 1}$. The diameters of the clusters grow at most linearly in time, so over approximately $(\log \rho_\ell)^2$ steps, the clusters remain within the dotted circles. Crosses on the timeline indicate times before collapse and expiry at which an activated particle reaches the midway point (solid circle). At these times, the number of particles in the watched cluster may remain the same or increase or decrease by one (indicated by $0, \pm 1$ above the crosses). At time $t$, the watched cluster gains a particle from cluster $j$.}
\label{fig: prop_strat}
\end{figure}

To explain the third estimate, we make two observations. First, consider a cluster $j$ separated from the watched cluster by a distance of $\rho$. In the relevant context, cluster $j$ will essentially be exponentially separated, so its diameter will be at most $\log \rho$. Consequently, a particle activated at cluster $j$ reaches the midway point with a probability of at most $\frac{\log \log \rho}{\log \rho}$. Because this probability is decreasing in $\rho$ and because $\rho \geq \rho_\ell$, $\frac{\log \log \rho_\ell}{\log \rho_\ell}$ further bounds above it. Second, the probability that a particle activated at the watched cluster reaches the midway point is at least $(\log \rho_\ell)^{-1}$, up to a factor depending on $n$. Combining these two observations with Bayes's rule, a particle which reaches the midway point was activated at the watched cluster with a probability of at least $( \log \log \rho_\ell)^{-1}$, up to an $n$-dependent factor.

\subsubsection{Coupling with random walk}\label{subsubsec: coupling}
Each time an activated particle reaches the midway point, there is a chance of at least $(\log \log \rho_\ell )^{-1}$ up to an $n$-dependent factor that the particle originated at the watched cluster and will ultimately deposit at another cluster. When this occurs, the watched cluster loses a particle. Alternatively, the activated particle may return to its cluster of origin---in which case the watched cluster retains its particles---or it deposits at the watched cluster, having originated at a different one---in which case the watched cluster gains a particle (Figure~\ref{fig: prop_strat}). 

We will couple the number of elements in the watched cluster with a lazy, one-dimensional random walk, which will never exceed $n$ and never hit zero before the size of the watched cluster does. It will take no more than $(\log \log \rho_\ell)^n$ instances of the activated particle reaching the midway point, for the random walk to make $n$ consecutive down-steps. This is a coarse estimate; with more effort, we could improve the $n$-dependence of this term, but it would not qualitatively change the result. On a high probability event, $\rho_\ell$ will be sufficiently large to ensure that $(\log \log \rho_\ell)^n = (\log \rho_\ell)^{o_n(1)}$. Then, because it will typically take no more than $(\log \rho_\ell)^{1+o_n(1)}$ steps to observe a visit to the midway point, we will wait a number of steps on the same order to observe the collapse of a cluster.

\subsection{Basic properties of clusters and collapse times}\label{sec: col prelims}
We will work in the following setting.
\begin{itemize}
\item For brevity, if we write $\theta_m$ with no parenthetical argument, we will mean $\theta_m (\gamma n)$ for the constant $\gamma$ given by 
\begin{equation}\label{eq: gamma}
\gamma = 18 \max\{c_1,c_2^{-1}\} + 36,
\end{equation} where $c_1$ and $c_2$ are the constants in Theorems~\ref{thm: hm} and \ref{thm: esc}. Any constant larger than $\gamma$ would also work in its place.
\item $U_0$ has $n \geq 2$ elements and $\diam (U_0)$ is at least $\theta_{4n}$.
\item The clustering parameter $r$ equals $\Phi (\diam (U_0))$, where we continue to denote by $\Phi (\cdot)$ the inverse function of $\theta_n (\cdot)$. In particular, $r$ satisfies
\begin{equation}\label{eq: rparbd}
r \geq \Phi (\theta_{4n}) = \theta_{3n} \geq e^n.
\end{equation}
\item We will assume that the initial configuration is exponentially clustered with parameter $r$ as $U_0 \mapsto_r \{U_0^i, x_i, \theta^{(i)}\}_{i=1}^k$. In particular, we assume that clustering produces $k$ clusters. We note that the choice of $r$ guarantees $\diam (U_0) > 2 \theta_{n-1} (r)$ which, by Lemma~\ref{lem: clustering}, guarantees that $k > 1$.
\item We denote a generic element of $\{1,2,\dots,k-1\}$ by $\ell$.
\end{itemize}

\subsubsection{Properties of cluster separation and diameter}

We will use the following terms to describe the separation of clusters.
\begin{definition}
We define pairwise cluster separation and the least separation by
\[\sep \, (U_t^i) = \min_{j \neq i} \dist (U_t^i, U_t^j) \quad \text{and} \quad \sep \,(U_t) = \min_i \sep \, (U_t^i).\]
(By convention, the distance to an empty set is $\infty$, so the separation of a cluster is $\infty$ at all times following its collapse.) 
If $U_t^i$ satisfies $\sep (U_t^i) = \sep (U_t)$, then we say that $U_t^i$ is least separated. Whenever there are at least two clusters, at least two clusters will be least separated. The least separation at a cluster collapse time will be an important quantity; we will denote it by
\[\rho_\ell = \sep (U_{\TT_{\ell - 1}}).\]  
\end{definition}

Next, we introduce the {\em expiry time} $\tt_\ell$ and the {\em truncated collapse time} $\TT_\ell^-$. As discussed in Section~\ref{sec: prop68strat}, if at time $\TT_{\ell-1}$ the least separation is $\rho_\ell$, then we will obtain a lower bound on the probability that a least separated cluster loses a particle, which holds uniformly from time $\TT_{\ell-1}$ to the first of $\TT_{\ell-1} + \tt_\ell$ and $\TT_\ell - 1$ (i.e., the time immediately preceding the $\ell$\textsuperscript{th} collapse), which we call the truncated collapse time, $\TT_\ell^-$. Here, $\tt_\ell$ is an $\UU_{\TT_{\ell-1}}$-measurable random variable which will effectively be $(\log \rho_\ell)^2$. It will be rare for $\TT_\ell$ to exceed $\TT_{\ell-1} + \tt_\ell$, so $\TT_\ell^-$ can be thought of as $\TT_\ell - 1$. 

\begin{definition}
Given the $\UU_{\TT_{\ell-1}}$ data (in particular $\rho_\ell$ and $\TT_{\ell-1}$), we define the expiry time $\tt_\ell$ to be
\[ \mathfrak{t}_\ell = (\log \rho_\ell)^2- 4 \log \left( \rho_\ell + \TT_{\ell - 1} \right) - \TT_{\ell - 1}. \] We emphasize that $\mathfrak{t}_\ell$ should be thought of as $(\log \rho_\ell)^2$; the other terms will be much smaller and are included to simplify calculations which follow. 
Additionally, we define the truncated $\ell$\textsuperscript{th} cluster collapse time to be \[ \TT^-_{\ell} = (\TT_{\ell-1} + \mathfrak{t}_\ell) \wedge (\TT_\ell - 1).\]
\end{definition}

Cluster diameter and separation have complementary behavior in the sense that diameter increases at most linearly in time but may decrease abruptly, while separation decreases at most linearly in time but may increase abruptly. We will not need a bound on decrease in diameter; we express the other properties in the following lemma.

\begin{lemma}\label{lem: expiry} Cluster diameter and separation obey the following properties.
	\begin{enumerate}
	\item Cluster diameter increases by at most one each step:
	\begin{equation}\label{eq: amlg}
	\diam (U_t^i) \leq \diam (U_{t-1}^i) + 1.
	\end{equation}
	\item Cluster separation decreases by at most one each step:
	\begin{equation}\label{eq: amlg2}
	\dist (U_t^i, U_t^j) \geq \dist (U_{t-1}^i, U_{t-1}^j) - 1 \quad\text{and}\quad \sep (U_t^i) \geq \sep (U_{t-1}^i) - 1.
	\end{equation}
 	\item For any two times $s$ and $t$ satisfying $\TT_{\ell - 1} \leq s < t < \TT_\ell$ and any two clusters $i$ and $j$:
	\begin{equation*}
	\dist (U_t^i,U_t^j) \leq \dist (U_s^i,U_s^j) + \diam (U_s^i) + \diam (U_s^j) + (t-s).
	\end{equation*} 
	\end{enumerate}
\end{lemma}

\begin{proof}
The first two properties are obvious; we prove the third. Let $i,j$ label two clusters which are nonempty at time $\TT_{\ell-1}$ and let $s,t$ satisfy the hypotheses. If there are $m_i$ activations at the $i$\textsuperscript{th} cluster from time $s$ to time $t$, then for any $x'$ in $U_t^i$, there is an $x$ in $U_s^i$ such that $|x - x'| \leq m_i$. The same is true of any $y'$ in the $j$\textsuperscript{th} cluster with $m_j$ in the place of $m_i$. Since the sum of $m_i$ and $m_j$ is at most $t-s$, two uses of the triangle inequality give
\[ \dist(U_t^i,U_t^j) \leq \max_{x' \in U_t^i,\,y' \in U_t^j} |x'-y'| \leq \max_{x \in U_s^i,\,y \in U_s^j} |x-y| + t-s.\]
This implies property (3) because, by two more uses of the triangle inequality,
\[ \max_{x \in U_s^i,\,y\in U_s^j} |x-y| \leq \dist(U_s^i, U_s^j) + \diam (U_s^i) + \diam (U_s^j).\]
\end{proof}

\subsubsection{Consequences of timely collapse}

If clusters collapse before their expiry times---i.e., if the event
\[\Quick (\ell) = \cap_{m=1}^{\ell} \{\TT_{m} - \TT_{m-1} \leq \tt_{m}\}\]
occurs---then we will be able to control the separation (Lemma \ref{lem: quick}) and diameters (Lemma \ref{lem: timelydiambd}) of the clusters by combining the initial exponential separation of the clusters with the properties of Lemma~\ref{lem: expiry}.

The next lemma states that, when cluster collapses are timely, cluster separation decreases little. To state it, we recall that $\sep (U_t^i)$ is the distance between $U_t^i$ and the nearest other cluster, and that $\rho_\ell$ is the least of these distances among all pairs of distinct clusters at time $\TT_{\ell-1}$. In particular, $\sep (U_{\TT_{\ell-1}}^i) \geq \rho_\ell$ for each $i$.

\begin{lemma}\label{lem: quick}
For any cluster $i$, when $\Quick (\ell-1)$ occurs and when $t$ is at most $\TT_\ell^-$,
\begin{equation}\label{eq: quick}
\sep \left( U_t^i \right) \geq (1-e^{-n})\, \sep \big( U_{\TT_{\ell - 1}}^i \big).
\end{equation}
Additionally, when $\Quick (\ell-1)$ occurs,
\begin{equation}\label{eq: quick2}
\rho_\ell \geq \tfrac{1}{2} \rho_1 \geq e^{\theta_{2n}}.
\end{equation}
\end{lemma}

The factor of $1-e^{-n}$ in \eqref{eq: quick} does not have special significance; other factors of $1-o_n (1)$ would work, too. \eqref{eq: quick} and the first inequality in \eqref{eq: quick2} are consequences of the fact \eqref{eq: amlg2} that separation decreases at most linearly in time and, when $\Quick (\ell-1)$ occurs, $\TT_{\ell-1}$ is small relative to the separation of the remaining clusters. The second inequality in \eqref{eq: quick2} follows from our choice of $r$ in \eqref{eq: rparbd}.

\begin{proof}[Proof Lemma \ref{lem: quick}]
We will prove \eqref{eq: quick} by induction, using the fact that separation decreases at most linearly in time \eqref{eq: amlg2} and that (by the definition of $\TT_\ell^-$) at most $\tt_\ell$ steps elapse between $\TT_{\ell-1}$ and $\TT_\ell^-$.

For the base case, take $\ell = 1$. Suppose cluster $i$ is nonempty at time $\TT_{\ell-1}$. We must show that, when $t \leq \TT_1^-$,
\begin{equation*}
\sep \left( U_t^i \right) \geq (1-e^{-n})\, \sep \left( U_0^i \right).
\end{equation*} Because separation decreases at most linearly in time \eqref{eq: amlg2} and because $t \leq \TT_1^-$,
\[ \sep (U_t^i) \geq \sep (U_0^i) - t \geq \sep (U_0^i) - \TT_1^-.\]
This implies \eqref{eq: quick} for $\ell=1$ because
\[ \sep (U_0^i) - \TT_1^- \geq \big( 1 - \tfrac{(\log \rho_1)^2}{\rho_1}\big) \, \sep (U_0^i)  \geq (1-e^{2n} e^{-e^n}) \, \sep (U_0^i) \geq (1-e^{-n}) \, \sep (U_0^i) .\]
The first inequality is a consequence of the definitions of $\TT_1^-$, $\mathfrak{t}_1$, and $\rho_1$, which imply $\TT_1^- \leq \mathfrak{t}_1 \leq (\log \rho_1)^2$ and $\sep (U_0^i) \geq \rho_1$. Since the ratio of $(\log \rho_1)^2$ to $\rho_1$ decreases as $\rho_1$ increases, the second inequality follows from the bound $\rho_1 \geq e^{e^n}$, which is implied by the fact that $U_0$ satisfies the exponential separation property \eqref{eq: exp sep} with parameter $r \geq e^n$ \eqref{eq: rparbd}. The third inequality is due to the fact that $e^n \geq 3n$ when $n \geq 2$.

The argument for $\ell > 1$ is similar. Assume \eqref{eq: quick} holds for $\ell-1$. We have
\begin{equation}\label{eq: sepbd1}
\sep (U_t^i) \geq \sep (U_{\TT_{\ell-1}}^i) - (t - \TT_{\ell-1}) \geq \sep (U_{\TT_{\ell - 1}}^i) - \tt_\ell \geq \big( 1 - \tfrac{(\log \rho_\ell)^2}{\rho_\ell} \big) \, \sep (U_{\TT_{\ell-1}}^i).
\end{equation}
The first inequality is implied by \eqref{eq: amlg2}. The second inequality follows from the definitions of $\TT_\ell^-$ and $\tt_\ell$, which imply $\TT_\ell^- - \TT_{\ell-1} \leq \tt_\ell \leq (\log \rho_\ell)^2$, and $t \leq \TT_\ell^-$. The third inequality is due to the same upper bound on $\tt_\ell$ and the fact that $\sep (U_{\TT_{\ell-1}}^i) \geq \rho_\ell$ by definition. 

We will bound below $\rho_\ell$ to complete the induction step with \eqref{eq: sepbd1}, because the ratio of $(\log \rho_\ell)^2$ to $\rho_\ell$ decreases as $\rho_\ell$ increases. Specifically, we will prove \eqref{eq: quick2}. By definition, when $\Quick (\ell-1)$ occurs, so too does $\Quick (\ell-2)$. Accordingly, the induction hypothesis applies and we apply it $\ell-1$ times:
\[ \rho_{\ell-1} = \min_i \sep (U_{\TT_{\ell-2}}^i) \geq (1-e^{-n})^{\ell-1} \min_i \sep (U_0^i) = (1-e^{-n})^{\ell-1} \rho_1.
\]
The equalities follow from the definitions of $\rho_{\ell-1}$ and $\rho_1$. We also have
\[ \rho_\ell \geq \rho_{\ell-1} - \tt_{\ell-1} \geq \big(1 - \tfrac{(\log \rho_{\ell-1})^2}{\rho_{\ell-1}} \big) \rho_{\ell-1} \geq (1-e^{-n}) \rho_{\ell-1}.\]
The first inequality is due to \eqref{eq: amlg2} and the fact that at most $\tt_{\ell-1}$ steps elapse between $\TT_{\ell-2}$ and $\TT_{\ell-1}$ when $\Quick (\ell-1)$ occurs. The second inequality is due to $\tt_{\ell-1} \leq (\log \rho_{\ell-1})^2$ and the third is due to the fact that the ratio of $(\log \rho_{\ell-1})^2$ to $\rho_{\ell-1}$ decreases as $\rho_{\ell-1}$ increases.

Combining the two preceding displays and then using the fact that $\ell \leq n$ and $\rho_1 \geq e^{e^n}$, and the inequality $(1+x)^r \geq 1 + rx$, which holds for $x > -1$ and $r > 1$, we find
\[ \rho_\ell \geq (1-e^{-n})^\ell \rho_1 \geq (1-ne^{-n}) \rho_1.\]
Because $ne^{-n} \leq \tfrac12$ when $n \geq 2$, this proves $\rho_\ell \geq \tfrac12 \rho_1$, which is the first inequality of \eqref{eq: quick2}. To prove the second inequality in \eqref{eq: quick2}, we note that $\rho_1$ is at least $\theta_{3n}$ by \eqref{eq: rparbd}.

We now apply $\rho_\ell \geq \tfrac12 \rho_1$ to the ratio in \eqref{eq: sepbd1}:
\[ \frac{(\log \rho_\ell)^2}{\rho_\ell} \leq \frac{2(\log \rho_1)^2}{\rho_1} \leq e^{-n}.\]
The second inequality uses $\rho_1 \geq e^{e^n}$. We complete the induction step, proving \eqref{eq: quick}, by substituting this bound into \eqref{eq: sepbd1}.
\end{proof}

When cluster collapses are timely, $\TT_\ell^-$ is at most $(\log \rho_\ell)^2$, up to a factor depending on $n$.

\begin{lemma}\label{lem: tell} When $\Quick (\ell-1)$ occurs,
\begin{equation}\label{eq: tell}
\TT_\ell^- \leq 2n (\log \rho_\ell)^2.
\end{equation}
\end{lemma}

The factor of $2$ is for brevity; it could be replaced by $1+o_n (1)$. The lower bound on the least separation $\rho_\ell$ at time $\TT_{\ell-1}$ in \eqref{eq: quick2} indicates that, while $\rho_\ell$ may be much larger than $\rho_1$, it is at least half of $\rho_1$. Since the expiry time $\tt_\ell$ is approximately $(\log \rho_\ell)^2$, the truncated collapse time $\TT_\ell^-$---which is at most the sum of the first $\ell$ expiry times---should be of the same order, up to a factor depending on $\ell$ (which we will replace with $n$ since $\ell \leq n$).

\begin{proof}[Proof of Lemma \ref{lem: tell}]
We write
\[ \TT_\ell^- = \TT_\ell^- - \TT_{\ell -1} + \sum_{m=1}^{\ell-1} (\TT_m - \TT_{m-1}) \leq \sum_{m=1}^\ell \tt_m \leq \sum_{m=1}^{\ell} (\log \rho_m)^2.\] The first inequality follows from the fact that, when $\Quick (\ell-1)$ occurs, $\TT_m - \TT_{m-1} \leq \tt_m$ for $m \leq \ell-1$, and $\TT_\ell^- - \TT_{\ell-1} \leq \tt_\ell$. The second inequality holds because $\tt_m \leq (\log \rho_m)^2$ by definition.

Next, assume w.l.o.g.\ that cluster $i$ is least separated at time $\TT_{\ell-1}$, meaning $\rho_\ell = \sep (U_{\TT_{\ell-1}}^i)$. Since $\Quick (\ell-1)$ occurs, Lemma~\ref{lem: quick} applies and with its repeated use we establish \eqref{eq: tell}:
\[ \sum_{m=1}^{\ell} (\log \rho_m)^2 \leq \sum_{m=1}^{\ell} \big( \log \sep (U_{\TT_{m-1}}^i) \big)^2 \leq \sum_{m=1}^\ell \Big( \log \big( (1+\tfrac{e^{-n}}{1-e^{-n}})^{\ell-m} \rho_\ell \big) \Big)^2 \leq \ell (\log (2 \rho_\ell) )^2 \leq 2 n (\log \rho_\ell)^2.\]
The first inequality is due to the definition of $\rho_m$ as the least separation at time $\TT_{m-1}$. This step is helpful because it replaces each summand with one concerning the $i$\textsuperscript{th} cluster. The second inequality holds because, by Lemma \ref{lem: quick},
\[ \rho_\ell = \sep (U_{\TT_{\ell-1}}^i) \geq (1-e^{-n})^{\ell-m} \sep (U_{\TT_{m-1}}^i) \implies \sep (U_{\TT_{m-1}}^i) \leq \big( 1 + \tfrac{e^{-n}}{1-e^{-n}} \big)^{\ell-m} \rho_\ell.\]
The third inequality follows from $\ell \leq n$ and $(1+\tfrac{e^{-n}}{1-e^{-n}})^n \leq 2$ when $n \geq 2$. The fourth inequality is due to $\ell \leq n$ and $\rho_\ell \geq e^{\theta_{2n}}$ from \eqref{eq: quick2}. (The factor of $2$ could be replaced by $1+o_n (1)$.) Combining the displays proves \eqref{eq: tell}.
\end{proof}

When cluster collapse is timely, we can bound cluster diameter at time $t \in [\TT_{\ell-1},\TT_\ell^-]$ from above, in terms of its separation at time $\TT_{\ell-1}$ or at time $t$. 

\begin{lemma}\label{lem: timelydiambd}
For any cluster $i$, when $\Quick (\ell - 1)$ occurs and when $t$ is at most $\TT_\ell^-$,
\begin{equation}\label{eq: timelydiambd}
\diam (U_t^i) \leq \big( \log \sep (U_{\TT_{\ell-1}}^i) \big)^2.
\end{equation}
Additionally, if $x_i$ is the center of the $i$\textsuperscript{th} cluster resulting from the exponential clustering of $U_0$, then when $t$ is at most $\TT_\ell^-$,
\begin{equation}\label{eq: uti xi}
U_t^i \subseteq D_{x_i} \left( \big( \log  \sep (U_{\TT_{\ell-1}}^i) \big)^2 \right) \quad \text{and} \quad U_t {\setminus} U_t^i \subseteq D_{x_i} \big( 0.99\,  \sep (U_{\TT_{\ell-1}}^i) \big)^c.
\end{equation}
Lastly, if $i,j$ label any two clusters which are nonempty at time $\TT_{\ell-1}$, then when $t$ is at most $\TT_\ell^-$,
\begin{equation}\label{eq: uti rat}
\frac{\log \diam (U_t^i)}{\log \dist(U_t^i, U_t^j)} \leq \frac{2.1 \log \log \rho_\ell}{\log \rho_\ell}.
\end{equation}
\end{lemma}

We use factors of $0.99$ and $2.1$ for concreteness; they could be replaced by $1-o_n(1)$ and $2+o_n(1)$. Lemma \ref{lem: timelydiambd} implements the diameter and separation bounds we discussed in Section \ref{subsubsec: midway} (there, we used $\rho$ in the place of $\sep (U_{\TT_{\ell-1}}^i)$). Before proving the lemma, we discuss some heuristics which explain \eqref{eq: timelydiambd} through \eqref{eq: uti rat}.

If a cluster is initially separated by a distance $\rho$, then it has a diameter of at most $2 \log \rho$ by \eqref{eq: exp sep}, which is negligible relative to an expiry time of order $(\log \rho)^2$. Diameter increases at most linearly in time by \eqref{eq: amlg}, so when cluster collapse is timely the diameter of $U_t^i$ is at most $\big( \log \sep(U_{\TT_{\ell-1}}^i) \big)^2$. In fact, the definition of the expiry time subtracts the lower order terms, so the bound will be exactly this quantity. Moreover, since $(\log \rho)^2$ is negligible relative to the separation $\rho$, and as separation decreases at most linearly in time by \eqref{eq: amlg2}, the separation of $U_t^i$ should be at least $\sep (U_{\TT_{\ell-1}}^i)$, up to a constant which is nearly one.

Combining these bounds on diameter and separation suggests that the ratio of the diameter of $U_t^i$ to its separation from another cluster $U_t^j$ should be roughly the ratio of $\big(\log \sep (U_{\TT_{\ell-1}}^i) \big)^2$ to $\sep (U_{\TT_{\ell-1}}^i)$, up to a constant factor. Because this ratio is decreasing in the separation (for separation exceeding, say, $e^2$) and because the separation at time $\TT_{\ell-1}$ is at least $\rho_\ell$, the ratio $\tfrac{(\log \rho_\ell)^2}{\rho_\ell}$ should provide a further upper bound, again up to a constant factor. These three observations correspond to \eqref{eq: timelydiambd} through \eqref{eq: uti rat}.

\begin{proof}[Proof of Lemma \ref{lem: timelydiambd}]
We first address \eqref{eq: timelydiambd} and use it to prove \eqref{eq: uti xi}. We then combine the results to prove \eqref{eq: uti rat}. We bound $\diam (U_t^i)$ from above in terms of $\diam (U_0^i)$ as
\begin{equation}\label{eq: uti tt}
\diam (U_t^i) \leq \diam (U_0^i) + \TT_\ell^- \leq \diam (U_0^i) + \TT_{\ell-1} + \tt_\ell.
\end{equation}
The first inequality holds because diameter grows at most linearly in time \eqref{eq: amlg} and because $t$ is at most $\TT_\ell^-$. The second inequality is due to the definition of $\TT_\ell^-$. We then bound $\diam (U_0^i)$ from above in terms of $\sep (U_{\TT_{\ell-1}}^i)$ as
\begin{equation}\label{du0 su0}
\diam (U_0^i) \leq 2 \log \sep (U_0^i) \leq 2 \log \big( \sep (U_{\TT_{\ell-1}}^i) + \TT_{\ell-1} \big).
\end{equation}
The exponential separation property \eqref{eq: exp sep} implies the first inequality and \eqref{eq: amlg2} implies the second.

Combining the two preceding displays, we find
\[ \diam (U_t^i) \leq 2 \log \big(\sep (U_{\TT_{\ell-1}}^i) + \TT_{\ell-1} \big) + \TT_{\ell-1} + \tt_\ell.\]
Substituting the definition of $\tt_\ell$, the right-hand side becomes
\[ 2 \log \big(\sep (U_{\TT_{\ell-1}}^i) + \TT_{\ell-1} \big) + (\log \rho_\ell)^2 - 4 \log (\rho_\ell + \TT_{\ell-1}).\]
By definition, $\rho_\ell$ is the least separation at time $\TT_{\ell-1}$, so we can further bound $\diam (U_t^i)$ from above by substituting $\sep (U_{\TT_{\ell-1}}^i)$ for $\rho_\ell$:
\begin{equation}\label{eq: ls2ls}
\diam (U_t^i) \leq \big( \log  \sep (U_{\TT_{\ell-1}}^i) \big)^2 - 2 \log \big(\sep (U_{\TT_{\ell-1}}^i) + \TT_{\ell-1} \big).
\end{equation} Dropping the negative term gives \eqref{eq: timelydiambd}.

We turn our attention to \eqref{eq: uti xi}. To obtain the first inclusion of \eqref{eq: uti xi}, we observe that $U_t^i$ is contained in the disk $D_{x_i} \big( \diam (U_0^i) + \TT_{\ell-1} + \tt_\ell \big)$, the radius of which is the quantity in \eqref{eq: uti tt} that we ultimately bounded above by $\big( \log  \sep (U_{\TT_{\ell-1}}^i) \big)^2$.

Concerning the second inclusion of \eqref{eq: uti xi}, we observe that for any $y$ in $U_t {\setminus} U_t^i$, there is some $y'$ in $U_{\TT_{\ell-1}} {\setminus} U_{\TT_{\ell-1}}^i$ such that $|y-y'|$ is at most $\tt_\ell$, because $t$ is at most $\TT_\ell^-$. By the triangle inequality and the bound on $|y-y'|$,
\[ |x_i - y| \geq |x_i - y'| - |y - y'| \geq |x_i - y'| - \tt_\ell.\]
Next, we observe that the distance between $x_i$ and $y'$ is at least
\[ |x_i - y'| \geq \sep (U_{\TT_{\ell-1}}^i) - \diam (U_0^i).\]
The two preceding displays and \eqref{du0 su0} imply
\begin{equation}\label{uti2}
|x_i - y| \geq \sep (U_{\TT_{\ell-1}}^i) - 2 \log \big( \sep (U_{\TT_{\ell-1}}^i) + \TT_{\ell-1} \big) - \tt_\ell.
\end{equation}
We continue \eqref{uti2} with
\begin{equation}\label{eq: xiy1}
|x_i - y| \geq \sep (U_{\TT_{\ell-1}}^i) - \big(\log \sep (U_{\TT_{\ell-1}}^i) \big)^2 \geq \big(1 - \tfrac{(\log \rho_\ell)^2}{\rho_\ell} \big) \, \sep (U_{\TT_{\ell-1}}^i) \geq 0.99 \, \sep (U_{\TT_{\ell-1}}^i).
\end{equation}
The first inequality follows from substituting the definition of $\tt_\ell$ into \eqref{uti2} and from $\sep (U_{\TT_{\ell-1}}^i) \geq \rho_\ell$. The second inequality holds because the ratio of $\big( \log \sep (U_{\TT_{\ell-1}}^i) \big)^2$ to $\sep (U_{\TT_{\ell-1}}^i)$ decreases as $\sep (U_{\TT_{\ell-1}}^i)$ increases and because $\sep (U_{\TT_{\ell-1}}^i) \geq \rho_\ell$. The fact \eqref{eq: quick2} that $\rho_\ell$ is at least $e^{\theta_{2n}}$ when $\Quick (\ell-1)$ occurs implies that the ratio in \eqref{eq: xiy1} is at most $0.01$, which justifies the third inequality. \eqref{eq: xiy1} proves the second inclusion of \eqref{eq: uti xi}.

Lastly, to address \eqref{eq: uti rat}, we observe that any element $x$ in $U_t^i$ is within a distance $\big(\log \sep (U_{\TT_{\ell-1}}^i) \big)^2$ of $x_i$ by \eqref{eq: ls2ls}. So, by \eqref{eq: xiy1} and simplifying with $\rho_\ell \geq e^{\theta_{2n}}$, the distance between $U_t^i$ and $U_t^j$ is at least 
\[ \sep (U_{\TT_{\ell-1}}^i) - 2 \big(\log \sep (U_{\TT_{\ell-1}}^i) \big)^2 \geq 0.99\, \sep (U_{\TT_{\ell-1}}^i).\]
Combining this with \eqref{eq: timelydiambd}, and then using the fact that $\sep (U_{\TT_{\ell-1}}^i)$ is at least $\rho_\ell$, gives
\[ \frac{\log \diam (U_t^i)}{\log \dist(U_t^i, U_t^j)} \leq \frac{2 \log \log \sep (U_{\TT_{\ell-1}}^i)}{\log \big(0.99\, \sep (U_{\TT_{\ell-1}}^i) \big)} \leq \frac{2.1\log \log \rho_\ell}{ \log \rho_\ell}.\]
\end{proof}

The next lemma concerns two properties of the midway point introduced in Section~\ref{sec: prop68strat}. We recall that the midway point (for the period beginning at time $\TT_{\ell-1}$ and continuing until $\TT_\ell^-$) is a circle of radius $\tfrac12 \rho_\ell$, centered on the center $x_i$ (given by the initial exponential clustering of $U_0$) of a cluster $i$ which is least separated at time $\TT_{\ell-1}$. The first property is the simple fact that, when collapse is timely, the midway point separates $U_t^i$ from the rest of $U_t$ until time $\TT_\ell^-$. This is clear because the midway point is a distance of $\tfrac12 \rho_\ell$ from $U_{\TT_{\ell-1}}$ and $\TT_\ell^-$ is no more than $(\log \rho_\ell)^2$ steps away from $\TT_{\ell-1}$ when collapse is timely. The second property is the fact that a random walk from anywhere in the midway point hits $U_t^i$ before the rest of $U_t$ (excluding the site of the activated particle) with a probability of at most $0.51$, which is reasonable because the random walk begins effectively halfway between $U_t^i$ and the rest of $U_t$. In terms of notation, when activation occurs at $u$, the bound applies to the probability of the event
\[\big\{ \tau_{U_t^i {\setminus} \{u\}} < \tau_{U_t {\setminus} (U_t^i \, \cup \, \{u\})} \big\}.\] We will stipulate that $u$ belongs to a cluster in $U_t$ which is not a singleton as, otherwise, its activation at time $t$ necessitates $t = \TT_\ell$.

\begin{lemma}\label{lem: midway2} Suppose cluster $i$ is least separated at time $\TT_{\ell -1}$ and recall that $x_i$ denotes the center of the $i$\textsuperscript{th} cluster, determined by the exponential clustering of $U_0$. When $\Quick (\ell - 1)$ occurs and when $t$ is at most $\TT_\ell^-$: 
\begin{enumerate}
\item the midway point $C(i; \ell) = C_{x_i} \left(\tfrac12 \rho_\ell\right)$ separates $U_t^i$ from $U_t {\setminus} U_t^i$, and
\item for any $u$ in $U_t$ which does not belong to a singleton cluster and any $y$ in $C(i; \ell)$,
\begin{equation}\label{eq: midway bd} 
\P_y \left(\tau_{U_t^i {\setminus} \{u\}} < \tau_{U_t {\setminus} (U_t^i \,\cup\, \{u\})} \right) \leq 0.51.
\end{equation}
\end{enumerate}
\end{lemma}

\begin{proof} Property (1) is an immediate consequence of \eqref{eq: uti xi} of Lemma~\ref{lem: timelydiambd}, since $\tfrac12 \rho_\ell$ is at least $(\log \rho_\ell)^2$ and less than $0.99 \rho_\ell$.

Now let $u$ and $y$ satisfy the hypotheses, denote the center of the $i$\textsuperscript{th} cluster by $x_i$, and denote $C((\log \rho_\ell)^2)$ by $B$. To prove property $(2)$, we will establish
\begin{equation}\label{eq: midwaybd1}
\P_{y-x_i} \left( \tau_B < \tau_{z-x_i} \right) \leq 0.51,
\end{equation}
for some $z \in U_t {\setminus} (U_t^i \, \cup \, \{u\})$. This bound implies \eqref{eq: midway bd} because, by \eqref{eq: uti xi}, $B$ separates $U_t^i$ from the rest of $U_t$.

We can express the probability in \eqref{eq: midwaybd1} in terms of hitting probabilities involving only three points:
\begin{align*}
\P_{y-x_i} (\tau_{B} < \tau_{z-x_i})
& = \P_{y-x_i} (\tau_o < \tau_{z-x_i}) + \E_{y-x_i} \left[ \P_{S_{\tau_{B}}} ( \tau_{z-x_i} < \tau_o ) \1 (\tau_{B} < \tau_{z-x_i}) \right] \nonumber\\
& \leq \P_{y-x_i} (\tau_o < \tau_{z-x_i}) + \max_{v \in B} \P_v (\tau_{z-x_i} < \tau_o) \, \P_{y-x_i} (\tau_{B} < \tau_{z-x_i}).
\end{align*}
Rearranging, we find
\begin{equation}\label{eq: z0y}
\P_{y-x_i} (\tau_{B} < \tau_{z-x_i}) \leq \Big( 1 - \max_{v \in B} \P_v (\tau_{z-x_i} < \tau_o) \Big)^{-1} \P_{y-x_i} (\tau_o < \tau_{z-x_i}).
\end{equation}

We will choose $z$ so that the points $y-x_i$ and $z-x_i$ will be at comparable distances from the origin and, consequently, $\P_{y-x_i} (\tau_o < \tau_{z-x_i})$ will be nearly $1/2$. In contrast, every element of $B$ will be far nearer to the origin than to $z-x_i$, so $\P_v (\tau_{z-x_i} < \tau_o)$ will be nearly zero for every $v$ in $B$. We will write these probabilities in terms of the potential kernel using Lemma~\ref{lem: w and p}. We will need bounds on the distances $|z-x_i|$ and $|z-y|$ to simplify the potential kernel terms; we take care of this now.

Suppose cluster $j$ was nearest to cluster $i$ at time $\TT_{\ell-1}$. We then choose $z$ to be the element of $U_t^j$ nearest to $U_t^i$. Note that such an element exists because, when $t$ is at most $\TT_\ell^-$, every cluster surviving until time $\TT_{\ell-1}$ survives until time $t$. By \eqref{eq: uti xi} of Lemma~\ref{lem: timelydiambd},
\[ |z - x_i| \geq 0.99 \rho_\ell.\]
Part (2) of Lemma~\ref{lem: potkerbds} then gives the lower bound
\begin{equation}\label{eq: z0y1}
\aa (z-x_i) \geq \frac{2}{\pi} \log (0.99 \rho_\ell).
\end{equation}

In the inter-collapse period before $\TT_\ell^-$, the separation between $z$ and $y$ (initially $\tfrac12 \rho_\ell$) can grow by at most $\tt_\ell + \diam (U_{\TT_{\ell-1}}^j)$:
\begin{align*}
|z-y| & \leq \tfrac{1}{2} \rho_\ell + \tt_\ell + \diam (U_{\TT_{\ell-1}}^j).
\intertext{By \eqref{eq: timelydiambd}, the diameter of cluster $j$ at time $\TT_{\ell-1}$ is at most $(\log \rho_\ell)^2$; this upper bound applies to $\tt_\ell$ as well, so}
|z-y| & \leq \tfrac{1}{2} \rho_\ell + 2 (\log \rho_\ell)^2 \leq 0.51 \rho_\ell.
\end{align*} We obtained the second inequality using the fact \eqref{eq: quick2} that, when $\Quick(\ell-1)$ occurs, $\rho_\ell$ is at least $e^{\theta_{2n}}$. (In what follows, we will use this fact without restating it.)

Accordingly, the difference between $\aa (z-y)$ and $\aa (y-x_i)$ satisfies
\begin{equation}\label{eq: z0y2}
\aa (z-y) - \aa (y-x_i) \leq \frac{2}{\pi} \log (2\cdot 0.51) + 4 \lambda \rho_\ell^{-2} \leq \frac{2}{\pi}.
\end{equation}

By Lemma~\ref{lem: w and p}, the first term of \eqref{eq: z0y} equals
\begin{equation}\label{eq: vB1}
\P_{y-x_i} (\tau_o < \tau_{z-x_i}) = \frac12 + \frac{\aa (z-y) - \aa (y-x_i)}{2 \aa(z-x_i)}.
\end{equation}
Substituting \eqref{eq: z0y1} and \eqref{eq: z0y2} into \eqref{eq: vB1}, we find
\begin{equation}\label{eq: z0y4}
\P_{y-x_i} (\tau_o < \tau_{z-x_i}) \leq \frac12 + \frac{1}{\log \rho_\ell} \leq 0.501.
\end{equation}

We turn our attention to bounding above the maximum of $\P_v (\tau_{z-x_i} < \tau_o)$ over $v$ in $B$. For any such $v$, Lemma~\ref{lem: w and p} gives
\begin{equation}\label{eq: vB2}
\P_v (\tau_{z-x_i} < \tau_o) = \frac12 + \frac{\aa (v) - \aa (z-x_i - v)}{2 \aa (z-x_i)}.
\end{equation} By Lemma~\ref{lem: bdy est}, $\aa (v)$ is at most $\aa' ((\log \rho_\ell)^2) + 2 (\log \rho_\ell)^{-2}$. Then, since
\[
|z-x_i-v| \geq 0.99 \rho_\ell - (\log \rho_\ell)^2 \geq 0.98 \rho_\ell,
\]
we have
\begin{equation}\label{eq: z0y3}
\aa (z-x_i - v) - \aa (v) \geq \frac{2}{\pi} \log (0.98\rho_\ell) - \frac{4}{\pi} \log \log \rho_\ell - 4 (\log \rho_\ell)^{-2} \geq \frac{2\cdot 0.99}{\pi} \log ( 0.99 \rho_\ell).
\end{equation}

Substituting \eqref{eq: z0y1} and \eqref{eq: z0y3} into \eqref{eq: vB2}, we find
\[ \P_v (\tau_{z-x_i} < \tau_o) \leq \frac12 - \frac{0.99}{2} \leq 0.005.\]
This bound holds uniformly over $v$ in $B$. Applying it and \eqref{eq: z0y4} to \eqref{eq: z0y}, we find
\[ \P_{y-x_i} (\tau_{B} < \tau_{z-x_i}) \leq (1 - 0.005)^{-1} 0.501 \leq 0.51.\]
\end{proof}

Combined with the separation lower bound \eqref{eq: quick2} of Lemma \ref{lem: quick}, the inclusions \eqref{eq: uti xi} of Lemma~\ref{lem: timelydiambd} ensure that, when $\Quick (\ell-1)$ occurs, nonempty clusters at time $t \in [\TT_{\ell-1},\TT_\ell^-]$ are contained in well separated disks. A natural consequence is that, when $\Quick (\ell-1)$ occurs, every nonempty cluster has positive harmonic measure in $U_t$. Later, we will use this fact in conjunction with Theorem \ref{thm: hm} to control the activation step of the HAT dynamics.

\begin{lemma}\label{still ec}
Let $I_\ell$ be the set of indices of nonempty clusters at time $\TT_{\ell-1}$. When $\Quick (\ell-1)$ occurs and when $t$ is at most $\TT_\ell^-$, $\H_{U_t} (U_t^i) > 0$ for every $i \in I_\ell$.
\end{lemma}

The proof is similar to that of Lemma \ref{lem: n path}. Recall the definition of the $\ast$-exterior boundary \eqref{ast bd} and define the disk $D^i$ to be the one from \eqref{eq: uti xi}
\begin{equation}\label{didxi}
D^i = D_{x_i} \big( \big( \log  \sep (U_{\TT_{\ell-1}}^i) \big)^2 \big) \,\,\,\text{for each $i \in I_\ell$}.
\end{equation}
For simplicity, assume $1 \in I_\ell$. Most of the proof is devoted to showing that there is a path $\Gamma$ from $\bdye D^1$ to a large circle $C$ about $U_t$, which avoids $\cup_{i \in I_\ell} D^i$ and thus avoids $U_t$. To do so, we will specify a candidate path from $\bdye D^1$ to $C$, and modify it as follows. If the path encounters a disk $D^i$, then we will reroute the path around $\bdye D^i$ (which will be connected and will not intersect another disk). The modified path encounters one fewer disk. We will iterate this argument until the path avoids every disk and therefore never returns to $U_t$.

\begin{proof}[Proof of Lemma \ref{still ec}]
Suppose $\Quick (\ell-1)$ occurs and $t \in [\TT_{\ell-1},\TT_\ell^-]$, and assume w.l.o.g.\ that $1 \in I_\ell$. Let $y \in U_t^1$ satisfy $\H_{U_t^1} (y) > 0$. For each $i \in I_\ell$, let $D^i$ be the disk defined in \eqref{didxi}. As $\H_{U_t^1} (y)$ is positive, there is a path from $y$ to $\bdye D^1$ which does not return to $U_t^1$. In a moment, we will show that $\bdye D^1$ is connected, so it will suffice to prove that there is a subsequent path from $\bdye D^1$ to $C = C_{x_1} (2 \, \diam (U_t))$ which does not return to $E = \cup_{i \in I_\ell} D^i$. This suffices when $\Quick (\ell-1)$ occurs because then, by Lemma \ref{lem: timelydiambd}, $U_t^i \subseteq D^i$ for each $i \in I_\ell$, so $U_t \subseteq E$.

We make two observations. First, because each $D^i$ is finite and $\ast$-connected, Lemma 2.23 of \cite{kesten1986} (alternatively, Theorem 4 of \cite{timar2013}) states that each $\bdye D^i$ is connected. Second, $\bdye D^i$ is disjoint from $E$ when $\Quick (\ell-1)$ occurs; this is an easy consequence of \eqref{eq: uti xi} and the separation lower bound \eqref{eq: quick2}.

We now specify a candidate path from $\bdye D^1$ to $C$ and, if necessary, modify it to ensure that it does not return to $E$. Because $\H_{U_t^1} (y)$ is positive, there is a shortest path $\Gamma$ from $\bdye D^1$ to $C$, which does not return to $U_t^1$. Let $L$ be the set of labels of disks encountered by $\Gamma$. If $L$ is empty, then we are done. Otherwise, let $i$ be the label of the first disk encountered by $\Gamma$, and let $\Gamma_a$ and $\Gamma_b$ be the first and last elements of $\Gamma$ which intersect $\bdye D^i$. By our first observation, $\bdye D^i$ is connected, so there is a shortest path $\Lambda$ in $\bdye D^i$ from $\Gamma_a$ to $\Gamma_b$. When edit $\Gamma$ to form $\Gamma'$ as
\[ \Gamma' = \left( \Gamma_1, \dots, \Gamma_{a-1}, \Lambda_1, \dots, \Lambda_{|\Lambda|}, \Gamma_{b+1},\dots, \Gamma_{|\Gamma|} \right).\]

Because $\Gamma_b$ was the last element of $\Gamma$ which intersected $\bdye D^i$, $\Gamma'$ avoids $D^i$. Additionally, by our second observation, $\Lambda$ avoids $E$, so if $L'$ is the set of labels of disks encountered by $\Gamma'$, then $|L'| \leq |L|-1$. If $L'$ is empty, then we are done. Otherwise, we can relabel $\Gamma$ to $\Gamma'$ and $L$ to $L'$ in the preceding argument to continue inductively, obtaining $\Gamma''$ and $|L''| \leq |L| - 2$, and so on. Because $|L| \leq n$, we need to modify the path at most $n$ times before the resulting path from $y$ to $C$ does not return to $E$.
\end{proof}

The last result of this section bounds above escape probabilities; we will shortly specialize it for our setting. Note that $\bdy A_\rho$ denotes the exterior boundary of the $\rho$-fattening of $A$, not the $\rho$-fattening of $\bdy A$.

\begin{lemma}\label{lem: esc up}
If $A$ is a subset of $\Z^2$ with at least two elements and if $\rho$ is at least twice the diameter of $A$, then, for $x$ in $A$,
\begin{equation}\label{eq: esc up}
\P_x \left(\tau_{\bdy (A {\setminus} \{x\})_\rho} < \tau_{A{\setminus} \{x\}} \right) \leq \frac{\log \diam (A) + 2}{\log \rho}.
\end{equation}
\end{lemma}

The added $2$ in \eqref{eq: esc up} is unimportant. Note that, if $A$ was a singleton set, then the probability in question would be proportional to $(\log \rho)^{-1}$. The $\log \diam (A)$ term arises from the fact that, if $|A| \geq 2$, then a random walk from $x$ must avoid at least one element in $A{\setminus}\{x\}$, at a distance of at most $\diam (A)$ from $x$.

\begin{proof}[Proof of Lemma \ref{lem: esc up}]
We will replace the event in \eqref{eq: esc up} with a more probable but simpler event and bound above its probability instead.

By hypothesis, $A$ has at least two elements, so for any $x$ in $A$, there is some $y$ in $A {\setminus} \{x\}$ nearest to $x$. To escape to $\bdy (A{\setminus} \{x\})_\rho$ without hitting $A {\setminus} \{x\}$ it is necessary to escape to $C_y (\rho)$ without hitting $y$. Accordingly, for a random walk from $x$, the following inclusion holds
\begin{equation}\label{eq: esc up inc}
\{ \tau_{\bdy (A {\setminus} \{x\})_\rho} < \tau_{A {\setminus}\{x\}} \} \subseteq \{ \tau_{C_y (\rho)} < \tau_y\}.
\end{equation}
To prove \eqref{eq: esc up} it therefore suffices to obtain the same bound for the larger event.

The hypothesis $\rho \geq 2 \, \diam (A)$ ensures that $x-y$ lies in $D(\rho)$, so we can apply the optional stopping theorem to the martingale $\aa (S_{j \wedge \tau_{o}})$ at the stopping time $\tau_{C(\rho)}$. Doing so, we find
\begin{equation}\label{eq: axy1} \P_x (\tau_{C_y (\rho)} < \tau_y) = \P_{x-y} (\tau_{C(\rho)} < \tau_{o}) = \frac{\aa (x-y)}{\E_{x-y} [\aa (S_{\tau_{C(\rho)}}) \bigm\vert \tau_{C(\rho)} < \tau_{o}]}.
\end{equation}
We apply Lemma~\ref{lem: bdy est} with $r = \rho$ and $x = o$ to find
\begin{equation}\label{eq: axy2}
\E_{x-y} [\aa (S_{\tau_{C(\rho)}}) \bigm\vert \tau_{C(\rho)} < \tau_{o}] \geq \aa' (\rho) - \rho^{-1} \geq \frac{2}{\pi} \log \rho.
\end{equation}
By \eqref{potker} and the facts that $1 \leq |x-y| \leq \diam (A)$ and $\kappa + \lambda \leq 1.1$, the numerator of \eqref{eq: axy1}, is at most
\begin{equation}\label{eq: a11bd}
\aa (x-y) \leq \frac{2}{\pi} \log |x-y| + \kappa + \lambda |x-y|^{-2} \leq \frac{2}{\pi} \log \diam (A) + 1.1.
\end{equation}

Substituting \eqref{eq: axy2} and \eqref{eq: a11bd} into \eqref{eq: axy1}, and simplifying with $\frac{1.1\pi}{2} \leq 2$, we find
\[ \P_x (\tau_{C_y (\rho)} < \tau_y) \leq \frac{\log \diam (A) + 2}{\log \rho}.\]
Due to the inclusion \eqref{eq: esc up inc}, this implies \eqref{eq: esc up}.
\end{proof}

\subsection{Proof of Proposition~\ref{prop: l clust col}}\label{sec: lclustcol pf}

Recall that, for $t \in [\TT_{\ell-1}, \TT_\ell^-]$, the midway point is a circle which surrounds one of the clusters which is least separated at time $\TT_{\ell-1}$. We call this cluster the watched cluster, to distinguish it from other clusters which are least separated at $\TT_{\ell-1}$. The results of this section are phrased in these terms and through the following events.

\begin{definition}
For any $x \in \Z^2$, time $t \geq 0$, and any $1 \leq i \leq k$, define the activation events
\[ \Act (x,t) = \left\{ \text{$x$ is activated at time $t$} \right\} \quad \text{and} \quad \Act (i,t) = \bigcup_{x \in U_t^i} \Act (x,t).\]

Additionally, define the deposition event
\[ \Dep (i,t) = \bigcup_{x \in U_t} \Act (x,t) \cap \left\{ \tau_{U_t^i \, {\setminus} \, \{x\}} < \tau_{U_t \, {\setminus} \, (U_t^i \, \cup \, \{x\})} \right\}.\] In words, the deposition event requires that, at time $t$, the activated particle deposits at the $i$\textsuperscript{th} cluster.

When $\Quick (\ell - 1)$ occurs, if the $i$\textsuperscript{th} cluster is the watched cluster at time $\TT_{\ell - 1}$, then for any time $t \in \left[ \TT_{\ell - 1}, \TT_\ell^- \right]$, define the ``midway'' event as
\[ \Mid (i,t ; \ell) = \bigcup_{x \in U_t} \Act (x,t) \cap \left\{ \tau_{C(i; \ell)} < \tau_{U_t {\setminus} \{x\}} \right\}.\] In words, the midway event specifies that, at time $t$, the activated particle reaches $C(i; \ell)$ before deposition.
\end{definition}

We will now use the results of the preceding subsection to bound below the probability that activation occurs at the watched cluster and that the activated particle subsequently reaches the midway point. Essentially, Theorem~\ref{thm: hm} addresses the former probability and Theorem~\ref{thm: esc} addresses the latter. However, it is necessary to first ensure that the watched cluster has positive harmonic measure, so that at least one of its particles can be activated and the lower bound \eqref{eq: ext hm thm} of Theorem \ref{thm: hm} can apply. This is handled by Lemma \ref{still ec}, the hypotheses of which are satisfied whenever $\Quick (\ell-1)$ occurs and $t \in [\TT_{\ell-1},\TT_\ell^-]$. The hypotheses of Theorem \ref{thm: esc} will be satisfied in this context so long as we estimate the probability of escape to a distance $\rho$ which is at least twice the cluster diameter. The distance from the watched cluster to the midway point is roughly $\rho_\ell$, while the cluster diameter is at most $(\log \rho_\ell)^2$ by \eqref{eq: uti xi} of Lemma \ref{lem: timelydiambd}, so this will be the case.

The lower bounds from Theorems~\ref{thm: hm} and \ref{thm: esc} will imply that a particle with positive harmonic measure is activated and reaches the midway point with a probability of at least
\[\exp (-c_1 n \log n + \log (c_2 n^{-2})) \cdot (\log \rho_\ell)^{-1}\]
for constants $c_1,c_2$. From our choice of $\gamma$ \eqref{eq: gamma}, the first factor in the preceding display is at least
\begin{equation}\label{eq: alphan}
\alpha_n = e^{\gamma n \log n}.
\end{equation}

\begin{proposition}\label{prop: dep lower bd} Let cluster $i$ be least separated at time $\TT_{\ell - 1}$. When $\Quick (\ell-1)$ occurs and when $t \in \left[ \TT_{\ell - 1}, \TT_\ell^- \right]$, we have
	\begin{equation}\label{eq: dep lower bd}
	\PP \left( \Mid (i,t ; \ell) \cap \Act (i,t) \bigm\vert \UU_t \right) \geq (\alpha_n \log \rho_\ell)^{-1}.
	\end{equation}
\end{proposition}

\begin{proof} Fix $\ell$, suppose the $i$\textsuperscript{th} cluster is least separated at time $\TT_{\ell - 1}$ and $\Quick (\ell-1)$ occurs, and let $t \in \left[ \TT_{\ell - 1}, \TT_\ell^- \right]$. For any $x \in U_t^i$, we have
\begin{equation}\label{eq: combo control}
\PP ( \Mid (i,t ; \ell) \cap \Act(i,t) \bigm\vert \UU_t) \geq \PP \left( \Mid (i,t ; \ell) \bigm\vert \Act (x,t),\,\UU_t \right) \PP(\Act (x,t) \bigm\vert \UU_t).
\end{equation} Let $B$ denote the set of all points within distance $\rho_\ell$ of $U_t^i$. We have the following inclusion when $\Act (x,t)$ occurs: \begin{equation}\label{eq: key incl} \left\{ \tau_{\bdy B} < \tau_{U_t^i} \right\} \subseteq \left\{ \tau_{C(i; \ell)} < \tau_{U_t {\setminus} \{x\}}\right\} = \Mid (i,t ; \ell).\end{equation}
From \eqref{eq: key incl}, we have
\begin{equation}\label{eq: key incl2}
\PP \left( \Mid (i,t ; \ell) \bigm\vert \Act (x,t),\, \UU_t \right) \geq \PP\left( \tau_{\bdy B} < \tau_{U_t^i} \Bigm\vert \Act (x,t),\, \UU_t \right) = \P_x \left( \tau_{\bdy B} < \tau_{U_t^i} \right).
\end{equation} 

Now let $x$ be an element of $U_t^i$ which is exposed and which maximizes \eqref{eq: key incl2}. Such an element must exist because, by Lemma \ref{still ec}, when $\Quick (\ell-1)$ occurs and when $t \in [\TT_{\ell-1},\TT_\ell^-]$, $\H_{U_t} (U_t^i)$ is positive. We aim to apply Theorem~\ref{thm: esc} to bound below the probability in \eqref{eq: key incl2}. The hypotheses of Theorem~\ref{thm: esc} require $|U_t^i| \geq 2$ and $\rho_\ell \geq 2\, \diam (U_t^i)$. First, the cluster $U_t^i$ must contain at least two elements as, otherwise, activation at $x$ would necessitate $t = \TT_\ell$. Second, $\rho_\ell$ is indeed at least twice the diameter of $U_t^i$ because, when $\Quick (\ell-1)$ occurs, $U_t^i$ is contained in a disk of radius $(\log \rho_\ell)^2$ by \eqref{eq: uti xi}. Theorem~\ref{thm: esc} therefore applies to \eqref{eq: key incl2}, giving
\begin{equation}\label{eq: key esc}
\PP \left( \Mid (i,t ; \ell) \bigm\vert \Act (x,t),\, \UU_t \right) \geq c_2 (n^2 \log \rho_\ell)^{-1}.
\end{equation}

The harmonic measure lower bound \eqref{eq: ext hm thm} of Theorem \ref{thm: hm} applies because $x$ has positive harmonic measure. According to \eqref{eq: ext hm thm}, the harmonic measure of $x$ is at least
\begin{equation}\label{eq: key hm}
\PP \left( \Act (x,t) \bigm\vert \UU_t \right) = \H_{U_t} (x) \geq e^{-c_1 n \log n}.
\end{equation} 
Combining \eqref{eq: key esc} and \eqref{eq: key hm}, we find
\begin{equation*}
\PP ( \Mid (i,t ; \ell) \cap \Act(i,t) \bigm\vert \UU_t) \geq c_2 (n^2 \log \rho_\ell)^{-1} \cdot e^{-c_1 n \log n} \geq (\alpha_n \log \rho_\ell)^{-1}.
\end{equation*} The second inequality is due to the definition of $\alpha_n$ \eqref{eq: alphan}.
\end{proof}

Next, we will bound below the conditional probability that activation occurs at the watched cluster, given that the activated particle reaches the midway point. 


\begin{proposition}\label{prop: comp rates}
Let cluster $i$ be the watched cluster at time $\TT_{\ell - 1}$. When $\Quick (\ell-1)$ occurs and when $t \in \left[ \TT_{\ell - 1}, \TT_\ell^- \right]$, we have
\begin{equation}\label{eq: comp rates}
\PP \left( \Act (i,t) \bigm\vert \Mid (i,t ; \ell),\,\UU_t\right) \geq (3 \alpha_n \log \log \rho_\ell)^{-1}.
\end{equation}
\end{proposition}

\begin{proof}
Suppose $t \in [\TT_{\ell-1}, \TT_\ell^-]$ and $\Quick (\ell-1)$ occurs. If we obtain a lower bound $p_1$ on $\PP \left( \Mid (i,t ; \ell) \cap \Act (i,t) \bigm\vert \UU_t \right)$ and an upper bound $p_2$ on $\PP \left( \bigcup_{j \neq i} \Mid (i,t ; \ell) \cap \Act (j,t) \bigm\vert \UU_t \right)$, then
\begin{equation}\label{eq: p rat 1}
\PP \left( \Act (i,t) \bigm\vert \Mid (i,t ; \ell), \, \UU_t \right) \geq \frac{p_1}{p_1 + p_2}.
\end{equation}

First, the probability $\PP \left(\Mid (i,t ; \ell) \cap \Act (i,t) \bigm\vert \UU_t \right)$ is precisely the one we used to establish \eqref{eq: combo control} in the proof of Proposition~\ref{prop: dep lower bd}; $p_1$ is therefore at least $(\alpha_n \log \rho_\ell)^{-1}$.

Second, for any $j \neq i$, we use the trivial upper bound $\PP (\Act (j,t) \bigm\vert \UU_t) \leq 1$ and address the midway component by writing
\begin{equation}\label{eq: p rat 2}
\PP \left( \Mid (i,t ; \ell) \bigm\vert \Act (j,t),\, \UU_t \right) = \EE \left[ \P_X \left( \tau_{C(i; \ell)} < \tau_{U_t {\setminus} \{X\}} \right) \bigm\vert \Act (j,t),\, \UU_t \right].
\end{equation}

Use $\rho$ to denote $\dist (U_{\TT_{\ell-1}}^i, U_{\TT_{\ell-1}}^j)$ and $B$ to denote the set of all points within a distance $\rho$ of $U_t^j{\setminus}\{X\}$. (We use $\rho$ instead of $\rho_\ell$ because $j$ is not necessarily the cluster nearest cluster $i$.) We can use Lemma~\ref{lem: esc up} to bound the probability in \eqref{eq: p rat 2} because, for any random walk from $X$, the following inclusion holds:
\[ \{\tau_{C(i; \ell)} < \tau_{U_t {\setminus} \{X\}}\} \subseteq \left\{ \tau_B < \tau_{U_t^j {\setminus} \{X\}}\right\}.\]

Because the cluster $U_t^j$ has at least two elements and because $\rho$ is at least twice its diameter, an application of Lemma~\ref{lem: esc up} with $A = U_t^j$ and $\rho$ yields
\[ \P_x \left( \tau_B < \tau_{U_t^j {\setminus} \{x\}}\right) \leq \frac{\log \diam (U_t^j) + 2}{\log \rho} \leq \frac{2.2 \log \log \rho_\ell}{ \log \rho_\ell},\]
uniformly for $x$ in $U_t^j$. The second inequality follows from \eqref{eq: uti rat}, which bounds the ratio of $\log \diam (U_t^j)$ to $\log \rho$ by $\tfrac{2.1 \log \log \rho_\ell}{\log \rho_\ell}$.

Applying the preceding bound to \eqref{eq: p rat 2}, we find
\[ \PP (\Mid (i,t;\ell) \bigm\vert \Act (j,t) \UU_t) \leq \frac{2.2 \log \log \rho_\ell}{\log \rho_\ell} =: p_2.\]
Then, substituting $p_1$ and $p_2$ in \eqref{eq: p rat 1}, we conclude
\[ \PP \left( \Act (i,t) \bigm\vert \Mid (i,t ; \ell), \, \UU_t \right) \geq (1 + 2.2 \alpha_n \log \log \rho_\ell)^{-1} \geq (3 \alpha_n \log \log \rho_\ell)^{-1}.\]
\end{proof}

We now use Lemma~\ref{lem: midway2} to establish that an activated particle, upon reaching the midway point, deposits at the watched cluster with a probability of no more than $0.51$. 

\begin{proposition}\label{prop: mid dep bd} Let cluster $i$ be the watched cluster at time $\TT_{\ell -1}$. When $\Quick (\ell - 1)$ occurs and when $t \in \left[ \TT_{\ell - 1}, \TT_\ell^- \right]$, for $x$ in $U_t$, we have
	\begin{equation}\label{eq: mid dep bd}
	\PP \left( \Dep (i,t) \Bigm\vert \Mid (i,t ; \ell),\,\Act (x,t), \,\UU_t\right) \leq 0.51.
	\end{equation}
\end{proposition}

\begin{proof} Using the definitions of $\Dep (i,t)$ and $\Mid(i,t;\ell)$, we write \begin{multline*} 
\PP \left( \Dep (i,t) \bigm\vert \Mid (i,t ; \ell),\,\Act (x,t),\, \UU_t\right)\\ 
	= \PP\left( \bigcup_{y \in U_t} \Act (y,t) \cap \left\{ \tau_{U_t^i \, {\setminus} \, \{y\}} < \tau_{U_t \, {\setminus} \, (U_t^i \, \cup \, \{y\})} \right\} \Biggm\vert \bigcup_{y \in U_t} \Act (y,t) \cap \left\{ \tau_{C(i; \ell)} < \tau_{U_t {\setminus} \{y\}} \right\},\, \Act (x,t),\,\UU_t\right).\end{multline*} Because $\Act (y,t)$ only occurs for one particle $y$ in $U_t^i$ at any given time $t$, the right-hand side simplifies to \[ \EE\left[ \P_x \left (\tau_{U_t^i \, {\setminus} \, \{x\}} < \tau_{U_t \, {\setminus} \, (U_t^i \, \cup \, \{x\})} \Bigm\vert \tau_{C(i; \ell)} < \tau_{U_t {\setminus} \{x\}} \right) \Bigm\vert \Act (x,t),\,\UU_t \right].\] We then apply the strong Markov property to $\tau_{C(i; \ell)}$ to find that the previous display equals
\begin{align*} 
\EE \left[ \P_{S_{\tau_{C(i; \ell)}}} \left( \tau_{U_t^i \, {\setminus} \, \{x\}} < \tau_{U_t \, {\setminus} \, (U_t^i \, \cup \, \{x\})} \right) \Bigm\vert \Act (x,t),\,\UU_t \right]
&\leq 0.51,
\end{align*} where the inequality follows from the estimate \eqref{eq: midway bd}.
\end{proof}

The preceding three propositions realize the strategy of Section~\ref{subsubsec: midway}. We proceed to implement the strategy of Section~\ref{subsubsec: coupling}. In brief, we will compare the number of particles in the watched cluster to a random walk and bound the collapse time using the hitting time of zero of the walk. 

Let cluster $i$ be the watched cluster at time $\TT_{\ell-1}$ and denote by $(\eta_\ell (m))_{m \geq 0}$ the consecutive times at which the midway event $\Mid (i,\cdot;\ell)$ occurs. Set $\eta_\ell (0) \equiv \TT_{\ell-1}-1$ and for all $m \geq 1$ define
\[ \eta_\ell (m) = \inf \{t > \eta_\ell (m-1): \Mid (i,t;\ell)\,\,\text{occurs}\}.\]
Additionally, we denote the number of midway event occurrences by time $t$ as 
\[ \mm_\ell (t) = \sum_{m=1}^\infty \1 \left( \eta_\ell (m) \leq t \right).\]

The number of elements in cluster $i$ viewed at these times can be coupled to a lazy random walk $(W_m)_{m \geq 0}$ on $\{0,\dots,n\}$ from $W_0 \equiv \big| U_{\TT_{\ell-1}}^i\big|$, which takes down-steps with probability $q_W = (7 \alpha_n \log \log \rho_\ell)^{-1}$ and up-steps with probability $1-q_W$, unless it attempts to take a down-step at $W_m = 0$ or an up-step at $W_m = n$, in which case it remains where it is.

When $\Quick (\ell-1)$ occurs, at each time $\eta_\ell (\cdot)$, the watched cluster has a chance of losing a particle of at least $0.49 (3 \alpha_n \log \log \rho_\ell)^{-1} \geq q_W$ (Propositions~\ref{prop: dep lower bd} and \ref{prop: comp rates}). The standard coupling of $|U_{\eta_\ell (m) + 1}^i|$ and $W_m$ will then guarantee $|U_{\eta_\ell (m) + 1}^i| \leq W_m$. However, this inequality will only hold when $m \leq \mm_\ell (\TT_\ell^-)$.

\begin{lemma}\label{lem: rw lem1}
Let cluster $i$ be the watched cluster at time $\TT_{\ell-1}$. There is a coupling of $\big( \big| U_{\eta_\ell (m)+1}^i \big| \big)_{m \geq 0}$ and $(W_m)_{m \geq 0}$ such that, when $\Quick (\ell-1)$ and $\{\mm_\ell (\TT_\ell^-) \geq M\}$ occur, $\big|U_{\eta_\ell (m)}^i \big| \leq W_m$ for all $m \leq M$.
\end{lemma}
\begin{proof}
Define
\[q (m) = \PP \Bigg( \Act (i,\eta_\ell (m)) \cap \bigcup_{j \neq i} \Dep (j,\eta_\ell (m)) \Biggm\vert \UU_{\eta_\ell (m)} \Bigg).\] In words, the event in the previous display is the occurrence of $\Mid (i, \eta_\ell (m) ; \ell)$, preceded by activation at cluster $i$ and followed by deposition at cluster $j \neq i$; this is the probability that the watched cluster loses a particle.

Couple $\big(\big|U_{\eta_\ell (m) + 1}^i \big|\big)_{m\geq 0}$ and $(W_m)_{m \geq 0}$ in the standard way. When $\Quick (\ell -1)$ occurs, by Propositions~\ref{prop: comp rates} and \ref{prop: mid dep bd}, the estimates \eqref{eq: comp rates} and \eqref{eq: mid dep bd} hold for all $t \in [\TT_{\ell - 1}, \TT_\ell^-]$. In particular, these estimates hold at time $\eta_\ell (m)$ for any $m \leq M$ when $\{\mm_\ell (\TT_\ell^-) \geq M\}$ occurs. Accordingly, for any such $m$, we have $q(m) \geq 0.49 (3 \alpha_n \log \log \rho_\ell)^{-1} \geq q_W$.
\end{proof}

Denote by $\tau^U_0$ and $\tau^W_0$ the first hitting times of zero for $\big( \big| U_{\eta_\ell (m) + 1}^i \big| \big)_{m \geq 0}$ and $(W_m)_{m\geq 0}$. Under the coupling, $\tau^W_0$ cannot precede $\tau_0^U$. So an upper bound on $\tau_0^W$ of $m$ implies $\tau_0^U \leq m$ and therefore it takes no more than $m$ occurrences of the midway event after $\TT_{\ell-1}$ for the collapse of the watched cluster to occur. In other words, $\TT_\ell - \TT_{\ell-1}$ is at most $\eta_\ell (m) + 1$.

\begin{lemma}\label{lem: rw lem2}
Let cluster $i$ be the watched cluster at time $\TT_{\ell-1}$. When $\Quick (\ell -1)$ and $\{\mm_\ell (\TT_\ell^-) \geq M\}$ occur,
\begin{equation}\label{eq: rw lem2}
\left\{ \tau^W_0 \leq M \right\} \subseteq \left\{ \TT_\ell - \TT_{\ell - 1} \leq \eta_\ell (M) + 1 \right\}.
\end{equation}
\end{lemma}
\begin{proof}
From Lemma~\ref{lem: rw lem1}, there is a coupling of $\big( \big| U_{\eta (m)+1}^i \big| \big)_{m \geq 0}$ and $(W_m)_{m \geq 0}$ such that, when $\Quick (\ell - 1)$ and $\{\mm_\ell (\TT_\ell^-) \geq M\}$ occur, $\big| U_{\eta (m)+1}^i \big| \leq W_m$ for all $m \leq M$. In particular,
\[\{\tau_0^W \leq M\} \subseteq \{\tau_0^U \leq M\}.\]

If $\{\tau_0^U \leq M\}$ occurs, cluster $i$ is empty after the time of the $M$\textsuperscript{th} occurrence of the midway event, $\eta_\ell (M) + 1$. That is, we have the inclusion 
\[\{\tau_0^U \leq M\} \subseteq \{\TT_\ell - \TT_{\ell -1} \leq \eta_\ell (M)+1\},\]
which implies \eqref{eq: rw lem2}.
\end{proof}

We now show that $\tau_0^W$, the hitting time of zero for $W_m$, is not more than $\log (\log \rho_\ell)^n$, up to a factor depending on $n$, with high probability. With more effort, we could prove a much better bound (in terms of dependence on $n$), but this improvement would not affect the conclusion of Proposition~\ref{prop: l clust col}. By Lemma~\ref{lem: rw lem2}, the bound on $\tau_0^W$ will imply a bound on $\TT_\ell - \TT_{\ell-1}$ in terms of $\eta_\ell (\cdot)$. For brevity, denote $\beta_n = (8 \alpha_n)^n$. 

\begin{lemma}\label{lem: rw lem3} Let cluster $i$ be the watched cluster at time $\TT_{\ell-1}$ and let $K \geq 1$. If $\Quick (\ell - 1)$ and $\{\mm_\ell (\TT_\ell^-) \geq K \cdot \beta_n (\log \log \rho_\ell)^n\}$ occur, then 
\begin{equation}\label{eq: rw lem3}
\PP \left( \TT_\ell - \TT_{\ell - 1} \leq \eta_\ell (K \cdot \beta_n (\log \log \rho_\ell)^n) + 1 \bigm\vert \UU_{\TT_{\ell - 1}}\right) \geq 1 - e^{- K}.
\end{equation}
\end{lemma}

The factor $(\log \log \rho_\ell)^n$ appears because $(W_m)_{m \geq 0}$ takes down-steps with a probability which is the reciprocal of $O_n(\log \log \rho_\ell)$, and we will require it to take $n$ consecutive down-steps. Note that the event involving $K$ cannot occur if $K$ is large enough, because $N_\ell (\TT_\ell^-)$ cannot exceed $\TT_\ell^- - \TT_{\ell-1} \leq (\log \rho_\ell)^2$ (i.e., there can be no more occurrences of the midway event than there are HAT steps). The implicit bound on $K$ is $(\log \rho_\ell)^{2-o_n(1)}$. We will apply the lemma with a $K$ of approximately $(\log \rho_\ell)^\delta$ for a $\delta \in (0,1)$.

\begin{proof}[Proof of Lemma \ref{lem: rw lem3}]
Set $M = K \cdot \lfloor \beta_n (\log  \log \rho_\ell)^n \rfloor$ and denote the distribution of $(W_m)_{m \geq 0}$ by $\P_W$. If $\Quick (\ell - 1)$ and $\{\mm_\ell (\TT_\ell^-) \geq M\}$ occur, then by Lemma~\ref{lem: rw lem2}, we have the inclusion \eqref{eq: rw lem2}:
\[ \{ \tau_0^W \leq M\} \subseteq \{ \TT_\ell - \TT_{\ell-1} \leq \eta_\ell (M) + 1\}.\]

Since $(W_m)_{m\geq 0}$ is never greater than $n$, it never takes more than $n$ down-steps for $W_m$ to hit zero. Since $W_{m+1} = W_m - 1$ with a probability of $q_W$ whenever $m \leq M-n$, we have
\[\P_W \left( \tau_0^W > m + n \bigm\vert \tau_0^W > m \right) \leq 1 - q_W^n.\] 
Applying this to all $m \leq M-n$, we find
\[\P_W \left( \tau_0^W > M \right) \leq \left( 1 - q_W^n \right)^{M} \leq e^{-K}.\] For the second inequality, we used the fact that $\lfloor \beta_n (\log \log \rho_\ell)^n \rfloor$ is at least $q_W^{-n}$ and therefore $M$ is at least $K \cdot q_W^{-n}$. Combining this with \eqref{eq: rw lem2} gives \eqref{eq: rw lem3}.
\end{proof}

To conclude Proposition~\ref{prop: l clust col} from Lemma~\ref{lem: rw lem3}, we will show that if $\TT_\ell - \TT_{\ell - 1}$ exceeds, say, $(\log \rho_\ell)^{1+2\delta}$, then with high probability there are many---at least $(\log \rho_\ell)^{\delta}$---occurrences of the midway event (and therefore steps of the walk $W_m$), with high probability, for an appropriate choice of $\delta$. Reflecting this aim, we define the event
\[ \Many_\delta = \left\{ \eta_\ell \big( (\log \rho_\ell)^{\delta} \big) \leq (\log \rho_\ell)^{1+2\delta} \right\}.\]
When $\Many_\delta$ occurs, we will find that $\TT_\ell - \TT_{\ell - 1} > (\log \rho_\ell)^{1+2\delta}$ is unlikely, as the walk $W_m$ will hit zero with high probability after $(\log \rho_\ell)^{\delta}$ steps.

For convenience, in what follows, we will treat terms of the form $(\log \rho_\ell)^\delta$ as integers, as the distinction will be immaterial.

\begin{proposition}\label{prop: tmid} Let cluster $i$ be the watched cluster at time $\TT_{\ell-1}$ and let $\delta = (4n)^{-2}$. If $\Quick (\ell - 1)$ and $\{\TT_\ell - \TT_{\ell - 1}  > (\log \rho_\ell)^{1+6\delta}\}$ occur, then
\begin{equation}\label{eq: tmid}
\PP \left( \Many_{3\delta} \bigm\vert \UU_{\TT_{\ell -1}}\right) \geq 1 - e^{-5n (\log \rho_\ell)^{2\delta}}.
\end{equation} 
\end{proposition}

\begin{proof}
By Proposition~\ref{prop: dep lower bd}, the estimate \eqref{eq: dep lower bd} holds for any $t \in [\TT_{\ell -1}, \TT_\ell^-]$. Accordingly, when $\left\{\TT_\ell - \TT_{\ell - 1} > (\log \rho_\ell)^{1+6\delta}\right\}$ occurs, \eqref{eq: dep lower bd} applies to every time $t$ up to $(\log \rho_\ell)^{1+6\delta}$:
\begin{equation}\label{eq: appdeplow}
\PP (\Mid (i,t;\ell) \bigm\vert \UU_t) \geq (\alpha_n \log \rho_\ell)^{-1}.
\end{equation}

Define the time
\[ \mathfrak{s}_\ell (\delta) = 6 n \alpha_n (\log \rho_\ell)^{1+2\delta}.\] 
Suppose that the number of occurrences $M$ of the midway event is such that the time $\eta_\ell (M) + \ss_\ell (\delta)$ is at most $(\log \rho_\ell)^{1+6\delta}$. We then define, for any $m \leq M$ the event that the $m$\textsuperscript{th} and $(m+1)$\textsuperscript{st}  occurrences of the midway event are ``close'' in time:
\[\Close_\delta (m) = \left\{ \eta_\ell (m+1) - \eta_\ell (m) \leq \ss_\ell (\delta) \right\}.\]

In order for $\Close_\delta (m)$ to fail to occur, we must fail to observe the occurrence of $\Mid (i,t;\ell)$ in $\mathfrak{s}_\ell (\delta)$-many consecutive steps. Using the Markov property and the bound \eqref{eq: appdeplow}, we find that
\begin{equation}\label{eq: tmid calc}
\PP \left( \Close_{\delta}(m)^c \bigm\vert \UU_{\eta_\ell (m)}\right) \leq \left( 1 - \frac{1}{ \alpha_n \log \rho_\ell} \right)^{\mathfrak{s_\ell} (\delta)} \leq e^{-6 n (\log \rho_\ell)^{2\delta}}.
\end{equation}

Denote $\Close_\delta = \bigcap_{m=0}^{(\log \rho_\ell)^{3\delta}-1} \Close_\delta (m)$. We claim that $\Close_\delta$ is a subset of $\Many_\delta$ and that
\begin{equation}\label{eq: tmid3}
\PP (\Close_\delta \bigm\vert \UU_{\TT_{\ell-1}}) \geq 1 - e^{-5n (\log \rho_\ell)^{2\delta}},
\end{equation}
which implies \eqref{eq: tmid}.

To prove the inclusion, we note that when $\Close_\delta$ occurs, because $\rho_\ell$ is at least $e^{\theta_{2n}}$ when $\Quick (\ell-1)$ occurs \eqref{eq: quick2}, we have
\[ \eta_\ell \left( (\log \rho_\ell)^{3\delta} \right) \leq (\log \rho_\ell)^{3\delta} \cdot \mathfrak{s}_\ell (\delta) \leq 6 n \alpha_n (\log \rho_\ell)^{1+5\delta} \leq (\log \rho_\ell)^{1+6\delta}.\]
Specifically, the first bound holds due to the definition of $\Close_\delta$; the second due to the definition of $\mathfrak{s}_\ell (\delta)$; and the third due because $6n \alpha_n \leq (\log \rho_\ell)^{\delta}$ when $\rho_\ell \geq e^{\theta_{2n}}$. This implies that $\Close_\delta$ is a subset of $\Many_\delta$.

To prove \eqref{eq: tmid3} we use a union bound over the $(\log \rho_\ell)^{3\delta}$-many constituent events of $\Close_\delta$ and \eqref{eq: tmid calc}, finding that
\[ \PP (\Close_\delta \bigm\vert \UU_{\TT_{\ell-1}}) \geq 1 - (\log \rho_\ell)^{3\delta} e^{-6n (\log \rho_\ell)^{2\delta}} \geq 1 - e^{-5n \log (\rho_\ell )^{2\delta}}.\]
\end{proof}

We now have all the inputs required to complete the proof of Proposition \ref{prop: l clust col}.
\begin{proof}[Proof of Proposition~\ref{prop: l clust col}] Let $\delta = (4n)^{-2}$. We will show that it is rare for $\TT_\ell - \TT_{\ell-1}$ to exceed $(\log \rho_\ell)^{1+6\delta}$ by arguing that, if it does, then with high probability there are many occurrences of the midway event---and correspondingly many steps of the coupled random walk---over which the coupled random walk must avoid hitting zero.

In terms of notation, we will call this rare event $F_\delta$:
\[F_\delta = \bigcap_{\ell=1}^{k-1} F_{\ell,\delta} \quad \text{where} \quad F_{\ell, \delta} = \left\{ \TT_\ell - \TT_{\ell - 1} \leq (\log \rho_\ell)^{1+6\delta} \right\}.\]
The event which bounds $\TT_\ell - \TT_{\ell-1}$ in terms of the number of occurrences of the midway event is
\[G_{\ell, \delta} = \left\{ \TT_\ell - \TT_{\ell - 1} \leq \eta_\ell ((\log \rho_\ell)^{3\delta}) +1\right\}.\]
The event $G_{\ell, \delta}$ will be probable because, by Lemma~\ref{lem: rw lem3}, $\TT_\ell - \TT_{\ell-1}$ typically does not exceed the time it takes for the midway event to occur $\beta_n ( \log \log \rho_\ell)^n = (\log \rho_\ell)^{o_n(1)}$ times.
We will also use the probable event that there are approximately as many occurrences of the midway event as Proposition~\ref{prop: dep lower bd} suggests there should be:
\[\Many_{\ell,\delta} = \left\{ \eta_\ell \big( (\log \rho_\ell)^{3\delta} \big) \leq (\log \rho_\ell)^{1+6\delta} \right\}.\]

We will be able to bound the probability of $F_{\ell,\delta}^c$ for each $\ell$ in terms of the probabilities of the rare events $G_{\ell,\delta}^c$ and $\Many_{\ell,\delta}^c$ because of the following inclusion:
\begin{equation}\label{eq: fgmany}
F_{\ell,\delta} \subseteq G_{\ell,\delta} \cap \Many_{\ell,\delta}.
\end{equation}
We will then apply Lemma~\ref{lem: rw lem3} and Proposition~\ref{prop: tmid} to bound the probabilities of $G_{\ell,\delta}^c$ and $\Many_{\ell,\delta}^c$. After bounding the probability of each event $F_{\ell,\delta}$, we will use a union bound to bound the probability of $F_\delta$.

Consider $\ell = 1$. (Assumptions of the occurrence of $\Quick (\ell-1)$ are satisfied automatically when $\ell=1$.) Due to \eqref{eq: fgmany},
\begin{equation}\label{eq: pf11}
\PP \left( F_{1,\delta}^c \bigm\vert \UU_0 \right) \leq \PP \left( F_{1,\delta}^c \cap G_{1,\delta}^c \cap \Many_{1,\delta} \bigm\vert \UU_0 \right) + \PP \left( F_{1,\delta}^c \cap \Many_{1,\delta}^c \bigm\vert \UU_0 \right).
\end{equation}
We apply Lemma~\ref{lem: rw lem3} to bound the first term on the right-hand side of \eqref{eq: pf11}. It is easy to check that, because $\rho_1$ is at least $e^{\theta_{2n}}$,
\begin{equation}\label{eq: sixn fact}
(\log \rho_1)^\delta > 5n \cdot \beta_n (\log \log \rho_1)^n.
\end{equation} Here, $\beta_n = (8\alpha_n)^n$ is the same quantity which appears in the statement of Lemma~\ref{lem: rw lem3}. By \eqref{eq: sixn fact}, the quantity $(\log \rho_1)^{3\delta}$ which appears in the definition of $G_{1,\delta}$ satisfies
\[(\log \rho_1)^{3\delta} > 5n \cdot \beta_n ( \log \log \rho_1)^n \cdot (\log \rho_1)^{2\delta}.\]
When $F_{1,\delta}^c \cap \Many_{1,\delta}$ occurs, there are at least $(\log \rho_1)^{3\delta}$ occurrences of the midway event. In the terminology of Lemma \ref{lem: rw lem3}, $N_1 (\TT_1^-) \geq (\log \rho_1)^{3\delta}$ which, by \eqref{eq: sixn fact}, means we can take $K$ as large as $5n (\log \rho_1)^{2\delta}$. We apply the bound of Lemma~\ref{lem: rw lem3} with $K = 5n (\log \rho_1)^{2\delta}$, finding that
\[ \PP \left( F_{1,\delta}^c \cap G_{1,\delta}^c \cap \Many_{1,\delta} \bigm\vert \UU_0 \right) \leq e^{-5n (\log \rho_1)^{2\delta}}.\]

Next, we can apply Proposition~\ref{prop: tmid} directly to the second term on the right-hand side of \eqref{eq: pf11}:
\[\PP \left( F_{1,\delta}^c \cap \Many_{1,\delta}^c \bigm\vert \UU_0 \right) \leq e^{-5n (\log \rho_1)^{2\delta}}.\]

Substituting the bounds for the terms in \eqref{eq: pf11}, we find
\[ \PP \left( F_{1,\delta}^c \bigm\vert \UU_0 \right) \leq 2e^{-5n (\log \rho_1)^{2\delta}}.\]

Continuing inductively, suppose $\bigcap_{i=1}^{\ell} F_{i,\delta}$ occurs. It is easy to show that 
\[ \bigcap_{i=1}^{\ell} F_{i,\delta} = \bigcap_{i=1}^{\ell} \left\{ \TT_\ell - \TT_{\ell - 1} \leq (\log \rho_\ell)^{1+6\delta} \right\} \subseteq \Quick (\ell).\] Accordingly, the hypotheses of Lemma~\ref{lem: rw lem3} and Proposition~\ref{prop: tmid} involving $\Quick (\ell-1)$ are satisfied.

By Lemma~\ref{lem: quick}, $\rho_{\ell+1}$ is at least $e^{\theta_{2n}}$, so \eqref{eq: sixn fact} holds analogously. Furthermore, by Lemma~\ref{lem: quick}, for any $\ell'$ up to $\ell + 1$, we have $\rho_{\ell'} \geq \rho_1/2$. An argument identical to the $\ell = 1$ case establishes
\[ \PP \left( F_{\ell+1,\delta}^c \bigm\vert \UU_{\TT_{\ell}} \right) \1 \big( \cap_{i=1}^{\ell} F_{i,\delta} \big) \leq 2e^{- 5n \log (\rho_{\ell+1})^{2\delta}} \leq 2e^{-5n \log (\rho_1/2)^{2\delta}} \leq e^{-3n (\log \rho_1)^{2\delta}}.\]

By a union bound and the preceding display,
\begin{align}\label{eq: fdelc}
\PP \left( F_\delta^c \bigm\vert \UU_0 \right)
& = \PP \left( \cup_{i=1}^{k-1} \left\{ \cap_{j=1}^{i-1} F_{j,\delta} \cap F_{i,\delta}^c \right\} \Bigm\vert \UU_0 \right) \nonumber\\
& \leq \sum_{i=1}^{k-1} \PP \left( \cap_{j=1}^{i-1} F_{j,\delta} \cap F_{i,\delta}^c \Bigm\vert \UU_0 \right) \nonumber\\
& \leq \sum_{i=1}^{k-1} \EE \left[ e^{-3n (\log \rho_1)^{2\delta}} \1 \left( {\cap_{j=1}^{i-1} F_{j,\delta}} \right) \Bigm\vert \UU_0 \right] \leq \sum_{i=1}^{k-1} e^{-3n (\log \rho_1)^{2\delta}} \leq e^{-2n (\log \rho_1)^{2\delta}}.
\end{align}

It remains to bound the time $\TT_{k-1}$ when $F_\delta$ occurs. One can show that, when $F_\delta$ occurs, $\rho_\ell$ is never more than twice the diameter $d$ of the initial configuration $U_0$. We write
\[ \TT_{k-1} = \sum_{\ell=1}^{k-1} (\TT_\ell - \TT_{\ell-1}) \leq \sum_{\ell=1}^{k-1} (\log \rho_\ell)^{1+6\delta} \leq 2n (\log d)^{1+6\delta} \leq (\log d)^{1+7\delta}.\] The preceding display and \eqref{eq: fdelc} establish \eqref{eq: l clust col}.
\end{proof}


\section{Existence of the stationary distribution}\label{sec: stat dist}

In this section, we will prove Theorem \ref{thm: stat dist}, which has two parts. The first part states the existence of a unique stationary distribution, $\pi_n$, supported on the equivalence classes of non-isolated configurations, $\wtnoniso (n)$, to which the HAT dynamics converges from any $n$-element configuration. The second part provides a tail bound on the diameter of configurations under $\pi_n$. We will prove these parts separately, as the following two propositions.

\begin{proposition}\label{prop: exist stat dist} For all $n\geq 1$, from any $n$-element subset $U$, HAT converges to a unique stationary distribution $\pi_n$ on $\wtnoniso (n)$, given by
\begin{equation}\label{eq: pos rec stat}
\pi_n ( \wt U) = \frac{1}{\EE_{\wt U} \TT_{\wt U}}, \,\,\, \text{for $\wt U \in \wtnoniso (n)$,}
\end{equation}
in terms of the return time $\TT_{\wt U} = \inf \{ t \geq 1: \wt U_t = \wt U\}$.
\end{proposition}

\begin{proposition}\label{prop: stat dist tail} For any $d \geq 2\theta_{4n}$,
\begin{equation}\label{eq: diam tail bd}
\pi_{n}\big( \diam (\wt U) \ge d\big)\le \exp \left( - \frac{d}{(\log d)^{1+o_n(1)}} \right).
\end{equation}
\end{proposition}

\begin{proof}[Proof of Theorem \ref{thm: stat dist}]
Combine Propositions \ref{prop: exist stat dist} and \ref{prop: stat dist tail}.
\end{proof}

It will be relatively easy to establish Proposition \ref{prop: stat dist tail} using the inputs to the proof of Proposition \ref{prop: exist stat dist} and Corollary \ref{cor: cons}, so we focus on presenting the key components of the proof of Proposition~\ref{prop: exist stat dist}.

By standard theory for countable state space Markov chains, to prove Proposition \ref{prop: exist stat dist}, we must prove that the HAT dynamics is positive recurrent, irreducible, and aperiodic. We address each of these in turn.


\begin{proposition}[Positive recurrent]\label{prop: pos rec} For any $U \in \noniso (n)$, $\EE_{\wt U} \TT_{\wt U} < \infty$.
\end{proposition}



To prove Proposition \ref{prop: pos rec}, we will estimate the return time to an arbitrary $n$-element configuration $\wt U$ by separately estimating the time it takes to reach the line segment $\wt L_n$ from $\wt U$, where $L_n = \left\{y \,\ee_2: y \in \{0, 1, \dots, n-1\} \right\}$, and the time it takes to hit $\wt U$ from $\wt L_n$. The first estimate is the content of the following result.

\begin{proposition}\label{prop: set to line} There is a constant $c$ such that, if $U$ is a configuration in $\noniso (n)$ with a diameter of $R$, then, for all $K \geq \max\{R,\theta_{4n} (cn)\}$,
\begin{equation}\label{eq: set to line}
\PP_{U} \left(\TT_{\wt L_n} \leq K^3 \right) \geq 1 - e^{-K}.
\end{equation}
\end{proposition}

The second estimate is provided by the next proposition.

\begin{proposition}\label{prop: line to set} There is a constant $c$ such that, if $U$ is a configuration in $\noniso (n)$ with a diameter of $R$, then, for all $K \geq \max\{e^{R^{2.1}},\theta_{4n} (cn)\}$,
\begin{equation}\label{eq: line to set}
\PP_{\wt L_n} \left(\TT_{\wt U} \leq K^5 \right) \geq 1 - e^{-K}.
\end{equation}
\end{proposition}

The proof of Proposition \ref{prop: pos rec} applies \eqref{eq: set to line} and \eqref{eq: line to set} to the tail sum formula for $\EE_{\wt U} \TT_{\wt U}$.

\begin{proof}[Proof of Proposition \ref{prop: pos rec}]
Let $U \in \noniso (n)$. We have
\begin{equation}\label{etailbd2}
\EE_{\wt U} \TT_{\wt U} = \sum_{t=0}^\infty \PP_{\wt U} \big( \TT_{\wt U} > t \big) \leq \sum_{t=0}^\infty \Big( \PP_{\wt U} \big( \TT_{\wt L_n} > \tfrac{t}{2} \big) + \PP_{\wt L_n} \big( \TT_{\wt U} > \tfrac{t}{2} \big) \Big).
\end{equation}
Suppose $U$ has a diameter of at most $R$ and let $J = \max\{e^{R^{2.1}},\theta_{4n} (cn)\}$, where $c$ is the larger of the constants from Propositions \ref{prop: set to line} and \ref{prop: line to set}. We group the sum \eqref{etailbd2} over $t$ into blocks:
\[ \EE_{\wt U} \TT_{\wt U} \leq O(J^5) + \sum_{K=J}^\infty \sum_{t= 2K^5}^{2(K+1)^5} \Big( \PP_{\wt U} \big( \TT_{\wt L_n} > \tfrac{t}{2} \big) + \PP_{\wt L_n} \big( \TT_{\wt U} > \tfrac{t}{2} \big) \Big).\]
By \eqref{eq: set to line} and \eqref{eq: line to set} of Propositions \ref{prop: set to line} and \ref{prop: line to set}, each of the $O(K^4)$ summands in the $K$\textsuperscript{th} block is at most
\begin{equation}\label{etailbd3}
\PP_{\wt U} \big( \TT_{\wt L_n} > K^5 \big) + \PP_{\wt L_n} \big( \TT_{\wt U} > K^5 \big) \leq 2e^{-K}.
\end{equation}
Substituting \eqref{etailbd3} into \eqref{etailbd2}, we find
\[ \EE_{\wt U} \TT_{\wt U} \leq O(J^5) + O(1) \sum_{K=J}^\infty K^4 e^{-K} < \infty.\]
\end{proof}

Propositions \ref{prop: set to line} and \ref{prop: line to set} also imply irreducibility.

\begin{proposition}[Irreducible]\label{prop: irred}
For any $n \geq 1$, HAT is irreducible on $\wtnoniso (n)$. 
\end{proposition}

\begin{proof}
Let $\wt U, \wt V \in \wtnoniso (n)$. It suffices to show that HAT reaches $\wt V$ from $\wt U$ in a finite number of steps with positive probability. By Propositions \ref{prop: set to line} and \ref{prop: line to set}, there is a finite number of steps $K = K(U, V)$ such that
\[ 
\PP_{ U} \big( \TT_{\wt L_n} < K \big) > 0 \quad \text{and} \quad \PP_{L_n} \big( \TT_{\wt V} < K \big) > 0.
\]
By the Markov property applied to $\TT_{\wt L_n}$, the preceding bounds imply that $\PP_{\wt U} (\TT_{\wt V} < 2K) > 0$.
\end{proof}


Lastly, because aperiodicity is a class property, it follows from irreducibility and the simple fact that $\wt L_n$ is aperiodic.

\begin{proposition}[Aperiodic]\label{prop: aper} $\wt L_n$ is aperiodic.
\end{proposition}
\begin{proof}
We claim that $\PP_{L_n} (U_1 = L_n) \geq \frac14$, which implies that $\PP_{\wt L_n} \big( \wt U_1 = \wt L_n \big) \geq \frac14 > 0$. Indeed, every element of $L_n$ neighbors another, so, regardless of which one is activated, we can dictate one random walk step which results in transport to the site of activation and $U_1 = L_n$.
\end{proof}

The preceding results constitute a proof of Proposition \ref{prop: exist stat dist}.

\begin{proof}[Proof of Proposition \ref{prop: exist stat dist}]
Combine Propositions \ref{prop: pos rec}, \ref{prop: irred}, and \ref{prop: aper}.
\end{proof}

The subsections are organized as follows. In Section \ref{sec7 prelims}, we prove some preliminary results, including a key lemma which states that it is possible to reach any configuration $U \in \noniso (n)$ from $L_n$, in a number of steps depending only on $n$ and $\diam (U)$. These results support the proofs of Propositions \ref{prop: set to line} and \ref{prop: line to set} in Sections \ref{sec7 prop74} and Sections \ref{sec7 prop75}, respectively. In Section \ref{sec7 tailbd}, we prove Proposition \ref{prop: stat dist tail}.



\subsection{Preliminaries of hitting estimates for configurations}\label{sec7 prelims}
The purpose of this section is to estimate the probability that HAT forms a given configuration $\wt V$ from $\wt L_n$. We accomplish this primarily through Lemma \ref{lem: line to set1}, which guarantees the existence of a sequence configurations from $\wt L_n$ to $\wt V$, which can be realized by HAT in a way which is amenable to estimates.

In this section, we will say that an element $x$ of a configuration $V$ is exposed if $\H_V (x) > 0$ and we will denote the exposed elements of a configuration $V$ by $\bdyexp V$. Additionally, we will denote the radius of a set $A$ by $\rad (A) = \sup \{|x|: x \in A\}$.

First, we have a consequence of Theorems~\ref{thm: hm} and \ref{thm: esc}.

\begin{lemma}\label{lem: simp hit} There is a constant $c$ such that, if $V_0$ is a subset of $\Z^2$ with $n \geq 2$ elements and a radius of at most $r > 1$, and if $V_1$ is such that $\PP_{V_0} (U_1 = V_1) > 0$, then
\begin{equation}\label{eq: simp hit1}
\PP_{V_0} (U_1 = V_1) \geq e^{-c n \log n} (\log r)^{-1} .
\end{equation}
\end{lemma}

\begin{proof}
We will prove \eqref{eq: simp hit1} by factoring $\PP_{V_0} (U_1 = V_1)$ into activation and transport components, and separately estimating the components with Theorems \ref{thm: hm} and \ref{thm: esc}.

Let $V_0$ and $V_1$ satisfy the hypotheses. Because $\PP_{V_0} (U_1 = V_1)$ is positive, there are exposed elements $x$ of $V_0$ and $y$ of $\bdy (V_0 {\setminus} \{x\})$ such that $V_1 = V_0 \cup \{y\} {\setminus} \{x\}$. Denote $W = V_0 {\setminus} \{x\}$. We write
\begin{equation}\label{eq: ppv0}
\PP_{V_0} (U_1 = V_1) \geq \H_{V_0} (x) \, \P_x \big( S_{\tau_{W} - 1} = y \big) \geq e^{-c_1 n \log n} \P_x \big( S_{\tau_{W} - 1} = y \big).
\end{equation}
Note that, for the first inequality to be an equality, we would need to sum the right-hand side over all $x,y$ such that $V_1 = V_0 \cup \{y\}{\setminus}\{x\}$. The second inequality is implied by \eqref{eq: ext hm thm} of Theorem \ref{thm: hm}, because $x$ is exposed in $V_0$, which has $n$ elements.

In terms of a distance $d$ (which we will specify shortly) and $\bdy W_d$, the exterior boundary of the $d$-fattening of $W$, we address the second factor of \eqref{eq: ppv0} as
\begin{equation}\label{eq: ppv02}
\P_x \big( S_{\tau_W - 1} = y \big) \geq \frac14 \P_x (\tau_{\bdy W_d} < \tau_{\bdy W}) \, \E_x \left[ \P_{S_{\tau_{\bdy W_d}}} (S_{\tau_{\bdy W}} = y) \bigm\vert \tau_{\bdy W_d} < \tau_{\bdy W} \right].
\end{equation} 
In words, the probability that a random walk from $x$ first steps into $W$ from $y$ is at least the probability that it does so after first reaching $\bdy W_d$. We choose this lower bound because the factors of \eqref{eq: ppv02} can be addressed by our escape probability and harmonic measure estimates. The factor of $\tfrac14$ arises from forcing the walk to hit $W$ in the next step, after reaching $y$ at time $\tau_{\bdy W}$.

To replace the hitting probability with harmonic measure, we recall a standard result. Theorem 2.1.3 of \cite{lawler2013intersections} states that there are constants $c_2$ and $m$ such that, if $A$ is a subset of $\Z^2$ contained in $D(r')$, if $z \in A$, and if $y \in D(m r')^c$, then
\[
\H_A (z,y) \geq c_2 \H_A (y).
\]
We apply this fact with $A = \bdy W$ and $r' = r$, where $r > 1$ is an upper bound on the radius of $V_0$. Note that $W$ and $\bdy W$ are contained in $D(r+1)$. Hence, if $d$ is at least $(m+1)(r+1)$, then $\bdy W_d$ is contained in $D(m (r+1))^c$. This implies
\begin{equation}\label{hv0 bd2}
\P_z (S_{\tau_{\bdy W}} = y) \geq c_2 \H_{\bdy W} (y) \geq e^{-c_3 n \log n} \,\,\, \text{for every $z \in \bdy W_d$}.
\end{equation}
The second inequality is implied by \eqref{eq: ext hm thm} of Theorem \ref{thm: hm}, because $y$ is exposed in a set of $|\bdy W| \leq 4n$ elements.

We will now use \eqref{eq: esc thm} of Theorem \ref{thm: esc} to bound the escape probability in \eqref{eq: ppv02}. Recall that if $A$ has at least two elements and if $d' \geq 2 \, \diam (A)$, then \eqref{eq: esc thm} states
\[ \P_x ( \tau_{\bdy A_{d'}} < \tau_A ) \geq \frac{c_4 \H_A (x)}{n \log (d')} \,\,\,\text{for every $x \in A$}.\]
We apply this fact with $A = \bdy W$ and $d' = 4d$ to find
\begin{equation}\label{eq: hv0 bd3}
\P_x (\tau_{\bdy W_d} < \tau_{\bdy W}) \geq \P_x (\tau_{\bdy A_{d'}} < \tau_A) \geq \frac{c_4 \H_{\bdy W} (x)}{2n \log (4d)} \geq e^{-c_5 n \log n} (\log d)^{-1}.
\end{equation}
The first inequality holds because $A$ has a diameter of at most $2(r+1)$ and so, if $d \geq (m+1)(r+1)$, then $d+2(r+1) \leq 4d$ and hence $\bdy W_d$ separates $A$ from $\bdy A_{4d}$. The second inequality is due to \eqref{eq: esc thm}, which applies because $4d \geq 2\, \diam (A)$. The third inequality is due to \eqref{eq: ext hm thm}, which applies because $x$ is exposed in $V_0$, an $n$-element set.

Substituting \eqref{hv0 bd2} and \eqref{eq: hv0 bd3} into \eqref{eq: ppv02}, and replacing $d$ with $(m+1)(r+1)$, we find
\[ \P_x \big( S_{\tau_W - 1} = y \big) \geq e^{-c_6 n\log n} (\log r)^{-1}.\] Lastly, applying this bound to \eqref{eq: ppv0}, we find \eqref{eq: simp hit1}:
\[\PP_{V_0} (U_1 = V_1) \geq e^{-c_7 n \log n} (\log r)^{-1}.\]
\end{proof}

The preceding lemma will help us bound below the probability of realizing a given configuration $V$ as $U_t$ for some time $t$ and from some initial configuration $V_0$. However, to apply the lemma, we need an upper bound on the number of HAT steps it takes to form $V$ from $V_0$. Supplying such an upper bound is the purpose of the next result, which is a key input to the proof of Proposition~\ref{prop: line to set}.

\begin{lemma}\label{lem: line to set1}
For any number of elements $n \geq 2$ and configuration $V$ in $\noniso (n)$, if the radius of $V$ is at most an integer $r \geq 10 n$, then there is a sequence of $k \leq 100 nr$ activation sites $x_1, \dots, x_k$ and transport sites $y_1, \dots, y_k$ which can be ``realized'' by HAT from $V_0 = L_n$ to $V_k = V$ in the following sense: if we set $V_i = V_{i-1} \cup \{y_i\} {\setminus} \{x_i\}$ for each $i \in \{1,\dots,k\}$, then each transition probability $\PP_{V_{i-1}} (U_i = V_i)$ is positive. Additionally, each $V_i$ is contained in $D(r+10n)$.
\end{lemma}

The factors of $10$ and $100$ in the lemma statement are for convenience and have no further significance. 
We will prove Lemma~\ref{lem: line to set1} by induction on $n$. Informally, we will remove one element of $L_n$ to facilitate the use of the induction hypothesis, forming most of $V$ before returning the removed element. There is a complication in this step, as we cannot allow the induction hypothesis to ``interact'' with the removed element. We will resolve this problem by proving a slightly stronger claim than the lemma requires.

The proof will overcome two main challenges. First, removing an element from a configuration $V$ in $\noniso (n)$ can produce a configuration in $\iso (n-1)$, in which case the induction hypothesis will not apply. Indeed, there are configurations of $\noniso (n)$ for which the removal of any exposed, non-isolated element produces a configuration of $\iso (n-1)$ (such a $V$ is depicted in Figure \ref{fig: induc1}). Second, if an isolated element is removed alone, it cannot be returned to form $V$ by a single step of the HAT dynamics. To see how these difficulties interact, suppose $\bdyexp V$ contains only one non-isolated element (say, at $v$), which is part of a two-element connected component of $V$. We cannot remove it and still apply the induction hypothesis, as $V{\setminus}\{v\}$ belongs to $\iso (n-1)$. We then have no choice but to remove an isolated element.

When we are forced to remove an isolated element, we will apply the induction hypothesis to form a configuration for which the removed element can be ``treadmilled'' to its proper location, chaperoned by a element which is non-isolated in the final configuration and so can be returned once the removed element reaches its destination.

We briefly explain what we mean by treadmilling a pair of elements. Consider elements $v_1$ and $v_1 + e_2$ of a configuration $V$. If $\H_V (v_1)$ is positive and if there is a path from $v_1$ to $v_1+2e_2$ which lies outside of $V{\setminus}\{v_1\}$, then we can activate at $v_1$ and transport to $v_1+2e_2$. The result is that the pair $\{v_1,v_1+2e_2\}$ has shifted by $e_2$. Call the new configuration $V'$. If $v_1+e_2$ is exposed in $V'$ and if there is a path from $v_1+e_2$ to $v_1+3e_2$ in $V'{\setminus}\{v_1+e_2\}$, we can analogously shift the pair $\{v_1+e_2,v_1+2e_2\}$ by another $e_2$.

\begin{proof}[Proof of Lemma~\ref{lem: line to set1}]
The proof is by induction on $n \geq 2$. We will actually prove a stronger claim, because it facilitates the induction step. To state the claim, we denote by $W_i = V_{i-1}{\setminus}\{x_i\}$ the HAT configuration ``in between'' $V_{i-1}$ and $V_i$ and by $E_i$ the event that, during the transition from $V_{i-1}$ to $V_i$, the transport step takes place inside of $B_i = D(r+10n){\setminus} W_i$:
\[ E_i = \big\{ \{ S_0, \dots, S_{\tau_{W_i}} \} \subseteq B_i \big\}. \] 
We claim that Lemma \ref{lem: line to set1} is true even if the conclusion $\PP_{V_{i-1}} (U_i = V_i) > 0$ is replaced by $\PP_{V_{i-1}} (U_i = V_i, E_i) > 0$.

To prove this claim, we will show that, for any $V$ satisfying the hypotheses, there are sequences of at most $100nr$ activation sites $x_1, \dots, x_k$, transport sites $y_1, \dots, y_k$, and random walk paths $\Gamma^1, \dots, \Gamma^k$ such that the activation and transport sites can be realized by HAT from $V_0 = L_n$ to $V_k = V$, and such that each $\Gamma^i$ is a finite random walk path from $x_i$ to $y_i$ which lies in $B_i$. While it is possible to explicitly list these sequences of sites and paths in the proof which follows, the depictions in upcoming Figures \ref{fig: induc1} and \ref{fig: induc2} are easier to understand and so we omit some cumbersome details regarding them. 

Concerning the base case of $n=2$, note that $\noniso (2)$ has the same elements as the equivalence class $\wt L_2$, so $x_1 = e_2$, $y_1 = e_2$, $\Gamma^1 = \emptyset$ works. Suppose the claim holds up to $n-1$ for $n \geq 3$. 
There are two cases:
\begin{enumerate}
\item There is a non-isolated $v$ in $\bdyexp V$ such that $V {\setminus} \{v\}$ belongs to $\noniso (n-1)$.
\item For every non-isolated $v$ in $\bdyexp V$, $V{\setminus} \{v\}$ belongs to $\iso (n-1)$.
\end{enumerate}
It will be easy to form $V$ using the induction hypothesis in Case 1. In Case 2, we will need to use the induction hypothesis to form a set related to $V$, and subsequently form $V$ from this related set. An instance of Case 2 is depicted in Figure \ref{fig: induc1}.

\begin{figure}[ht]
\centering {\includegraphics[width=0.6\linewidth]{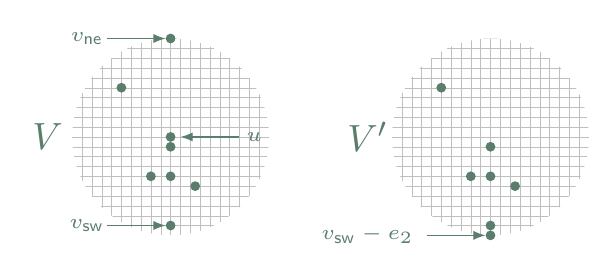}} 
\caption{An instance of Case 2. If any non-isolated element of $\bdyexp V$ is removed, the resulting set is isolated. We use the induction hypothesis to form $V' = (V{\setminus}\{v_{\mathsf{ne}},u\})\cup\{v_{\mathsf{sw}}-e_2\}$. The subsequent steps to obtain $V$ from $V'$ are depicted in Figure \ref{fig: induc2}.}
\label{fig: induc1}
\end{figure}

\textbf{Case 1}. Let $r$ be an integer exceeding $10n$ and the radius of $V$ and denote $R = r + 10(n-1)$. Recall that $V_0 = L_n$. Our strategy is to place one element of $L_n$ outside of $D(R)$ and then apply the induction hypothesis to $L_{n-1}$ to form most of $V$. This explains the role of the event $E_i$---it ensures that the element outside of the disk does not interfere with our use of the induction hypothesis.

To remove an element of $L_n$ to $D(R)^c$, we treadmill (see the explanation following the lemma statement) the pair $\{(n-2)\ee_2, (n-1)\ee_2\}$ to $\{Re_2,(R+1)e_2\}$, after which we activate at $Re_2$ and transport to $(n-2)\ee_2$. This process requires $R-n+2$ steps. It is clear that every transport step can occur via a finite random walk path which lies in $D(r+10n)$. Call $a = (R+1)e_2$. The resulting configuration is $L_{n-1} \cup \{a\}$.

We will now apply induction hypothesis. Choose a non-isolated element $v$ of $\bdyexp V$ such that $V' = V {\setminus} \{v\}$ belongs to $\noniso (n-1)$. Such a $v$ exists because we are in Case 1. By the induction hypothesis and because the radius of $V'$ is at most $r$, there are sequences of at most $100 (n-1) r$ activation and transport sites, which can be realized by HAT from $L_{n-1} \cup \{a\}$ to $V' \cup \{a\}$, and a corresponding sequence of finite random walk paths which lie in $D(R)$.

To complete this case, we activate at $a$ and transport to $v$, which is possible because $v$ was exposed and non-isolated in $V$. The existence of a random walk path from $a$ to $v$ which lies outside of $V'$ is a consequence of Lemma \ref{lem: n path}. Recall that Lemma \ref{lem: n path} applies only to sets in $\mathscr{H}_n$ ($n$-element sets which contain an exposed origin). If $A = V \cup \{a\}$, then $A-v$ belongs to $\mathscr{H}_n$. By Lemma \ref{lem: n path}, there is a finite random walk path from $a$ to $v$ which does not hit $V'$ and which is contained in $D(R+3) \subseteq D(r+10n)$.

In summary, there are sequences of at most $(R-n+2) + 100 (n-1) r + 1 \leq 100 n r$ (the inequality follows from the assumption that $r \geq 10n$) activation and transport sites which can be realized by HAT from $L_n$ to $V$, as well as corresponding finite random walk paths which remain within $D(r+10n)$. This proves the claim in Case 1.

\begin{figure}[ht]
\centering {\includegraphics[width=1\linewidth]{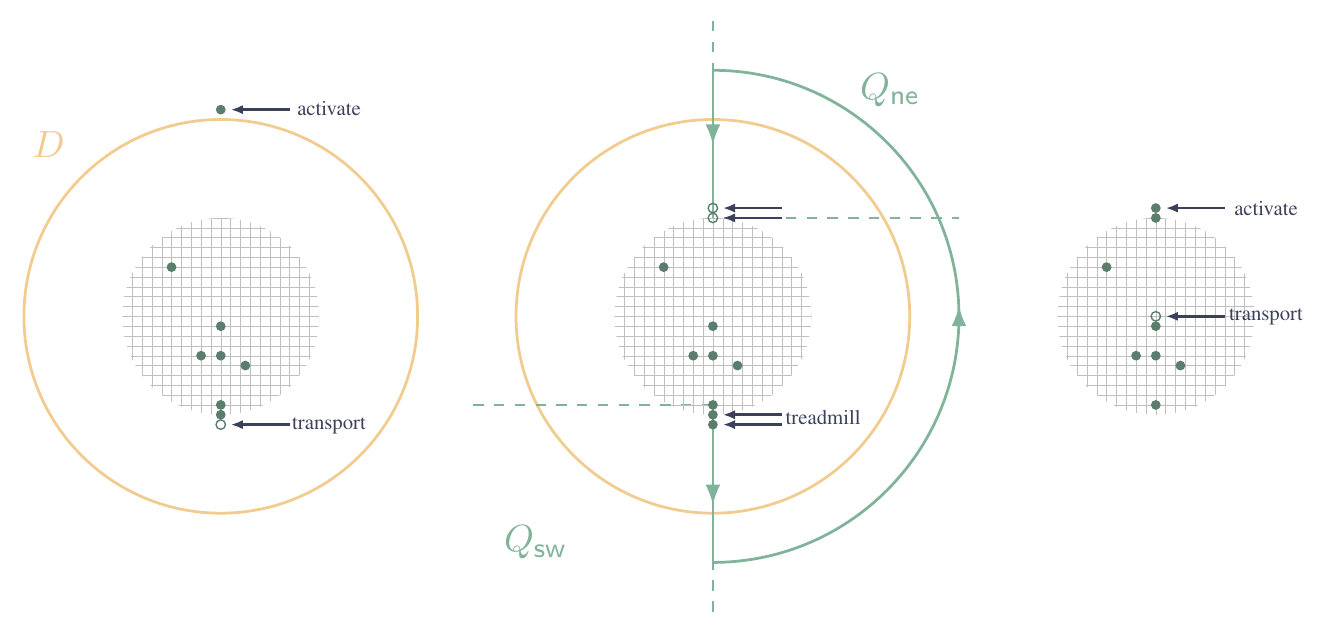}} 
\caption{An instance of Case 2 (continued). On the left, we depict the configuration which results from the use of the induction hypothesis. The element outside of the disk $D$ (the boundary of which is the orange circle) is transported to $v_{\mathsf{sw}}-2e_2$ (unfilled circle). In the middle, we depict the treadmilling of the pair $\{v_{\mathsf{sw}}-e_2,v_{\mathsf{sw}}-2e_2\}$ through the quadrant $Q_{\mathsf{sw}}$, around $D^c$, and through the quadrant $Q_{\mathsf{ne}}$, until one of the treadmilled elements is at $v_{\mathsf{ne}}$. The quadrants are depicted by dashed lines. On the right, the other element is returned to $u$ (unfilled circle). The resulting configuration is $V$ (see Figure \ref{fig: induc1}).}
\label{fig: induc2}
\end{figure}

\textbf{Case 2}. In this case, the removal of any non-isolated element $v$ of $\bdyexp V$ results in an isolated set $V {\setminus} \{v\}$, hence we cannot form such a set using the induction hypothesis. Instead, we will form a related, non-isolated set.

The first $R-n+2$ steps, which produce $L_{n-1} \cup \{a\}$ from $L_n$, are identical to those of Case 1. We apply the induction hypothesis to form the set
\[ V' = (V {\setminus} \{v_{\textsf{ne}},u\}) \cup \{v_{\textsf{sw}}-\ee_2\},\]
which is depicted in Figure \ref{fig: induc1}. Here, $v_{\textsf{ne}}$ is the easternmost of the northernmost elements of $V$, $v_{\textsf{sw}}$ is the westernmost of the southernmost elements of $V$, and $u$ is any non-isolated element of $\bdyexp V$ (e.g., $u = v_{\textsf{ne}}$ is allowed if $v_{\textsf{ne}}$ is non-isolated). 

The remaining steps are depicted in Figure \ref{fig: induc2}. By the induction hypothesis and because the radius of $V'$ is at most $r+1$, there are sequences of at most $100 (n-1) (r+1)$ activation and transport sites, which can be realized by HAT from $L_{n-1} \cup \{a\}$ to $V' \cup \{a\}$, and a corresponding sequence of finite random walk paths which lie in $D(R+1)$.

Next, we activate at $a$ and transport to $v_{\mathsf{sw}}-2\ee_2$, which is possible because $v_{\mathsf{sw}}-2\ee_2$ is exposed and non-isolated in $V'$. Like in Case 1, the existence of a finite random walk path from $a$ to $v_{\mathsf{sw}}-2\ee_2$ which lies in $D(R+3){\setminus}V' \subseteq D(r+10n)$ is implied by Lemma \ref{lem: n path}. Denote the resulting configuration by $V''$.

The choice of $v_{\mathsf{sw}}$ ensures that $v_{\mathsf{sw}} - \ee_2$ and $v_{\mathsf{sw}} - 2\ee_2$ are the only elements of $V''$ which lie in the quadrant defined by
\[ Q_{\mathsf{sw}} = (v_{\mathsf{sw}}-\ee_2) + \{v \in \Z^2: v \cdot \ee_1 \leq 0,\,\,v \cdot \ee_2 \leq 0\}.\] Additionally, the quadrant defined by 
\[Q_{\mathsf{ne}} = v_{\mathsf{ne}} + \{v \in \Z^2: v \cdot \ee_1 \geq 0,\,\,v \cdot \ee_2 \geq 0\}\]
contains no elements of $V''$. As depicted in Figure \ref{fig: induc2}, this enables us to treadmill the pair $\{v_{\mathsf{sw}} - \ee_2, v_{\mathsf{sw}} - 2\ee_2\}$ from $Q_{\mathsf{sw}}$ to $D(R+3)^c$ and then to $\{v_{\mathsf{ne}},v_{\mathsf{ne}}+\ee_2\}$ in $Q_{\mathsf{ne}}$, without the pair encountering the remaining elements of $V''$. It is clear that this can be accomplished by fewer than $10(R+3)$ activation and transport sites, with corresponding finite random walk paths which lie in $D(R+6)$. The resulting configuration is $V''' = V\cup\{v_{\mathsf{ne}}+e_2\}{\setminus}\{u\}$.

Lastly, we activate at $v_{\mathsf{ne}}+\ee_2$ and transport to $u$, which is possible because the former is exposed in $V'''$ and the latter is exposed and non-isolated in $V$. As before, the fact that there is a finite random walk path in $D(r+10n)$ which accomplishes the transport step is a consequence of Lemma \ref{lem: n path}. The resulting configuration is $V$.

In summary, there are sequences of fewer than $(R-n+2)+100(n-1)(r+1)+10(R+3)+2 \leq 100nr$ (the inequality follows from the assumption that $r \geq 10n$) activation and transport sites which can be realized by HAT from $L_n$ to $V$, as well as corresponding finite random walk paths which remain in $D(r+10n)$. This proves the claim in Case 2.
\end{proof}

We can combine Lemma~\ref{lem: simp hit} and Lemma~\ref{lem: line to set1} to bound below the probability of forming a configuration from a line.

\begin{lemma}\label{lem: quickv}
There is a constant $c$ such that, if $V$ is a configuration in $\noniso (n)$ with $n \geq 2$ and a diameter of at most $R \geq 10n$, then
\begin{equation*}
\PP_{\wt L_n} \left( \TT_{\wt V} \leq 200 nR \right) \geq e^{-c n^3 R^2}.
\end{equation*}
\end{lemma}

\begin{proof}
The hypotheses of Lemma~\ref{lem: line to set1} require an integer upper bound $r$ on the radius of $V$ of at least $10n$. We are free to assume that $V$ contains the origin, in which case a choice of $r = \lfloor R \rfloor + 1$ works, due to the assumption $R \geq 10n$. We apply Lemma \ref{lem: line to set1} with $r$ to find that there is a sequence of configurations $V_0 = L_n, V_1, \dots, V_{k-1}, V_k = V$ such that $k \leq 100 nr$, and such that $V_i \subseteq D(r+10n)$ and $\PP_{V_{i-1}} (U_i = V_i) > 0$ for each $i$.

Because the transition probabilities are positive and because they concern sets $V_{i-1}$ in the disk of radius $r+10n$, Lemma \ref{lem: simp hit} implies that each transition probability is at least
\[e^{-c_1 n \log n} (\log (r+10n))^{-1} \geq e^{-c_2 n^2 R}\] 
for a constant $c_1$. The inequality follows from coarse bounds of $n\log n = O(n^2)$ and $\log r = O(R)$. We use this fact in the following string of inequalities:
\[ \PP_{L_n} (\TT_{\wt V} \leq 200nR) \geq \PP_{L_n} (\TT_{\wt V} \leq k) \geq \PP_{L_n} \big( \wt U_{k} = \wt V \big) \geq e^{-100nr \cdot c_2 n^2 R} \geq e^{-c_3 n^3 R^2}.\]
The first inequality holds because $k \leq 100nr \leq 200nR$; the second because $\{\wt U_k = \wt V\} \subseteq \{\TT_{\wt V} \leq k\}$; the third follows from the Markov property, $k \leq 100nr$, and the preceding bound from Lemma \ref{lem: simp hit}; the fourth from $100 nr \leq 200 nR$.
\end{proof}

\subsection{Proof of Proposition~\ref{prop: set to line}}\label{sec7 prop74}

We now use Lemma~\ref{lem: simp hit} to obtain a tail bound on the time it takes for a given configuration to reach $ \wt L_n$. Our strategy is to repeatedly attempt to observe the formation of $\wt L_n$ in $n$ consecutive steps. If the attempt fails then, because the diameter of the resulting set may be larger---worsening the estimate \eqref{eq: simp hit1}---we will wait until the diameter becomes smaller before the next attempt.


\begin{proof}[Proof of Proposition \ref{prop: set to line}]
To avoid confusion of $U$ and $U_t$, we will use $V_0$ instead of $U$. We introduce a sequence of times, with consecutive times separated by at least $n$ steps and at which the diameter of the configuration is at most $\theta_1 = \theta_{4n} (c_1 n)$ (where $c_1$ is the constant in Corollary~\ref{cor: cons}). These will be the times at which we attempt to observe the formation of $\wt L_n$.

The reason for requiring that the consecutive times be separated by at least $n$ steps is because that is the number of steps it takes 
Define $\eta_0 = \inf \{t \geq 0: \diam (U_t) \leq \theta_1\}$ and, for all $i \geq 1$, the times
\[ \eta_i = \inf \{t \geq \eta_{i-1} + n : \diam (U_t) \leq \theta_1\}.\]

We use these times to define three events. Two of the events use a parameter $K$ which we assume is at least the maximum of $R$ and $\theta_2$, where $\theta_2$ equals $\theta_{4n} (c n)$ with $c = c_1 + 2c_2$ and $c_2$ is the constant guaranteed by Lemma~\ref{lem: simp hit}. (The constant $c$ is the one which appears in the statement of the proposition.) In particular, $K$ is at least the maximum diameter $\theta_1 + n$ of a configuration at time $\eta_{i-1} + n$.

The first is the event that it takes an unusually long time for the diameter to fall below $\theta_1$ for the first time:
\[ E_1 (K) = \left\{\eta_0 > 3K \big( \log (3K) \big)^{1+n^{-1}}\right\}.\]
The second is the event that an unusually long time elapses between $\eta_{i-1} + n$ and $\eta_{i}$ for some $1 \leq i \leq m$:
\[ E_2 (m,K) = \bigcup_{i=1}^{m} \left\{ \eta_i - (\eta_{i-1}+n) > 3K \big(\log (3K) \big)^{1+n^{-1}}\right\}.\] The third is the event that we do not observe the formation of $\wt L_n$ in $m \geq 1$ attempts: 
\[ E_3 (m) = \bigcap_{i=1}^{m} \left\{ \TT_{\wt L_n} > \eta_{i-1} + n\right\}.\]

Call $E (m,K) = E_1 (K) \cup E_2(m,K) \cup E_3 (m)$. When none of these events occur, we can bound $\TT_{\wt L_n}$:
\begin{align}\label{eq: ttbd1}
\TT_{\wt L_n} \1_{E(m,K)^c} & \leq \left(\eta_0 + \sum_{i=1}^m (\eta_i - (\eta_{i-1} + n))\right) \1_{E(m,K)^c} + n (m+1)\nonumber\\ 
& \leq 3K \big( \log (3K) \big)^{1+n^{-1}} + 3 m K \big(\log (3K) \big)^{1+n^{-1}} + n (m+1).
\end{align}

We will show that if $m$ is taken to be $3K (\log \theta_2)^n$, then $\PP_{V_0} (E (m,K))$ is at most $e^{-K}$. Substituting this choice of $m$ into \eqref{eq: ttbd1} and using $(\log \theta_2)^{2n} \leq \theta_2 \leq K$ to simplify, we obtain a further upper bound of 
\begin{equation}\label{eq: ttbd2}
\TT_{\wt L_n} \1_{E(m,K)^c} \leq K^3.
\end{equation}

By \eqref{eq: ttbd2}, if we show $\PP_{V_0} (E(m,K)) \leq e^{-K}$, then we are done. We start with a bound on $\PP_{V_0} (E_1 (K))$. Applying Corollary~\ref{cor: cons} with $3K$ in the place of $t$, $r$ in the place of $d$, and $3K = \max\{3K,R\}$ in the place of $\max\{t,d\}$, gives
\begin{equation}\label{eq: e1bd}
\PP_{V_0} (E_1 (K)) \leq e^{-3K}.
\end{equation}

We will use Corollary~\ref{cor: cons} and a union bound to bound $\PP_{V_0} (E_2(m,K))$. Because diameter grows at most linearly in time, the diameter of $U_{\eta_{i-1} + n} \in \UU_{\eta_{i-1}}$ is at most $\theta_1+n \leq 3K$. Consequently, Corollary~\ref{cor: cons} implies
\begin{equation}\label{eq: e2ub1}
\PP_{V_0} \left(\eta_i - (\eta_{i-1} + n) > 3 K \big(\log (3K) \big)^{1+n^{-1}} \Bigm\vert \UU_{\eta_{i-1}+n} \right) \leq e^{-3K}.
\end{equation} 
A union bound over the constituent events of $E_2 (m,K)$ and \eqref{eq: e2ub1} give 
\begin{equation}\label{eq: e2ub2}
\PP_{V_0} (E_2 (m,K)) \leq me^{-3K}.
\end{equation}

To bound the probability of $E_3 (m)$, we will use Lemma~\ref{lem: simp hit}. First, we need to identify a suitable sequence of HAT transitions. For any $0 \leq j \leq m-1$, given $\UU_{\eta_j}$, set $V_0' = U_{\eta_j} \in \UU_{\eta_j}$. 
There are pairs $\{(x_i, y_i):\,1 \leq i \leq n\}$ such that, setting $V_i' = V_{i-1}' \cup \{y_i\} {\setminus} \{x_i\}$ for $1 \leq i \leq n$, each transition probability $\PP_{V_{i-1}'} (U_i = V_i')$ is positive and $V_n' \in \wt L_n$. 
By Lemma~\ref{lem: simp hit}, each transition probability is at least
\begin{equation}\label{eq: lem631}
\PP_{V_{i-1}'} (U_i = V_i') \geq e^{-c_2 n \log n} \big(\log (\theta_1+n) \big)^{-1} \geq (\log \theta_2)^{-1}.
\end{equation}
For the first inequality we used the fact that the diameter of $V_0'$ is at most $\theta_1$, so after $i \leq n$ steps it is at most $\theta_1+n$.

By the strong Markov property and \eqref{eq: lem631},
\begin{align}\label{eq: lem632}
\PP_{V_0} \left(\TT_{\wt L_n} \leq \eta_j + n \Bigm\vert \UU_{\eta_j} \right)
& \geq \PP_{V_0} \left(U_{\eta_j + 1} = V_1',\dots, U_{\eta_j + n} = V_n' \Bigm\vert \UU_{\eta_j}\right) \nonumber \\
& \geq \prod_{i=1}^n \PP_{V_{i-1}'} (U_i = V_i') \geq (\log \theta_2)^{-n}.
\end{align}
Because $E_3 (j) \in \UU_{\eta_j}$, \eqref{eq: lem632} implies
\begin{equation}\label{eq: e3bd1}
\PP_{V_0} \left( \TT_{\wt L_n} \leq \eta_j + n \Bigm\vert E_3 (j) \right) \geq (\log \theta_2)^{-n}.
\end{equation}

Using \eqref{eq: e3bd1}, we calculate
\begin{equation}\label{eq: e3bd2}
\PP_{V_0} (E_3 (m))
= \prod_{j=0}^{m-1} \PP_{V_0} \left(\TT_{\wt L_n} > \eta_{j} + n \Bigm\vert E_3 (j) \right) \leq \prod_{j=0}^{m-1} \left(1 - (\log \theta_2)^{-n} \right) \leq e^{-3K}.
\end{equation}	

Combining \eqref{eq: e1bd}, \eqref{eq: e2ub2}, and \eqref{eq: e3bd2}, and simplifying using the fact that $K \geq \theta_2$, we find
\[\PP_{V_0} (E(m,K)) \leq (m+2)e^{-3K} \leq e^{-K}.\]
\end{proof}

\subsection{Proof of Proposition \ref{prop: line to set}}\label{sec7 prop75}
To prove this proposition, we will attempt to observe the formation of $\wt U$ from $\wt L_n$ and wait for the set to collapse if its diameter becomes too large, as we did in proving Proposition~\ref{prop: set to line}. However, there is an added complication: at the time that the set collapses, it does not necessarily form $\wt L_n$, so we will need to use Proposition~\ref{prop: set to line} to return to $\wt L_n$ before another attempt at forming $\wt U$. For convenience, we package these steps together in the following lemma.
\begin{lemma}\label{lem: coltoline}
There is a constant $c$ such that, if $V_0$ is a configuration in $\noniso (n)$ with a diameter of $R$, then for any $K \geq \max\{R, \theta_{4n} (cn)\}$,
\begin{equation}\label{eq: coltoline}
\PP_{V_0} \left( \TT_{\wt L_n} \leq 9K^3 \right) \geq 1 - e^{-K}.
\end{equation}
\end{lemma}

\begin{proof}
Call $\theta = \theta_{4n} (c n)$ where $c$ is the constant guaranteed by Proposition~\ref{prop: set to line}. First, we wait until the diameter falls to $\theta$. By Corollary~\ref{cor: cons},
\begin{equation}\label{eq: tt10}
\PP_{V_0} \left(\TT (\theta) \leq 2 K \big( \log (2K) \big)^{1+n^{-1}} \right) \geq 1 - e^{-2K}.
\end{equation} Second, from $U_{\TT (\theta)}$, we wait until the configuration forms a line. By Proposition~\ref{prop: set to line}, for any $K \geq \theta$,
\begin{equation}\label{eq: tt11}
\PP_{U_{\TT (\theta)}} \left( \TT_{\wt L_n} \leq 8 K^3 \right) \geq 1 - e^{-2K}.
\end{equation}
Simplifying with $K \geq \theta$, we have 
\[ 2 K \big( \log (2K) \big)^{1+n^{-1}} + 8K^3 \leq 9K^3.\]
Combining this bound with \eqref{eq: tt10} and \eqref{eq: tt11} gives \eqref{eq: coltoline}.
\end{proof}



\begin{proof}[Proof of Proposition \ref{prop: line to set}] We will use $V$ to denote the target configuration instead of $U$, to avoid confusion with $U_t$. Recall that, for any configuration $V$ in $\noniso (n)$ with a  diameter upper bound of $r \geq 10n$, Lemma~\ref{lem: quickv} gives a constant $c_1$ such that
\begin{equation*}
\PP_{\wt L_n} (\TT_{\wt V} \leq 200nr) \geq e^{-c_1 n^3 r^2}.
\end{equation*} Since $10nR \geq 10n$ is a diameter upper bound on $V$, we can apply the preceding inequality with $r = 10nR$:
\begin{equation}\label{eq: app qv}
\PP_{\wt L_n} \big(\TT_{\wt V} \leq 2000n^2R\big) \geq e^{-c_1 n^4 R^2}.
\end{equation}

With this result in mind, we denote $k = 2000n^2 R$ and define a sequence of times by
\[\zeta_0 \equiv 0 \quad \text{and} \quad \zeta_i = \inf \{t \geq \zeta_{i-1} + k: \wt U_t = \wt L_n\} \quad \text{for all $i \geq 1$}.\] Here, the buffer of $k$ steps is the period during which we attempt to observe the formation of $V$. After each failed attempt, because the diameter increases by at most one with each step, the diameter of $U_{\zeta_i + k}$ may be no larger than $k + n$.

We define two rare events in terms of these times and a parameter $K$, which we assume to be at least $\max\{e^{R^{2.1}},\theta_{4n} (c_2n)\}$, where $c_2$ is the greater of $c_1$ and the constant from Lemma~\ref{lem: coltoline}. In particular, under this assumption, $K$ is greater than $e^{4c_1 n^4 R^2}$ and $k+n$---a fact we will use later.

The first rare event is the event that an unusually long time elapses between $\zeta_{i-1} + k$ and $\zeta_i$, for some $i \leq m$:
\[F_1 (m,K) = \bigcup_{i=1}^m \left\{ \zeta_i - (\zeta_{i-1} + k) > 72 K^3 \right\}.\]
The second is the event that we do not observe the formation of $\wt V$ in $m \geq 1$ attempts:
\[ F_2 (m) = \bigcap_{i=1}^m \left\{ \TT_{\wt V} > \zeta_{i-1} + k\right\}.\]
Call $F (m,K) = F_1 (m,K) \cup F_2 (m)$. When $F (m,K)^c$ occurs, we can bound $\TT_{\wt V}$ as
\begin{align}\label{eq: tvbd1}
\TT_{\wt V} \1_{F(m,K)^c}
&= \sum_{i=0}^{m-1} (\zeta_i - (\zeta_{i-1} + k)) \1_{E(m,K)^c} + mk \leq 72 m K^3 + mk.
\end{align} We will show that if $m$ is taken to be $2K e^{c_1 n^4 R^2}$, then $\PP_{\wt L_n} (F(m,K))$ is at most $e^{-K}$. Substituting this value of $m$ into \eqref{eq: tvbd1} and simplifying with $K \geq k$ and then $K \geq e^{4c_1n^4 R^2}$ gives
\begin{equation}\label{eq: tvbd2}
\TT_{\wt V} \1_{F(m,K)^c} \leq K^4 e^{2c_1 n^4 R^2} \leq K^5.
\end{equation}

By \eqref{eq: tvbd2}, if we prove $\PP_{\wt L_n} (F(m,K)^c) \leq e^{-K}$, then we are done. We start with a bound on $\PP_{\wt L_n} (F_1 (m,K))$. By the strong Markov property applied to the stopping time $\zeta_{i-1}+k$,
\begin{equation}\label{eq: f1bd}
\PP_{\wt L_n} \left( \zeta_{i} - (\zeta_{i-1} + k) > 72 K^3 \Bigm\vert \UU_{\zeta_{i-1} + k} \right) = \PP_{U_{\zeta_{i-1}+k}} \big( \zeta_1 > 72 K^3 \big) \leq e^{-2K}.
\end{equation}
The inequality is due to Lemma \ref{lem: coltoline}, which applies to $U_{\zeta_{i-1}+k}$ and $K$ because $U_{\zeta_{i-1} + k}$ is a non-isolated configuration with a diameter of at most $k + n$ and because $K \geq \max\{k+n,\theta_{4n} (c_2 n)\}$. 
From a union bound over the events which comprise $F_1 (m,K)$ and \eqref{eq: f1bd}, we find
\begin{equation}\label{eq: f1bd2}
\PP_{\wt L_n} (F_1 (m,K)) \leq m e^{-2K}.
\end{equation}

To bound $\PP_{\wt L_n} (F_2 (m))$, we apply the strong Markov property to $\zeta_j$ and use \eqref{eq: app qv}:
\begin{equation}\label{eq: app qv2}
\PP_{\wt L_n} \left( \TT_{\wt V} \leq \zeta_j + k \Bigm\vert \UU_{\zeta_j} \right)
 \geq \PP_{\wt L_n} \left( \TT_{\wt V} \leq k \right) \geq 1 - e^{-c_1 n^4 R^2}.
\end{equation}
Then, because $F_2 (j) \in \UU_{\zeta_j}$ and by \eqref{eq: app qv2},
\begin{equation}\label{eq: app qv3}
\PP_{\wt L_n} \left( \TT_{\wt V} \leq \zeta_j + k \Bigm\vert F_2 (j) \right) \geq 1 - e^{-c_1 n^4 R^2}.
\end{equation} We use \eqref{eq: app qv3} to calculate
\begin{equation}\label{eq: app qv4}
\PP_{\wt L_n} (F_2 (m)) = \prod_{j=0}^{m-1} \PP_{\wt L_n} \left( \TT_{\wt V} > \zeta_j + k \Bigm\vert F_2 (j) \right) \leq \prod_{j=0}^{m-1} (1 - e^{-c_1 n^4 R^2}) \leq e^{-2K}.
\end{equation}
The second inequality is due to the choice $m = 2K e^{c_1 n^4 R^2}$.

Recall that $F(m,K)$ is the union of $F_1 (m,K)$ and $F_2 (m)$. We have
\[ \PP_{\wt L_n} (F(m,K)) \leq \PP_{\wt L_n} (F_1 (m,K)) + \PP_{\wt L_n} (F_2 (m)) \leq m e^{-2K} + e^{-2K} \leq e^{-K}.\]
The first inequality is a union bound; the second is due to \eqref{eq: f1bd2} and \eqref{eq: app qv4}; the third holds because $m+1 \leq e^K$.
\end{proof}

\subsection{Proof of Proposition~\ref{prop: stat dist tail}}\label{sec7 tailbd}

We now prove a tightness estimate for the stationary distribution---that is, an upper bound on $\pi_n \big( \diam (\wt U) \geq d \big)$. By Proposition~\ref{prop: exist stat dist}, the stationary probability $\pi_n (\wt U)$ of any non-isolated, $n$-element coniguration $\wt U$ is the reciprocal of $\EE_{\wt U} \TT_{\wt U}$. When $d$ is large (relative to $\theta_{4n}$), this expected return time will be at least exponentially large in $\tfrac{d}{(\log d)^{1+o_n(1)}}$. This exponent arises from the consideration that, for a configuration with a diameter below $\theta_{4n}$ to increase its diameter to $d$, it must avoid collapse over the timescale for which it is typical (i.e., $(\log d)^{1+o_n(1)}$) approximately $\tfrac{d}{(\log d)^{1+o_n(1)}}$ times consecutively. Because the number of $n$-element configurations with a diameter of approximately $d$ is negligible relative to their expected return times, the collective weight under $\pi_n$ of such configurations will be exponentially small in $\tfrac{d}{(\log d)^{1+o_n(1)}}$.

We note that, while there are abstract results which relate hitting times to the stationary distribution (e.g., \cite[Lemma 4]{MR3651058}), we cannot directly apply results which require bounds on hitting times which hold uniformly for any initial configuration. This is because hitting times from $\wt V$ depend on its diameter. We could apply such results after partitioning $\wtnoniso (n)$ by diameter, but we would then save little effort from their use.


\begin{proof}[Proof of Proposition \ref{prop: stat dist tail}]
Let $d$ be at least $2\theta_{4n}$ and take $\delta = n^{-1}$. 
We claim that, for any configuration $\wt U$ with a diameter in $[2^j d, 2^{j+1} d)$ for an integer $j \geq 0$, the expected return time to $\wt U$ satsfies
\begin{equation}\label{eq: eutu bd}
\EE_{\wt U} \TT_{\wt U} \geq \exp \left( \frac{2^j d}{(\log (2^j d) )^{1+2\delta}} \right).
\end{equation}
We can use \eqref{eq: eutu bd} to prove \eqref{eq: diam tail bd} in the following way. We write $\{ \diam (\wt U) \geq d\}$ as a disjoint union of events of the form $H_j = \{2^j \leq \diam (\wt U) < 2^{j+1} d\}$ for $j \geq 0$. Because a disk with a diameter of at most $2^{j+1} d$ contains fewer than $\lfloor 4^{j+1} d^2 \rfloor$ elements of $\Z^2$, the number of non-isolated, $n$-element configurations with a diameter of at most $2^{j+1} d$ satisfies
\begin{equation}\label{eq: conf count bd}
\big| \big\{\text{$\wt U$ in $\wtnoniso (n)$ with $2^j d \leq \diam (\wt U) < 2^{j+1} d$} \big\} \big| \leq \binom{\lfloor 4^{j+1} d^2 \rfloor}{n} \leq (4^{j+1} d^2)^n.
\end{equation}

We use \eqref{eq: pos rec stat} with \eqref{eq: eutu bd} and \eqref{eq: conf count bd} to estimate
\begin{equation}\label{eq: pin1}
\pi_n \big( \diam (\wt U) \geq d \big)
= \sum_{j = 0}^\infty \pi_n (H_j)
= \sum_{j = 0}^\infty \sum_{\wt U \in H_j} \pi_n (\wt U)
\leq \sum_{j = 0}^\infty (4^{j+1} d^2)^n e^{- \frac{2^j d}{(\log (2^j d))^{1+2\delta}}}.
\end{equation} Using the fact that $d \geq 2\theta_{4n}$, it is easy to check that the ratio of the $(j+1)$\textsuperscript{st} summand to the $j$\textsuperscript{th} summand in \eqref{eq: pin1} is at most $e^{-j-1}$, for all $j \geq 0$. By \eqref{eq: pin1}, we have
\[ \pi_n \big( \diam (\wt U) \geq d \big) \leq e^{-\frac{d}{(\log d)^{1+3\delta}}} \sum_{j=0}^\infty e^{-j},\]
which proves \eqref{eq: diam tail bd} when the claimed bound \eqref{eq: eutu bd} holds.

We will prove \eqref{eq: eutu bd} by making a comparison with a geometric random variable on $\{0, 1, \dots\}$ with a ``success'' probability of $e^{-\frac{d}{(\log d)^{1+\delta}}}$ (or with $2^j d$ in place of $d$). This geometric random variable will model the number of visits to configurations with diameters below $\theta_{4n}$ before reaching a diameter of $d$, and the success probability arises from the fact that, for a configuration to increase its diameter to $d$ from $\theta_{4n}$, it must avoid collapse over $d-\theta_{4n}$ steps. By Corollary~\ref{cor: cons}, this happens with a probability which is exponentially small in $\tfrac{d}{(\log d)^{1+\delta}}$.

Let $\wt U$ be a non-isolated, $n$-element configuration with a diameter in $[2^j d,2^{j+1}d)$. Additionally, let $\wt V$ minimize $\EE_{\wt V} \TT_{\wt U}$ among $\mathcal{\wt V}$, the configurations in $\noniso (n)$ with a diameter of at most $\theta_{4n}$. Denoting by $N$ the number of visits to configurations in $\mathcal{\wt V}$ before $\TT_{\wt U}$, we claim
\begin{equation}\label{eq: uvclaim}
\EE_{\wt U} \TT_{\wt U} \geq (\log (2^{j+1} d))^{-2n} \, \EE_{\wt V} N.
\end{equation}

By \eqref{eq: lem631},
\begin{equation*}
\PP_{\wt U} (\TT_{\wt L_n} < \TT_{\wt U}) \geq \big(\log (2^{j+1}d) \big)^{-2n}.
\end{equation*}
By this bound and the strong Markov property (applied to $\TT_{\wt L_n}$), and due to our choice of $\wt V$,
\begin{equation}\label{eq: gcomp1}
\EE_{\wt U} \TT_{\wt U} \geq \big(\log (2^{j+1} d) \big)^{-2n} \, \EE_{\wt L_n} \TT_{\wt U} \geq \big(\log (2^{j+1} d) \big)^{-2n} \EE_{\wt V} \TT_{\wt U}.
\end{equation}
The time it takes to reach $\wt U$ from $\wt V$ is at least the number $N$ of visits $U_t$ makes to $\mathcal{\wt V}$ before $\TT_{\wt U}$, so \eqref{eq: gcomp1} implies \eqref{eq: uvclaim}.

The virtue of the lower bound \eqref{eq: uvclaim} is that we can bound below $\EE_{\wt V} N$ as
\begin{equation*}
\EE_{\wt V} N = \PP_{\wt V} (\TT_{\mathcal{\wt V}} < \TT_{\wt U}) \left(1 + \EE_{\wt V} \left[ N \bigm\vert \TT_{\mathcal{\wt V}} < \TT_{\wt U}\right] \right) \geq \PP_{\wt V} (\TT_{\mathcal{\wt V}} < \TT_{\wt U}) \left (1 + \EE_{\wt V} N \right).
\end{equation*}
This bound implies that $\EE_{\wt V} N$ is at least the expected value of a geometric random variable on $\{0,1,\dots\}$ with success parameter $p$ of $\PP_{\wt V} (\TT_{\wt U} < \TT_{\mathcal{\wt V}})$:
\begin{equation}\label{eq: gcomp3}
\EE_{\wt V} N \geq \frac{1-p}{p}.
\end{equation}
It remains to obtain an upper bound on $p$.

Because diameter increases at most linearly in time, $\TT_{\wt U}$ is at least $2^j d - \theta_{4n}$ under $\PP_{\wt V}$. Consequently,
\begin{equation}\label{eq: vatleastt}
\PP_{\wt V} (\TT_{\wt U} < \TT_{\mathcal{\wt V}}) \leq \PP_{\wt V} \big(\TT (\theta_{4n}) > 2^j d - \theta_{4n} \big).
\end{equation}
We apply Corollary~\ref{cor: cons} with $t$ equal to $\tfrac{2^j d-\theta_{4n}}{(\log (2^j d))^{1+\delta}}$, finding
\begin{equation*}
\PP_{\wt V} (\TT (\theta_{4n}) > 2^j d - \theta_{4n}) \leq \exp \left( - \frac{2^j d-\theta_{4n}}{(\log (2^j d) )^{1+\delta}} \right).
\end{equation*} By \eqref{eq: vatleastt}, this is also an upper bound on $p <\tfrac12$ and so, by \eqref{eq: gcomp3}, $\EE_{\wt V} N$ is at least $(2p)^{-1}$. Substituting these bounds into \eqref{eq: uvclaim} and simplifying with the fact that $d \geq 2 \theta_{4n}$, we find that the expected return time to $\wt U$ satisfies \eqref{eq: eutu bd}:
\begin{equation*}
\EE_{\wt U} \TT_{\wt U} \geq \tfrac12 \big(\log (2^{j+1} d) \big)^{-2n} \exp \left( \frac{2^j d-\theta_{4n}}{\big(\log (2^j d)\big)^{1+\delta}} \right) \geq \exp \left( \frac{2^j d}{\big(\log (2^j d) \big)^{1+2\delta}} \right).
\end{equation*}
\end{proof}



\section{Motion of the center of mass}\label{sec: cm}

As a consequence of the results of Section \ref{sec: stat dist} and standard renewal theory, the center of mass process $(\MM_t)_{t\geq 0}$, after linear interpolation and rescaling $(t^{-1/2}\MM_{st})_{s \in [0,1]}$, and when viewed as a measure on $\mathscr{C} ([0,1])$, converges weakly to two-dimensional Brownian motion as $t \to \infty$. This is the content of Theorem \ref{thm: cm}.

We will use the following lemma to bound the coordinate variances of the Brownian motion limit. To state it, we denote by $\tau_{i} = \inf \{t > \tau_{i-1} : \wt U_t = \wt L_n\}$ the $i$\textsuperscript{th} return time to $\wt L_n$.

\begin{lemma}\label{lem: m tails} Let $c$ be the constant from Proposition~\ref{prop: line to set} and abbreviate $\theta_{4n} (cn)$ by $\theta$. If, for some $i \geq 0$, $X$ is one of the random variables
\[\tau_{i+1} - \tau_i, \quad |\MM_{\tau_{i+1}} - \MM_{\tau_i}|, \quad \text{or} \quad | \MM_t - \MM_{\tau_i} | \1 (\tau_i \leq t \leq \tau_{i+1}),\]
then the distribution of $X$ satisfies the following tail bound
\begin{equation}\label{eq: m tails}
\PP_{\wt L_n} \big( X > K^5 \big) \leq e^{-K}, \quad K \geq \theta.
\end{equation}
Consequently,
\begin{equation}\label{eq: e var bd}
\EE_{\wt L_n} X \leq 2 \theta^6 \quad \text{and} \quad \var_{\wt L_n} X \leq 2 \theta^{12}.
\end{equation}
\end{lemma}

\begin{proof} Because the diameter of $\wt L_n$ is at most $n$, for any $K \geq \theta$, Proposition~\ref{prop: line to set} implies
\[ \PP_{\wt L_n} (\tau_1 > K^5) \leq e^{-K}.\]
Applying the strong Markov property to $\tau_{i}$, we find \eqref{eq: m tails} for $X = \tau_{i+1} - \tau_i$. Using \eqref{eq: m tails} with the tail sum formulas for the first and second moments gives \eqref{eq: e var bd} for this $X$. The other cases of $X$ then follow from
\begin{equation*}
\big| \MM_{\tau_{i+1}} - \MM_{\tau_i} \big| \leq \tau_{i+1} - \tau_i.
\end{equation*}
\end{proof}

\begin{proof}[Proof of Theorem~\ref{thm: cm}]
Standard arguments (e.g., Section 8 of \cite{billingsley1999}) combined with the renewal theorem show that $\big(t^{-1/2} \MM_{st}\big)_{t \geq 1}$ is a tight sequence of functions. We claim that the finite-dimensional distributions of the rescaled process converge as $t \to \infty$ to those of two-dimensional Brownian motion.

For any $m \geq 1$ and times $0 =s_0 \leq s_1 < s_2 < \cdots < s_m \leq 1$, form the random vector
\begin{equation}\label{eq: fdd vec} t^{-1/2} \left( \MM_{s_1 t}, \, \MM_{s_2 t} - \MM_{s_1 t}, \, \dots, \, \MM_{s_m t} - \MM_{s_{m-1} t} \right).
\end{equation}
For $s$ in $[0,1]$, we denote by $I (s)$ the number of returns to $\wt L_n$ by time $st$. Lemma \ref{lem: m tails} and Markov's inequality imply that $|\MM_{s_i t} - \MM_{\tau_{I (s_i)}}| \to 0$ in probability as $t \to \infty$, hence, by Slutsky's theorem, the distributions of \eqref{eq: fdd vec} and
\begin{equation}\label{eq: fdd vec2}
t^{-1/2} \left( \MM_{\tau_{I(s_1)}},\, \MM_{\tau_{I (s_2)}} - \MM_{\tau_{I (s_1)+1}},\, \dots,\, \MM_{\tau_{I (s_m)}} - \MM_{\tau_{I(s_{m-1}) + 1}} \right)
\end{equation}
have the same $t \to \infty$ limit. By the renewal theorem, $I(s_1) < I (s_2) < \cdots < I (s_m)$ for all sufficiently large $t$, so the strong Markov property implies the independence of the entries in \eqref{eq: fdd vec2} for all such $t$.

A generic entry in \eqref{eq: fdd vec2} is a sum of independent increments of the form $\MM_{\tau_{i+1}} - \MM_{\tau_i}$. As noted in Section~\ref{sec: intro}, the transition probabilities are unchanged when configurations are multiplied by elements of the symmetry group $\mathcal{G}$ of $\Z^2$. This implies \[ \EE_{\wt L_n} \left[ \MM_{\tau_{i+1}} - \MM_{\tau_i} \right] = o \quad\text{and}\quad \Sigma = \nu^2 {\bf I},\] where $\Sigma$ is the variance-covariance matrix of $\MM_{\tau_{i+1}} - \MM_{\tau_i}$ and $\nu$ is a constant which, by Lemma \ref{lem: m tails}, is finite. The renewal theorem implies that the scaled variance $t^{-1} \nu^2 (I(s_i) - I(s_{i-1}))$ of the $i$\textsuperscript{th} entry converges almost surely to $(s_i - s_{i-1}) \chi^2$ where $\chi^2 = \nu^2 / \EE_{\wt L_n} [\tau_1]$, hence, by Slutsky's theorem, we can replace the scaled variance of each entry in \eqref{eq: fdd vec2} with its almost-sure limit, without affecting the limiting distribution of the vector.

By the central limit theorem,
\[ \frac{1}{\chi \sqrt{t}} \left( \MM_{\tau_{I (s_i)}} - \MM_{\tau_{I (s_{i-1})+1}} \right) \stackrel{\text{d}}{\longrightarrow} \mathcal{N} \left( o, (s_i - s_{i-1}) {\bf I} \right),\] which, by the independence of the entries in \eqref{eq: fdd vec2} for all sufficiently large $t$, implies
\begin{multline}\label{eq: fdd vec3} 
\frac{1}{\chi \sqrt{t}} \left( \MM_{s_1 t}, \, \MM_{s_2 t} - \MM_{s_1 t}, \, \dots, \, \MM_{s_m t} - \MM_{s_{m-1} t} \right)\\ \stackrel{\text{d}}{\longrightarrow} \left( {\bf B} (s_1), {\bf B} (s_2-s_1), \dots, {\bf B} (s_m - s_{m-1}) \right),
\end{multline} as $t \to \infty$. Because $m$ and the $\{s_i\}_{i=1}^m$ were arbitrary, the continuous mapping theorem and \eqref{eq: fdd vec3} imply the convergence of the finite-dimensional distributions of $\left(\frac{1}{\chi \sqrt{t}} \MM_{st}, 0 \leq s \leq 1\right)$ to those of $\left({\bf B} (s), 0 \leq s \leq 1 \right)$. This proves the weak convergence component of Theorem \ref{thm: cm}.

It remains to bound $\chi^2$, which we do by estimating $\EE_{\wt L_n} [\tau_1]$ and $\nu^2$. $\EE_{\wt L_n} [\tau_1]$ is bounded above by $2 \theta^5$, due to Lemma~\ref{lem: m tails}, and below by $1$. Here, $\theta = \theta_{4n} (c_1n)$ and $c_1$ is the constant from Proposition~\ref{prop: line to set}.
To bound below $\nu^2$, denote the $e_2$ component of $\MM_{\tau_{i+1}} - \MM_{\tau_i}$ by $X$ and observe that $\PP_{\wt L_n} \left( X = n^{-1} \right)$ is at least the probability that, from $L_n$, the element at $o$ is activated and subsequently deposited at $(0,n)$ (recall that $L_n$ is the segment from $o$ to $(0,n-1)$), resulting in $\tau_1 = 1$ and $\MM_{\tau_1} = \MM_0 + n^{-1} \ee_2$. This probability is at least $e^{-c_2 n}$ for a constant $c_2$. 
Markov's inequality applied to $X^2$ then gives 
\begin{equation*}
\var_{\wt L_n} X \geq \PP_{\wt L_n} (X^2 \geq n^{-2}) \geq n^{-2} e^{-c_2n} \geq e^{-c_3 n}.
\end{equation*} By Lemma~\ref{lem: m tails}, $\nu^2$ is at most $2 \theta^{10}$. In summary,
\[ 1 \leq \EE_{\wt L_n} [\tau_1] \leq 2 \theta^5 \quad \text{and} \quad e^{-c_3 n} \leq \nu^2 \leq 2 \theta^{10},\]
which implies
\[ \theta_{5n} (cn)^{-1} \leq e^{-c_3 n} (2 \theta^5)^{-1} \leq \chi^2 \leq 2 \theta^{10} \leq \theta_{5n} (cn),\]
with $c = \max\{c_1,c_3\}$.
\end{proof}


\appendix
\renewcommand{\thesection}{\hspace{-1mm}}
\section{Proofs of auxiliary lemmas}
\renewcommand{\thesection}{A}

\setcounter{section}{0}
\setcounter{theorem}{0}
\renewcommand{\thetheorem}{\thesection.\arabic{theorem}}
\setcounter{equation}{0}
\renewcommand{\theequation}{\thesection.\arabic{equation}}

\subsection{Potential kernel bounds}

The following lemma collects several facts about the potential kernel which are used in Section \ref{sec: hm est}. As each fact is a simple consequence of \eqref{potker}, we omit its proof. 

\begin{lemma}\label{lem: potkerbds}
In what follows, $x,y,z,z'$ are elements of $\Z^2$.
\begin{enumerate}
\item For $\aa (y)$ to be at least $\aa (x)$, it suffices to have
\[ |y| \geq |x| (1 + \pi \lambda |x|^{-2} + (\pi \lambda)^2 |x|^{-4}).\]
In particular, if $|x| \geq 2$, then $|y| \geq 1.06 |x|$ suffices.
\item When $|x| \geq 1$, $\aa (x)$ is at least $\tfrac{2}{\pi} \log |x|$. When $|x| \geq 2$, $\aa (x)$ is at most $4 \log |x|$.
\item If $z,z' \in C(r)$ and $y \in D(R)^c$ for $r \leq \tfrac{1}{100} R$ and $R \geq 100$, then
\[ |\aa (y-z) - \aa (y-z') | \leq \tfrac{4}{\pi}.\]
\item If $x$ and $y$ satisfy $|x|, |y| \geq 1$ and $K^{-1} \leq \tfrac{|y|}{|x|} \leq K$ for some $K \geq 2$, then
\[ \aa(y) - \aa (x) \leq \log K.\]
\item Let $x, y \in \Z^2$ with $|x| \geq 8 |y|$ and $|y| \geq 10$. Then \[ |\aa (x + y) - \aa (x)| \leq 0.7 \frac{|y|}{|x|}. \]
\item Let $R \geq 10r$ and $r \geq 10$. Then, uniformly for $x \in C(R)$ and $y \in C(r)$, we have \[ 0.56  \log (R/r) \leq \aa (x) - \aa (y) \leq \log (R/r).\]
\end{enumerate}
\end{lemma}

In the next section, we will need the following comparison of $\aa$ and $\aa'$.

\begin{lemma}\label{lem: bdy est} Let $\mu$ be any probability measure on $C_x (r)$. Suppose $r \geq 2 (|x| + 1)$. Then \[ \left | \sum_{y \in C_x (r)} \mu (y) \mathfrak{a} (y) -\mathfrak{a}' (r) \right| \leq \left(\frac{5}{2 \pi} + 2 \lambda \right) \left( \frac{|x| + 1}{r} \right).\]
\end{lemma}

\begin{proof} We recall that, for any $x \in \Z^2$, the potential kernel has the form specified in \eqref{potker} where the error term conceals a constant of $\lambda$, which is no more than $0.07$ \cite{kozma2004asymptotic}. That is, \[ \left| \mathfrak{a} (x) - \frac{2}{\pi} \log |x| - \kappa \right| \leq \lambda |x|^{-2}.\] 

For $y \in C_x (r)$, we have $r - |x| - 1 \leq |y| \leq r + |x| + 1$. Accordingly, 
\begin{align*}
\mathfrak{a} (y) & \leq \frac{2}{\pi} \log \left| r + |x| + 1 \right| + \kappa + O ( |r - |x| - 1|^{-2} )\\
& = \frac{2}{\pi} \log r + \kappa + \frac{2}{\pi} \log \left( 1 + \frac{|x| + 1}{r} \right) + O( |r - |x| - 1|^{-2}).
\intertext{Using the assumption $(|x|+1)/r \in (0,1/2)$ with Taylor's remainder theorem gives}
\mathfrak{a} (y) & \leq \mathfrak{a}' (r) + \frac{2}{\pi} \left( \frac{|x| + 1}{r} + \frac12 \left( \frac{|x| + 1}{r} \right)^2\right) + O( |r - |x| - 1|^{-2}).
\intertext{Simplifying with $r \geq 2(|x|+1)$ and $r \geq 2$ leads to}
\mathfrak{a} (y) & \leq \mathfrak{a}' (r) + \frac{2}{\pi} \left(\frac54 + \pi \lambda \right) \left(\frac{|x| + 1}{r} \right) = \mathfrak{a}' (r) + \left( \frac{5}{2\pi} + 2 \lambda \right) \left( \frac{|x| + 1}{r} \right).  
\end{align*}
The lower bound is similar. 
Because this holds for any $y \in C_x (r)$, for any probability measure $\mu$ on $C_x (r)$, we have \[ \left| \sum_{y \in C_x (r)} \mu (y) \mathfrak{a} (y) - \mathfrak{a}' (r) \right| \leq \left( \frac{5}{2\pi} + 2 \lambda \right) \left( \frac{|x| + 1}{r} \right).\]
\end{proof}

\subsection{Comparison between harmonic measure and hitting probabilities}

To prove Lemma~\ref{lem: near unif}, we require a comparison (Lemma~\ref{lem: hit har}) between certain values of harmonic measure and hitting probabilities. In fact, we need additional quantification of an error term which appears in standard versions of this result (e.g. \cite[Theorem 2.1.3]{lawler2013intersections}). Effectively, this additional quantification comes from a bound on $\lambda$, the implicit constant in \eqref{potker}. The proof is similar to that of Theorem 3.17 in \cite{popov2021two}.

\begin{lemma}\label{lem: hit har} Let $x \in D(R)^c$ for $R \geq 100r$ and $r \geq 10$. Then
\begin{equation}\label{eq: hit har}
0.93 \H_{C(r)} (y) \leq \H_{C(r)} (x,y) \leq 1.04 \H_{C(r)} (y).
\end{equation}
\end{lemma}

\begin{proof}
We have \begin{equation}\label{eq: hh} \H_{C(r)} (x,y) - \H_{C(r)} (y) = - \aa (y-x) + \sum_{z \in C(r)} \P_y \left( S_{\tau_{C(r)}} = z \right) \aa (z - x).\end{equation} Since $C(10r)$ separates $x$ from $C(r)$, the optional stopping theorem applied to $\sigma_{C(10r)} \wedge \tau_{C(r)}$ and the martingale $\aa \left(S_{t \wedge \tau_x} - x\right)$ gives
\begin{multline}\label{eq: a mult} \aa (y-x) = \sum_{z \in C(r)} \P_y \left( S_{\tau_{C(r)}} = z \right) \aa (z - x)\\ + \E_y \left[ \aa \left(S_{\sigma_{C(10r)}} - x \right) - \aa \left(S_{\tau_{C(r)}} - x \right) \Bigm\vert \sigma_{C(10r)} < \tau_{C(r)}\right] \P_y \left( \sigma_{C(10r)} < \tau_{C(r)} \right). \end{multline}

In the second term of \eqref{eq: a mult}, we analyze the difference in potentials by observing \[S_{\sigma_{C(10r)}} - x - \left( S_{\tau_{C(r)}} - x \right) = S_{\sigma_{C(10r)}} - S_{\tau_{C(r)}}.\] Accordingly, letting $u = S_{\tau_{C(r)}} - x$ and $v = S_{\sigma_{C(10r)}} - S_{\tau_{C(r)}}$, \[ \aa \left(S_{\sigma_{C(10r)}} - x \right) - \aa \left(S_{\tau_{C(r)}} - x \right) = \aa (u+v) - \aa (u).\] We observe that $|v| \leq 11r+2$ and $|u| \geq 99r-2$, so $|u| \geq 8 |v|$. Since we also have $|v| \geq 9r-2 \geq 10$, (5) of Lemma~\ref{lem: potkerbds} applies to give \[ \aa(u+v) - \aa (u) \leq 0.7 \frac{|v|}{|u|} \leq \frac{2}{25}.\]

We analyze the other factor of \eqref{eq: a mult} as 
\begin{align*} 
\P_y \left( \sigma_{C(10r)} < \tau_{C(r)} \right) 
& = \frac14 \sum_{z \notin C(r): z \sim y} \P_z \left( \sigma_{C(10r)} < \tau_{C(r)} \right)\\ 
&= \frac14 \sum_{z \notin C(r): z \sim y} \frac{\aa (z - z_0) - \E_z \aa \left( S_{\tau_{C(r)}} - z_0 \right)}{\E_z \left[ \aa \Big( S_{\sigma_{C(10r)}} \Big) - \aa \Big( S_{\tau_{C(r)}} \Big) \Bigm\vert \sigma_{C(10r)} < \sigma_{C(r)} \right]},
\end{align*}
where $z_0 \in A$. To obtain an upper bound on the potential difference in the denominator, we apply (6) of Lemma~\ref{lem: potkerbds}, which gives
\[ \P_y \left( \sigma_{C(10r)} < \tau_{C(r)} \right)  \leq \frac{1}{0.6 \log 10} \H_{C(r)} (y).\] Combining this with the other estimate for the second term of \eqref{eq: a mult}, we find \[ \aa (y-x) \leq \sum_{z \in C(r)} \P_y \left( S_{\tau_{C(r)}} = z \right) \aa (z-x) + \underbrace{\frac{2}{25} \cdot \frac{1}{0.56 \log 10}}_{\leq 0.063} \H_{C(r)} (y).\] Substituting this into \eqref{eq: hh}, we have \[ \H_{C(r)} (x,y) - \H_{C(r)} (y) \geq - 0.063 \H_{C(r)} (y) \implies \H_{C(r)} (x,y) \geq 0.93 \H_{C(r)} (y).\]

We again apply (5) and (6) of Lemma~\ref{lem: potkerbds} to bound the factors in the second term of \ref{eq: a mult} as \[ \aa (u + v) - \aa (u) \geq -0.0875 \quad \text{and} \quad \P_y \left( \sigma_{C(10r)} < \tau_{C(r)} \right) \geq \frac{1}{ \log 10} \H_{C(r)} (y).\] Substituting these into \eqref{eq: a mult}, we find \[ \aa (y-x) \geq \sum_{z \in C(r)} \P_y \left( S_{\tau_{C(r)}} = z \right) \aa (z-x) - 0.0875 \cdot \frac{1}{\log 10} \H_{C(r)} (y).\] Consequently, \eqref{eq: hh} becomes \[ \H_{C(r)} (x,y) - \H_{C(r)} (y) \leq \frac{0.0875}{ \log 10} \H_{C(r)} (y) \leq \frac{1}{25} \H_{C(r)} (y).\] Rearranging, we find \[ \H_{C(r)} (x,y) \leq 1.04 \H_{C(r)} (y).\]
\end{proof}

\subsection{Uniform lower bound on a conditional entrance measure}

We now use Lemma~\ref{lem: hit har} to prove an inequality which is needed for the proof of Lemma \ref{lem: near unif}. The proof of Lemma~\ref{lem: unif} is similar to that of Lemma 2.1 in \cite{dembo2006late}.

\begin{lemma}\label{lem: unif} Let $\eps > 0$, denote $\eta = \tau_{C(R)} \wedge \tau_{C (\eps R)}$, and denote by $\mu$ the uniform measure on $C(\eps R)$. There is a constant $c$ such that, if $\eps \leq \tfrac{1}{100}$ and $R \geq 10\eps^{-2}$, and if 
\begin{equation}\label{eq: min p} \min_{x \in C(\eps R)} \P_x \left( \tau_{C(\eps^2 R)} < \tau_{C(R)} \right) > \frac{1}{10},\end{equation}
then, uniformly for $x \in C( \eps R)$ and $y \in C (\eps^2 R)$, 
\[ \P_x \left( S_\eta = y, \tau_{C(\eps^2 R)} < \tau_{C(R)}\right) \geq c \mu (y) \, \P_x \left(\tau_{C(\eps^2 R)} < \tau_{C (R)} \right).\]
\end{lemma}

\begin{proof} Fix $\eps$ and $R$ which satisfy the hypotheses. Let $x \in C(\eps R)$ and $y \in C(\eps^2 R)$. We have \begin{equation}\label{eq: h minus p} \P_x \left( S_{\tau_{C(\eps^2 R)}} = y, \tau_{C(\eps^2 R)} < \tau_{C(R)}\right) = \H_{C(\eps^2 R)} (x, y) - \P_x \left( S_{\tau_{C(\eps^2 R)}} = y, \tau_{C(\eps^2 R)} > \tau_{C(R)}\right).\end{equation} By the strong Markov property applied to $\tau_{C(R)}$, \begin{equation}\label{eq: e h} \P_x \left( S_{\tau_{C(\eps^2 R)}} = y, \tau_{C(\eps^2 R)} > \tau_{C(R)}\right) = \E_x \left[ \H_{C(\eps^2 R)} \big( S_{\tau_{C(R)}}, y \big) ; \tau_{C(\eps^2 R)} > \tau_{C (R)} \right].\end{equation} We will now use Lemma~\ref{lem: hit har} to uniformly bound the terms of the form $\H_{C(\eps^2 R)} (\cdot, y)$ appearing in \eqref{eq: h minus p} and \eqref{eq: e h}.

For any $w \in C(R)$, the hypotheses of Lemma~\ref{lem: hit har} are satisfied with $\eps^2 R$ in the place of $r$ and $R$ as presently defined, because then $r \geq 10$ and $R \geq 100 \eps^2 R$. Therefore, by \eqref{eq: hit har}, uniformly for $w \in C(R)$, \begin{equation}\label{eq: 98 bd}\H_{C(\eps^2 R)} (w,y) \leq 1.04 \, \H_{C(\eps^2 R)} (y).\end{equation}

Now, for any $x \in C(\eps R)$, the hypotheses of Lemma~\ref{lem: hit har} are again satisfied with the same $r$ and with $\eps R$ in the place of $R$, as $\eps R \geq 100 \eps^2 R$ by assumption. We apply \eqref{eq: hit har} to find \begin{equation}\label{eq: half bd} \H_{C(\eps^2 R)} (x,y) \geq 0.93 \H_{C(\eps^2 R)} (y).\end{equation}

Substituting \eqref{eq: 98 bd} into \eqref{eq: e h}, we find \[ \P_x \left( S_{\tau_{C(\eps^2 R)}} = y, \tau_{C(\eps^2 R)} > \tau_{C(R)}\right) \leq 1.04\, \H_{C(\eps^2 R)} (y) \P_x \left( \tau_{C(\eps^2 R)} > \tau_{C(R)} \right).\] Similarly, substituting \eqref{eq: half bd} into \eqref{eq: h minus p} and using the previous display, we find \begin{multline*} \P_x \left( S_{\tau_{C(\eps^2 R)}} = y, \tau_{C(\eps^2 R)} < \tau_{C(R)}\right) \geq 0.93 \, \H_{C(\eps^2 R)} (y) \P_x \left( \tau_{C(\eps^2 R)} < \tau_{C(R)} \right)\\ - \left( 1.04 - 0.93 \right) \, \H_{C(\eps^2 R)} (y) \P_x \left( \tau_{C(\eps^2 R)} > \tau_{C(R)} \right).\end{multline*} Applying hypothesis \eqref{eq: min p}, we find that the right-hand side is at least \[ c_1 \, \H_{C(\eps^2 R)} (y) \P_x \left( \tau_{C(\eps^2R)} < \tau_{C(R)} \right),\] for a positive constant $c_1$. The result then follows the existence of a positive constant $c_2$ such that $\H_{C(\eps^2 R)} (y) \geq c_2 \mu (y)$ for any $y \in C(\eps^2 R)$.
\end{proof}

\subsection{Estimate for the exit distribution of a rectangle}

Informally, Lemma~\ref{lem: aspect ratio} says that the probability a walk from one end of a rectangle (which may not be aligned with the coordinate axes) exits through the opposite end is bounded below by a quantity depending upon the aspect ratio of the rectangle. We believe this estimate is known but, as we are unable to find a reference for it, we provide one here. In brief, the proof uses an adaptive algorithm for constructing a sequence of squares which remain inside the rectangle and the sides of which are aligned with the axes. We then bound below the probability that the walk follows the path determined by the squares until exiting the opposite end of the rectangle.

Recall that $\Rec (\phi, w, \ell)$ denotes the rectangle of width $w$, centered along the line segment from $-e^{{\bf i} \phi} w$ to $e^{{\bf i} \phi} \ell$, intersected with $\Z^2$ (see Figure~\ref{fig: rect_fig}).

\begin{figure}[htbp]
\centering {\includegraphics[width=0.35\linewidth]{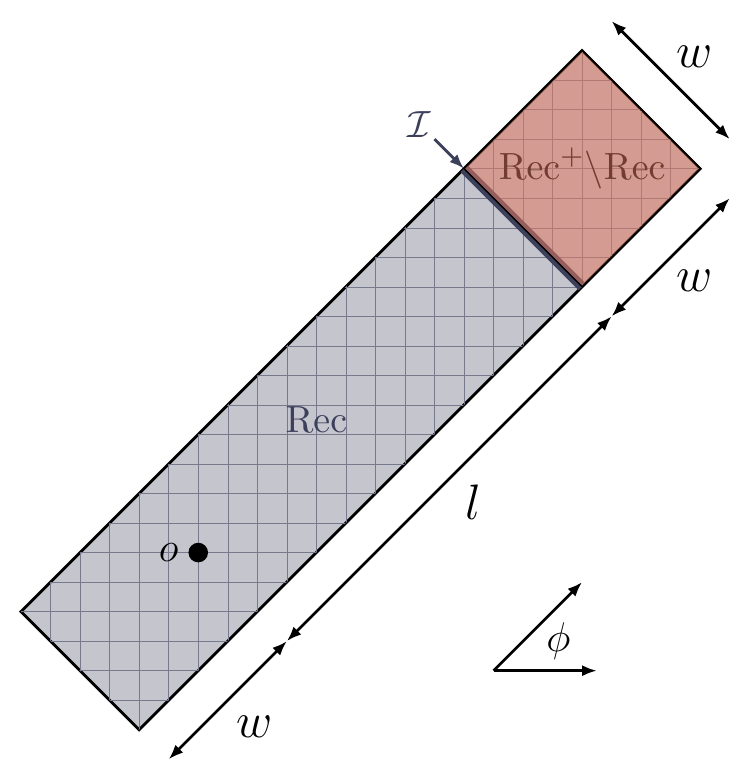}} 
\caption{On the left, we depict the rectangles $\Rec = \Rec (\phi, w, l)$ (shaded blue) and $\Rec^+ = \Rec (\phi,w,l + w)$ (union of blue- and red-shaded regions) for $\phi = \pi/4$, $w = 4 \sqrt{2}$, and $\ell = 11\sqrt{2}$. $\mathcal{I}$ denotes $\Rec \cap \bdy (\Rec^+ {\setminus} \Rec)$.}
\label{fig: rect_fig}
\end{figure}

\begin{lemma}\label{lem: aspect ratio} For any $24 \leq w \leq \ell$ and any $\phi$, let $\Rec = \Rec (\phi, w, \ell)$ and $\Rec^+ = \Rec (\phi, w, \ell+w)$. Then, 
\begin{equation*}
\P_o \left( \tau_{\bdy \Rec} < \tau_{\bdy \Rec^+} \right) \geq c^{\ell / w},
\end{equation*} for a universal positive constant $c < 1$.
\end{lemma}

We use the hypothesis $w \geq 24$ to deal with the effects of discreteness; the constant $24$ is otherwise unimportant, and many choices would work in its place.

\begin{proof}[Proof of Lemma \ref{lem: aspect ratio}]
We will first define a square, centered at the origin and with each corner in $\Z^2$, which lies in $\Rec^+$. We will then translate it to form a sequence of squares through which we will guide the walk to $\Rec^+ {\setminus} \Rec$ without leaving $\Rec^+$ (see Figure~\ref{fig: rect_fig}). We split the proof into three steps: (1) constructing the squares; (2) proving that they lie in $\Rec^+$; and (3) establishing a lower bound on the probability that the walk hits $\bdy \Rec$ before hitting the interior boundary of $\Rec^+$.

\textbf{Step 1: Construction of the squares.} Without loss of generality, assume $0 \leq \phi < \pi/2$. For $x \in \Z^2$, we will denote its first coordinate by $x^1$ and its second coordinate by $x^2$. We will use this convention only for this proof. Let $\mathfrak{l}$ be equal to $\lfloor \tfrac{w}{8} \rfloor$ if it is even and equal to $\lfloor \tfrac{w}{8} \rfloor - 1$ otherwise. With this choice, we define \[ Q = \{x \in \Z^2: \max\{x^1, x^2\} \leq \mathfrak{l} \}.\] Since $\mathfrak{l}$ is even, the translates of $Q$ by integer multiples of $\frac12 \mathfrak{l}$ are also subsets of $\Z^2$.

We construct a sequence of squares $Q_i$ in the following way, where we make reference to the line $L_\phi^\infty = e^{{\bf i} \phi} \R$. Let $y_1 = o$ and $Q_1 = y_1 + Q$. For $i \geq 1$, let \[ y_{i+1} = \begin{cases} y_i + \tfrac12 \mathfrak{l} \, (0,1) & \text{if $y_i$ lies on or below $L_\phi^\infty$}\\ y_i + \tfrac12 \mathfrak{l} \, (1,0) & \text{if $y_i$ lies above $L_\phi^\infty$}\end{cases} \quad \text{and} \quad Q_{i+1} = y_{i+1} + Q .\] In words, if the center of the present square lies on or below the line $L_\phi^\infty$, then we translate the center north by $\tfrac12 \mathfrak{l}$ to obtain the next square. Otherwise, we translate the center to the east by $\tfrac12 \mathfrak{l}$.

We further define, for $i \geq 1$, \begin{equation}\label{eq: mi 2} M_i = \begin{cases} \left\{x \in Q_i : x^2 -y_i^2 = \tfrac12 \mathfrak{l} \,\,\,\text{and}\,\,\, | y_i^1 - x^1 | \leq \tfrac12 \mathfrak{l} - 1\right\} & \text{if $y_i$ lies on or below $L_\phi^\infty$} \\ \left\{x \in Q_i : x^1 - y_i^1 = \tfrac12 \mathfrak{l} \,\,\,\text{and}\,\,\, | y_i^2 - x^2 | \leq \tfrac12 \mathfrak{l} - 1\right\} & \text{if $y_i$ lies above $L_\phi^\infty$}.\end{cases}\end{equation} In words, if $y_i$ lies on or below the line $L_\phi^\infty$, we choose $M_i$ to be the northernmost edge of $Q_i$, excluding the corners. Otherwise, we choose it to be the easternmost edge, excluding the corners. These possibilities are depicted in Figure~\ref{fig: tunnel_fig}. In fact, we leave the corners out of the $M_i$, as indicated, by the bounds of $\tfrac12 \mathfrak{l} - 1$ instead of $\tfrac12 \mathfrak{l}$ in \eqref{eq: mi 2}. We must do so to ensure that $\P_\omega \left( \tau_{M_{i+1}} \leq \tau_{\bdyo Q_{i+1}} \right)$ is harmonic for all $\omega \in M_i$; we will shortly need this to apply the Harnack inequality. Upcoming Figure~\ref{fig: sq cases} provides an illustration of $M_i$ in this context.

\begin{figure}
\centering {\includegraphics[width=0.8\linewidth]{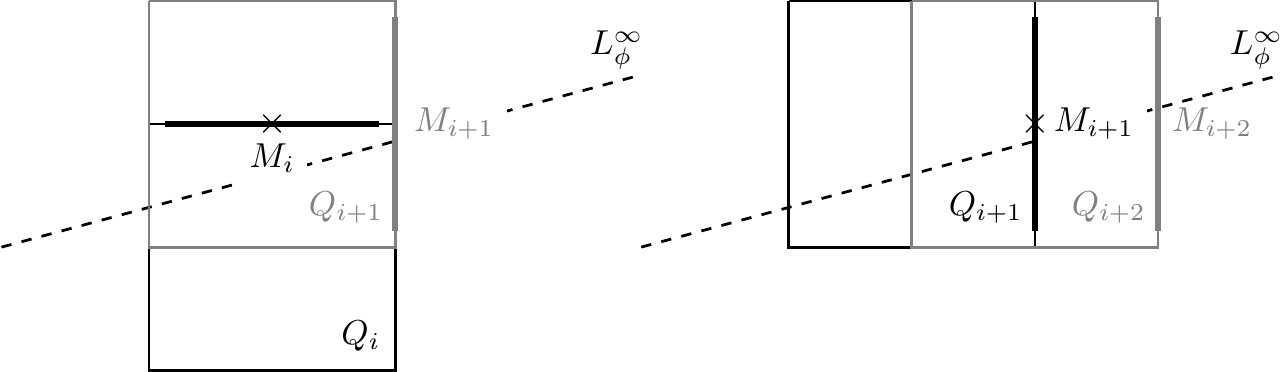}}
\caption{Two steps in the construction of squares. Respectively on the left and right, $y_{i+1} \in M_i$ and $y_{i+2} \in M_{i+1}$ (indicated by the $\times$ symbols) lie above $L_\phi^\infty$, so $M_{i+1}$ and $M_{i+2}$ are situated on the eastern sides of $Q_{i+1}$ and $Q_{i+2}$. However, on the left, as $Q_i$ was translated north to form $Q_{i+1}$, the relative orientation of $M_i$ and $M_{i+1}$ is perpendicular. In contrast, as $Q_{i+1}$ is translated east to form $Q_{i+2}$, the right-hand side has parallel $M_{i+1}$ and $M_{i+2}$.}
\label{fig: tunnel_fig}
\end{figure}

We will guide the walk to $\bdy \Rec$ without leaving $\Rec^+$ by requiring that it exit each square $Q_i$ through $M_i$ for $1 \leq i \leq J$, where we define \[ J = \min\{i \geq 1: M_i \subseteq \Rec^c \}.\] That is, $J$ is the first index for which $M_i$ is fully outside $\Rec$. 
It is clear that $J$ is finite.

\textbf{Step 2: Proof that $\cup_{i=1}^J Q_i$ is a subset of $\Rec^+$.} Let $v$ be the northeastern endpoint of $L_\phi$, where $L_\phi$ is the segment of $L_\phi^\infty$ from $o$ to $e^{{\bf i} \phi} (\ell + w/2)$ and define $k$ to be the first index for which $y_k$ satisfies \[ y_k^1 > v^1 \quad\text{or}\quad y_k^2 > v^2.\] It will also be convenient to denote by $\mathcal{I}$ the interface between $\Rec$ and $\Rec^+ {\setminus} \Rec$ (the dashed line in Figure~\ref{fig: rect_fig}), given by \[ \mathcal{I} = \Rec \cap \bdy \left( \Rec^+ {\setminus} \Rec \right).\]

By construction, we have $|y_k - y_{k-1}| = \tfrac12 \mathfrak{l}$ and $|y_k - v| \leq \tfrac12 \mathfrak{l}$. By the triangle inequality, $|y_{k-1} - v| \leq \mathfrak{l}$. As $|v| = \ell + w/2$ and because $\dist(o, \mathcal{I}) \leq \ell + 1$, we must have---again by the triangle inequality---that $\dist( v, \mathcal{I}) \geq w/2 - 1$. From a third use of the triangle inequality and the hypothesized lower bound on $w$, we conclude \begin{equation}\label{eq: dist w 2} \dist (y_{k-1}, \mathcal{I}) \geq \frac{w}{2} - 1 - \mathfrak{l} \geq \frac{w}{2} -1 - \frac{w}{8} \geq \frac{w}{3} > 2 \mathfrak{l}.\end{equation} To summarize in words, $y_{k-1}$ is not in $\Rec$ and it is separated from $\Rec$ by a distance strictly greater than $2\mathfrak{l}$. 

Because the sides of $Q_{k-1}$ have length $\mathfrak{l}$, \eqref{eq: dist w 2} implies $Q_{k-1} \subseteq \Rec^c$. Since $M_{k-1}$ is a subset of $Q_{k-1}$, we must also have $M_{k-1} \subseteq \Rec^c$, which implies $J \leq k-1$. As $k$ was the first index for which $y_k^1 > v^1$ or $y_k^2 > v^2$, $y_J$ satisfies $y_J^1 \leq v^1$ and $y_J^2 \leq v^2$. Then, by construction, for all $1 \leq i \leq J$, the centers satisfy \begin{equation}\label{eq: cent 2} y^1 \leq y_i^1 \leq v^1 \quad\text{and}\quad y^2 \leq y_i^2 \leq v^2.\end{equation} From \eqref{eq: cent 2} and the fact that $\dist (y_i , L_\phi^\infty) \leq \tfrac12 \mathfrak{l}$, we have \[ \dist(y_i, L_\phi) = \dist (y_i, L_\phi^\infty) \leq \frac12 \mathfrak{l} \quad \forall\,\, 1 \leq i \leq J.\]

As the diagonals of the $Q_i$ have length $\sqrt{2} \mathfrak{l}$, \eqref{eq: cent 2} and the triangle inequality imply \[\dist(x, L_\phi) \leq \dist(y_i, L_\phi) + \frac12 \sqrt{2} \mathfrak{l} = \frac12 (1 + \sqrt{2}) \mathfrak{l} < \frac{w}{4} \quad \quad \forall \,\,x \in \bigcup_{i=1}^J Q_i.\] To summarize, any element of $Q_i$ for some $1 \leq i \leq J$ is within a distance $w/4$ of $L_\phi$. As $\Rec^+$ contains all points $x$ within a distance $\tfrac{w}{2}$ of $L_\phi$, we conclude \[ \bigcup_{i=1}^J Q_i \subseteq \Rec^+.\]

\textbf{Step 3: Lower bound for $\P_{o} \left( \tau_{\bdy \Rec} < \tau_{\bdy \Rec^+}\right)$.} From the previous step, to obtain a lower bound on the probability that the walk exits $\Rec$ before $\Rec^+$, it suffices to obtain an upper bound $J^\ast$ on $J$ and a lower bound $c < 1$ on \[\P_{\omega} \left( \tau_{M_{i+1}} \leq \tau_{\bdyo Q_{i+1}} \right),\] uniformly for $\omega \in M_i$, for $0 \leq i \leq J - 1$. This way, if we denote $Y_0 \equiv y$ and $Y_i = S_{\tau_{\bdyo Q_i}}$ for $1 \leq i \leq J-1$, we can apply the strong Markov property to each $\tau_{M_i}$ and use the lower bound for each factor to obtain the lower bound \begin{equation}\label{eq: prod low bd} \P_{o} \left( \tau_{\bdy \Rec} < \tau_{\bdy \Rec^+}\right) \geq c^{J^\ast}.\end{equation}

To obtain an upper bound on $J$, we first recall that $L_\phi$ has a length of $\ell + w/2$, which satisfies \begin{equation}\label{eq: len bd} \ell + w/2 = \frac{\mathfrak{l}}{2} \left( \frac{2 \ell}{\mathfrak{l}} + \frac{w}{\mathfrak{l}} \right) \leq \frac{\mathfrak{l}}{2} \left( \frac{2 \ell}{w/8 - 1} + \frac{w}{w /8- 1} \right) \leq \frac{\mathfrak{l}}{2} \left(48\frac{ \ell}{w} + 24\right),\end{equation} due to the fact that $\mathfrak{l} \geq \lfloor w/8 \rfloor - 1 \geq w/8 - 2$ and the hypothesis of $w \geq 24$. The number of steps to reach $J$ is no more than twice the ratio $(\ell + w/2)/(\mathfrak{l}/2)$. Accordingly, using the bound in \eqref{eq: len bd} and the hypothesis that $\ell /w \geq 1$, we have \begin{equation}\label{eq: iota bd} J \leq 2 \left( 48 \frac{\ell}{w} + 24 \right) \leq 144 \frac{\ell}{w} =: J^\ast.\end{equation} We now turn to the hitting probability lower bounds.

\begin{figure}
\centering {\includegraphics[width=0.6\linewidth]{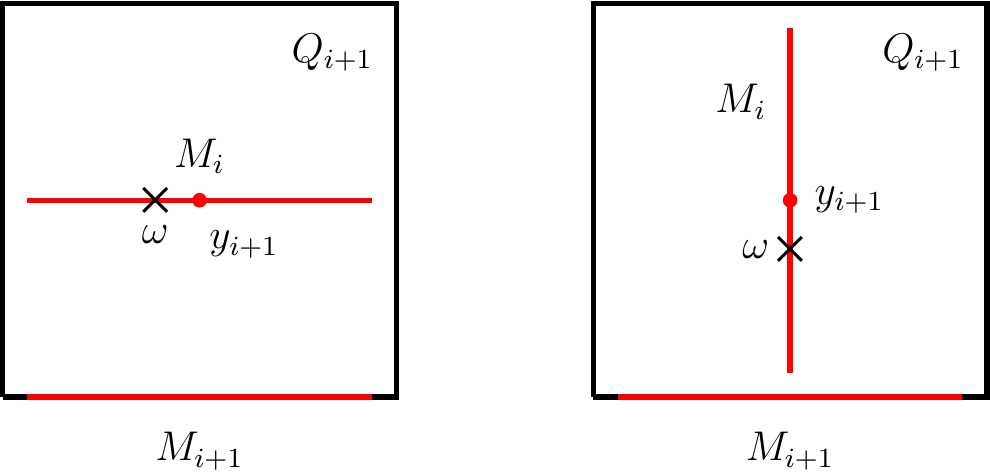}}
\caption{The two cases for lower-bounding $M_{i+1}$ hitting probabilities.}
\label{fig: sq cases}
\end{figure}

From the construction, there are only two possible orientations of $M_i$ relative to $M_{i+1}$ (Figure~\ref{fig: sq cases}). Either $M_i$ and $M_{i+1}$ have parallel orientation or they do not. Consider the former case. The hitting probability $\P_\omega \left( \tau_{M_{i+1}} \leq \tau_{\bdyo Q_{i+1}} \right)$ is a harmonic function of $\omega$ for all $\omega$ in $Q_{i+1} {\setminus} \bdyo Q_{i+1}$ and $M_{i+1}$ in particular. Therefore, by the Harnack inequality \cite[Theorem 1.7.6]{lawler2013intersections}, there is a constant $a_1$ such that \begin{equation}\label{eq: har bd} \P_\omega \left( \tau_{M_{i+1}} \leq \tau_{\bdyo Q_{i+1}} \right) \geq a_1 \P_{y_{i+1}} \left( \tau_{M_{i+1}} \leq \tau_{\bdyo Q_{i+1}} \right)\quad\forall\,\,\omega \in M_{i+1}.\end{equation} The same argument applies to the case when $M_i$ and $M_{i+1}$ do not have parallel orientation and we find there is a constant $a_2$ such that \eqref{eq: har bd} holds with $a_2$ in place of $a_1$. Setting $a = \min\{a_1, a_2\}$, we conclude that, for all $0 \leq i \leq J - 1$ and any $\omega \in M_i$, \begin{equation}\label{eq: har bd 2} \P_\omega \left( \tau_{M_{i+1}} \leq \tau_{\bdyo Q_{i+1}} \right) \geq a \P_{y_{i+1}} \left( \tau_{M_{i+1}} \leq \tau_{\bdyo Q_{i+1}} \right).\end{equation} We have reduced the lower bound for any $\omega \in M_i$ and either of the two relative orientations of $M_i$ and $M_{i+1}$ to a lower bound on the hitting probability of one side of $Q_{i+1}$ from the center. By symmetry, the walk hits $M_{i+1}$ first with a probability of exactly $1/4$. We emphasize that the probability on the left-hand side of \eqref{eq: har bd 2} is exactly $1/4$ as although $M_{i+1}$ does not include the adjacent corners of $Q_{i+1}$, which are elements of $\bdyo Q_{i+1}$, the corners are separated from $y_{i+1}$ by the other elements of $\bdyo Q_{i+1}$. 

Calling $b = a/4$ and combining \eqref{eq: iota bd} and \eqref{eq: har bd 2} with \eqref{eq: prod low  bd}, we have \[ \P_{o} \left( \tau_{\bdy \Rec} < \tau_{\bdy \Rec^+}\right) \geq b^{J^\ast} = b^{144 \ell /w} = c^{\ell /w}\] for a positive constant $c < 1$.
\end{proof}


\bibliographystyle{plain}
\bibliography{hat}

\end{document}